\documentclass[11pt]{article}  
\usepackage[english]{babel}
\usepackage[toc,page]{appendix}
\usepackage{bbm}
\usepackage{fullpage}
\usepackage[top=1in, bottom=1.25in, left=1in, right=1in]{geometry}
\usepackage{etaremune}
\usepackage{enumitem}

\usepackage{amssymb,amsmath,amsthm}

\usepackage{hyperref}

\usepackage{mathtools}

\mathtoolsset{showonlyrefs}

\DeclareSymbolFont{rsfs}{U}{rsfs}{m}{n}
\DeclareSymbolFontAlphabet{\mathscrsfs}{rsfs}
\usepackage{mathrsfs}

\usepackage{graphicx}

\usepackage{url}

\hypersetup{
  colorlinks   = true, 
  urlcolor     = blue, 
  linkcolor    = blue, 
  citecolor   = red 
}

\usepackage[labelformat=empty]{caption}

\newtheorem{theorem}{Theorem}[section]
\newtheorem{lemma}[theorem]{Lemma}

\newtheorem{proposition}[theorem]{Proposition}
\newtheorem{corollary}[theorem]{Corollary}

\theoremstyle{definition}
\newtheorem{definition}{Definition}
\newtheorem{remark}[theorem]{Remark}

\numberwithin{equation}{section}

\usepackage{times}

\newcommand{\bea}{\begin{eqnarray}}
\newcommand{\eea}{\end{eqnarray}}
\newcommand{\<}{\langle}
\renewcommand{\>}{\rangle}

\newcommand{\wt}{\widetilde}

\newcommand{\wh}{\widehat}

\def\eps{{\varepsilon}}

\def\supp{{\rm supp}}

\def\bx{{\boldsymbol{x}}}

\def\cC{{\mathcal C}}

\def\va{{\vec a}}

\def\bx{{\boldsymbol{x}}}

\def\va{\vec{a}}

\def\de{{\rm d}}

\def\<{\langle}
\def\>{\rangle}

\def\diag{{\rm diag}}
\def\diam{{\rm diam}}

\def\cH{{\cal H}}

\def\cN{{\cal N}}

\def\b0{{\boldsymbol{0}}}

\renewcommand{\b}{\mathbf{b}}

\def\lt{\left}
\def\rt{\right}

\def\la{\langle}
\def\ra{\rangle}

\def\eps{\varepsilon}

\def\bbE{{\mathbb{E}}}

\def\bbN{{\mathbb{N}}}
\def\bbP{{\mathbb{P}}}
\def\bbR{{\mathbb{R}}}
\def\bbS{{\mathbb{S}}}

\def\bbW{{\mathbb{W}}}
\def\bbZ{{\mathbb{Z}}}

\def\cN{{\mathcal{N}}}
\def\cP{{\mathcal{P}}}

\DeclareMathOperator*{\argmax}{\arg\max}
\DeclareMathOperator*{\argmin}{\arg\min}

\def\sph{\mathrm{sp}}

\def\sym{{\rm sym}}

\def\Lip{\mathrm{Lip}}

\author{Yury Polyanskiy\thanks{Department of Electrical Engineering and Computer Science, MIT. Email: \texttt{yp@mit.edu}}
\and Mark Sellke\thanks{Department of Statistics, Harvard University.
    Email: \texttt{msellke@fas.harvard.edu}}
    }

\title{Nonparametric MLE for Gaussian Location Mixtures: 
Certified Computation and Generic Behavior}

\date{}

\begin{document}

\maketitle

\begin{abstract}
    We study the nonparametric maximum likelihood estimator $\wh\pi$ for Gaussian location mixtures in one dimension. 
    It has been known since \cite{lindsay1983geometryA} that given an $n$-point dataset, this estimator always returns a mixture with at most $n$ components, and more recently \cite{polyanskiy2020self} gave a sharp $O(\log n)$ bound for subgaussian data. In this work we study computational aspects of $\wh \pi$.
    We provide an algorithm which for small enough $\eps>0$ computes an $\eps$-approximation of $\wh\pi$ in Wasserstein distance in time $K+Cnk^2\log\log(1/\eps)$. Here $K$ is data-dependent but independent of $\eps$, while $C$ is an absolute constant and $k=|\supp(\wh\pi)|\leq n$ is the number of atoms in $\wh\pi$. 
    We also certifiably compute the exact value of $|\supp(\wh\pi)|$ in finite time. 
    These guarantees hold almost surely whenever
    the dataset $(x_1,\dots,x_n)\in [-cn^{1/4},cn^{1/4}]$ consists of independent points from a probability distribution with a density (relative to Lebesgue measure).
    We also show the distribution of $\wh\pi$ conditioned to be $k$-atomic admits a density on the associated $2k-1$ dimensional parameter space for all $k\leq \sqrt{n}/3$, and almost sure locally linear convergence of the EM algorithm. 
    One key tool is a classical Fourier analytic estimate for non-degenerate curves.
\end{abstract}

\section{Introduction}

The \emph{nonparametric maximum likelihood estimator} (NPMLE) has a long history in statistical problems including density estimation, regression, and mixture models
(see \cite{groeneboom2012information}).
This article concerns the NPMLE for the $1$-dimensional Gaussian location model, which has been
studied since 1950s, cf.~\cite{kiefer1956consistency,robbins1950generalization}. To introduce this problem, for any probability
distribution $\pi$ on $\mathbb{R}$ we denote $P_\pi \triangleq \pi * \cN(0,1)$ the probability density function of its convolution with a standard Gaussian density:
\begin{equation}
\label{eq:P-def}
    P_{\pi}(x)
    =
    \bbE^{y\sim\pi}[e^{-|x-y|^2/2}/\sqrt{2\pi}]
\end{equation}
Given a finite sequence $X=(x_1,\dots,x_n)\in \mathbb R^n$, the NPMLE $\wh\pi\in\cP(\mathbb R)$ is a probability measure on $\mathbb R$ chosen so that $P_{\wh\pi}=\wh\pi * \cN(0,1)$ maximizes the log-likelihood
\begin{equation}
\label{eq:NPMLE-def}
    \ell_X(\pi)
    \triangleq
    \frac{1}{n}\sum_{i=1}^n
    \log P_{\pi}(x_i)
    \,.
\end{equation}
Despite the fact that $\wh\pi$ is defined as a solution of an infinite-dimensional optimization problem, both the problem and its solution are surprisingly well-behaved. First, the maximizer $\wh\pi$ exists and is unique~\cite{lindsay1993uniqueness}. Second,
despite being defined as an optimum over all possible probability measures, $\wh\pi$ is discrete, so that the $P_{\wh \pi}$ becomes a finite mixture. Furthermore, even though this estimator lacks any explicit regularization, it nevertheless enjoys spectacular statistical convergence properties, for example achieving best known density estimation rates~\cite{jiang2009general,saha2020nonparametric}.

The NPMLE is tractable to study in part because of its nature as an overparametrized convex relaxation. Indeed, while it is well known that the landscape of the maximization objective~\eqref{eq:NPMLE-def} if restricted to $k$-atomic distributions $\pi$ is non-concave and has spurious local maxima~\cite{jin2016local}, over the space of all measures $\mathcal{P}(\mathbb{R})$ the problem is concave, and thus can be characterized by local optimality conditions.

Practical solvers for maximizing~\eqref{eq:NPMLE-def} started with~\cite{Laird1978} proposing a variant of the then-recently discovered EM algorithm. Due to its very slow convergence, a number of other algorithms were proposed over the years mostly differing in how new locations (atoms) are added at each iteration, see e.g.\ \cite{dersimonian1986algorithm,bohning1986vertex,lesperance1992algorithm,bohning1992computer} and \cite[Chapter 6]{lindsay1983geometryB} for a detailed survey. However, a decade ago~\cite{koenker2014convex} discovered that due to the progress in convex optimization, the (empirically) fastest and most accurate way to maximize~\eqref{eq:NPMLE-def} is to fix the support of $\pi$ to be a fine equi-spaced grid $\eps \mathbb{Z}$ (truncated at the range of the samples $X$) and maximize the concave function $\ell_X(\pi)$ of the weights of $\pi$ via off-the-shelf software. This is the strategy implemented in the popular package REBayes~\cite{koenker2017rebayes}.
More recent methods based on general-purpose convex programming have also been proposed~\cite{kim2020fast,wang2025nonparametric}.

The ubiquity and empirical success of these approximate algorithms for finding $\wh \pi$ raises natural questions. If the grid-based NPMLE convex optimization algorithm hits its stopping criterion and returns a 3-atomic solution $\pi$, can we provably convince ourselves that the true NPMLE $\wh \pi$ has $3$ and not $1$ or $100$ atoms? More generally, given that an algorithm (heuristic or otherwise) appears to have approximately converged to some $\pi$, can we efficiently \textbf{certify} (prove) that $\wh\pi$ is $\eps$-close to $\pi$? 
Do any of the iterative algorithms converge to $\wh\pi$ at provably efficient rates?
Because the objective \eqref{eq:NPMLE-def} is poorly conditioned, off-the-shelf convex optimization theory does not provide such guarantees. 

We provide answers to all of these questions as a consequence of a new structural property of $\wh \pi$, perhaps of independent interest. We show that under random data the true NPMLE solution $\wh\pi$ has a certain \emph{generic behavior}. Namely, suppose $x_1,\dots,x_n$ are drawn IID\footnote{Our genericity results also hold if $x_1,\dots,x_n$ are drawn independently from different probability densities. In fact this is an immediate corollary of the IID case since the conclusions hold almost surely: given probability densities $\mu_1,\dots,\mu_n$, any event that holds almost surely for IID $x_1,\dots,x_n\sim (\mu_1+\dots+\mu_n)/n$ also holds almost surely for independent $x_i\sim \mu_i$.} from an absolutely continuous distribution on $\bbR$; we then say $X$ is \textbf{generic}. 
As recalled below, it has been known since Lindsay \cite{lindsay1983geometryA} that $\wh\pi$ is characterized by always local optimality conditions.
Our results show these local optimality conditions for $\wh\pi$ are almost surely \emph{strict} when $X$ is generic, as long as $\max_i |x_i|\leq cn^{1/4}$ for a small absolute constant $c$.

This almost sure strictness has computational consequences, enabling us to derive the first provably efficient algorithms for $\wh\pi$. Although it is not difficult to approximately maximize the objective $\ell_X$, or to $\eps$-approximate $\wh\pi$ with exponential or worse dependence on $\eps$ in run-time, neither of these yields efficient approximation of $\wh\pi$ (e.g. polynomial dependence on $1/\eps$). 
Using our genericity results, we obtain almost sure locally quadratic convergence in parameter distance distance for generic data.
This implies the same convergence rate in Wasserstein distance and exact computation of the number of atoms in $\wh\pi$.
Importantly, our algorithms certify an explicit Wasserstein error bound and exact number of atoms in finite time, so they have a well-defined output rather than just asymptotic convergence.
Our algorithms follow a simple template: first, use a convex optimization algorithm such as Frank--Wolfe to compute a sparse approximate NPMLE $\wh\pi_{\eps}$ which is supported on an $\eps$-grid such as $\eps \bbZ$.
Next, merge any adjacent atoms (i.e. at distance $\eps$) within the support of $\wh\pi_{\eps}$.
Finally, attempt to certify that the result is close enough to $\wh\pi$ that Newton's method converges quadratically, based on the local concavity of $\ell$. 
We prove that a version of this approach succeeds almost surely for generic $X$, once $\eps$ is sufficiently small.

As preparation, we recall the classical stationarity conditions characterizing $\wh\pi$.
Define the function $D_{\pi,X}:\bbR\to\bbR$ by
\begin{align}
\label{eq:D-formula}
    D_{\pi,X}(y)
    &=
    \frac{1}{n}\sum_{i=1}^n
    \frac{e^{-|x_i-y|^2/2}}{P_{\pi}(x_i)\sqrt{2\pi}}
    \triangleq
    \frac{1}{n}\sum_{i=1}^n T_{\pi,x_i}(y)
    .
\end{align}
$D_{\pi,X}$ is the derivative of $\ell$ with respect to perturbations in $\pi$, and plays a central role in the following characterization of $\wh\pi$ (see \cite{lindsay1983geometryA}).

\begin{proposition}
\label{prop:stationarity-conditions} 
For any $\pi$ and $\pi'$, denoting $\pi_t=(1-t)\pi+t\pi'$, we have
\begin{align}
\label{eq:D-sum-formula}
    \bbE^{y\sim\pi}
    D_{\pi,X}(y)
    &=
    1,
    \\
    \label{eq:ell-deriv-formula}
    \left.\frac{\de \ell(\pi_t)}{\de t}\right|_{t=0+}
    &=
    \int D_{\pi,X}(y) 
    ~
    \de\big(\pi'-\pi\big)(y).
\end{align}
The minimizer $\wh\pi$ of~\eqref{eq:NPMLE-def} is unique and $k$-atomic for some $k\leq n$, and satisfies for all $y\in\supp(\wh\pi)$:
    \begin{equation}
    \label{eq:stationarity}
    \begin{aligned}
        D_{\wh\pi,X}(y)&=1,
        \\
        D'_{\wh\pi,X}(y)&=0
    \end{aligned}
    \end{equation}
    Moreover $D_{\wh\pi,X}(y)\leq 1$ for all $y\in\bbR$, so that
    \begin{equation}
    \label{eq:hat-pi-global-max}
        \supp(\wh\pi)
        \subseteq
        \argmax(D_{\wh\pi,X}).
    \end{equation}
\end{proposition}

We note that uniqueness of $\wh \pi$ is non-obvious. Indeed, from the strong concavity of the
log it is easy to see that all maximizers of $\ell(\pi)$ have the same vector $P_\pi(X)
\triangleq (P_\pi(x_1), \ldots, P_\pi(x_n))$. However, the linear map $\pi \mapsto P_\pi(X)$ with
domain being signed measures has
infinite dimensional pre-image. The surprising part of Proposition~\ref{prop:stationarity-conditions} is that intersecting this pre-image with the
subset of probability distributions yields a unique $\wh \pi$.

\subsection{Structural Results}

Our first result proves essentially that  (with high probability) the converse implications of Lindsay conditions hold: for any $y \not \in \supp (\wh \pi)$ we have 
\[
D_{\wh \pi, X}(y) < 0
\]
and for every $y \in \supp (\wh \pi)$ the point $y$ is a non-degenerate maximum:
\[
D_{\wh \pi, X}'(y) = 0, \qquad D_{\wh \pi, X}''(y) < 0\,.  
\]

To that end we provide structural results showing that whenever $X$ is generic, both $\wh\pi$ and $D_{\wh\pi,X}$ also behave ``generically''. Specifically,  we prove that $\wh\pi$ admits a density on the natural parameter space, conditional on the number of atoms. Furthermore, one would like to think of $D_{\wh\pi,X}$ as a generic smooth function, but in light of Proposition~\ref{prop:stationarity-conditions} it may have multiple global maxima, which is not generic for a smooth function. Theorem~\ref{thm:main} parts \ref{it:no-coincidence-1},\ref{it:no-coincidence-2} essentially say that $D_{\wh\pi,X}$ behaves generically otherwise, in that it has no ``spurious'' global maxima, and all maxima are well-conditioned.

Below we define $\Pi_k$ to consist of all $\pi=\sum_{i=1}^k p_i \delta_{y_i}$ with exactly $k$ distinct atoms, which can be parametrized as an open subset of $\mathbb R^{2k-1}$; see \eqref{eq:Pik-def} for a more precise definition.
We also define $\Pi_{k,\eps}\subseteq\Pi_k$ to consist of those $\pi$ with $\min_i p_i\geq \eps$ and $\min_{i<j} |y_i-y_j|\geq \eps$.

\begin{theorem}
\label{thm:main}
    Let $X=(x_1,\dots,x_n)$ be IID from an absolutely continuous distribution $\mu$ on $\bbR$.
    Then for any $k$ such that $n>(2k+2)^2$:
    \begin{enumerate}[label=(\Roman*)]
    \item 
    \label{it:abs-cont}
    The restriction of the law of $\wh\pi$ to $\Pi_k$ is absolutely continuous.
    Further if $\mu$ has density supported in $[-L,L]$ with values in $[0,L']$, then the density of $\wh\pi$ is locally bounded, i.e. at most $C(n,k,\eps,L,L')$ on $\Pi_{k,\eps}$ for any $\eps>0$.
    \item
    \label{it:no-coincidence-1}
    Conditional on $|\supp(\wh\pi)|=k$ (assuming $k$ is such that $\bbP[|\supp(\wh\pi)|=k]>0$), the function $D_{\wh\pi,X}$ almost surely has exactly $k$ global maxima: 
    \begin{equation}
    \label{eq:no-coincidence-1}
    \supp(\wh\pi)=\argmax_{y\in \bbR} D_{\wh\pi,X}(y).
    \end{equation}
    \item
    \label{it:no-coincidence-2} 
    Conditional on $|\supp(\wh\pi)|=k$ (assuming $k$ is such that $\bbP[|\supp(\wh\pi)|=k]>0$), each global maximum of $D_{\wh\pi,X}$ is almost surely non-degenerate: 
    \begin{equation}
    \label{eq:no-coincidence-2}
    \max\limits_{y\in\supp(\wh\pi)}D_{\wh\pi,X}''(y)<0.
    \end{equation}
    \end{enumerate}
\end{theorem}

Using this result, we obtain further consequences on the local landscape of $\ell_X$ near $\wh\pi$.

\begin{theorem}
\label{thm:landscape-consequences}
For $1\leq k\leq n$, if the conclusions of Theorem~\ref{thm:main} \ref{it:no-coincidence-1}, \ref{it:no-coincidence-2} hold, then:
\begin{enumerate}[label=(\Alph*)]
    \item
    \label{it:local-strong-convexity}
    There is an open neighborhood of $\wh\pi$ in $\Pi_k$ on which $\ell_X$ is locally $c$-strongly concave for some $c>0$.
    \item
    \label{it:local-smoothness}
    There is an open neighborhood of $X$ in $\mathbb R^n$, such that $\wh\pi(\hat X)\in\Pi_k$ for all $\hat X$ in this neighborhood.
    Moreover, $\wh\pi$ is a smooth function on this neighborhood.
    \item 
    \label{it:EM-algorithm}
    The expectation-maximization (EM) algorithm converges linearly in a small $\Pi_k$-neighborhood of $\wh\pi$.
    Namely for any initialization $\pi_0$ in this neighborhood, the EM iterates $\pi_0,\pi_1,\dots$ satisfy $d_{\Pi_k}(\pi_t,\wh\pi)\leq C\cdot (1-\eta)^t d_{\Pi_k}(\pi_0,\wh\pi)$ for $C,\eta>0$. (Here $d_{\Pi_k}$, defined in \eqref{eq:Pik-distance}, denotes parameter distance in $\Pi_k$.)
\end{enumerate}
\end{theorem}

We emphasize that Theorem~\ref{thm:main}\ref{it:no-coincidence-1} improves the classical stationarity condition \eqref{eq:hat-pi-global-max} to a strict inequality.
Additionally, Theorem~\ref{thm:landscape-consequences}\ref{it:local-strong-convexity} is different from the concavity of $\ell_X$ on the space of probability measures, since the structure of $\Pi_k$ allows averaging the locations of atoms.
We also mention that Appendix~\ref{app:positive-prob-k-atoms} shows IID data with full support on $[0,O(k\sqrt{\log k})]$ satisfies $\bbP[|\supp(\wh\pi)|=k]>0$ whenever $n\geq k$, so the conditional distribution of $\wh\pi$ on $\Pi_k$ is defined.

For Theorem~\ref{thm:main} to be useful, we must have $k=|\supp(\wh\pi)|\leq O(\sqrt{n})$, which is not always the case.
However, following a related conjecture of \cite{koenker2019minimalist}, the first author showed in \cite{polyanskiy2020self} with Yihong Wu that the much stronger bound $|\supp(\wh\pi)|\leq O(\log n)$ holds with high probability for sub-Gaussian data (and is sharp in natural examples). 
In fact the bound depends deterministically on the empirical range of the data.

\begin{proposition}[{\cite[Theorem 1]{polyanskiy2020self}}]
\label{prop:PW20}
    There exists a universal constant $C$ such that if $X=(x_1,\dots,x_n)\in [-L,L]^n$ for $L\geq 1$, then $|\supp(\wh\pi)|\leq CL^2$.
\end{proposition}

Corollary~\ref{cor:light-tails} below is an immediate consequence: $\wh\pi$ exhibits generic behavior for generic data in $[-cn^{1/4},cn^{1/4}]$. 
For example, \eqref{eq:hat-pi-global-max} is almost surely an equality for such $X$.

\begin{corollary}
\label{cor:light-tails}
Suppose $X=(x_1,\dots,x_n)$ is IID from an absolutely continuous distribution.
Then with probability $1$, either both the conclusions \eqref{eq:no-coincidence-1},\eqref{eq:no-coincidence-2} of Theorem~\ref{thm:main}~\ref{it:no-coincidence-1},\ref{it:no-coincidence-2} hold, or otherwise $\max_{i}|x_i|\geq cn^{1/4}$ holds for some absolute constant $c>0$.

In particular if $\sup_i \bbE[|x_i|^q]<\infty$ for some $q>4$, then \eqref{eq:no-coincidence-1},\eqref{eq:no-coincidence-2} hold with probability $1-O(n^{1-\frac{q}{4}})$.
\end{corollary}

\begin{remark}
\label{rem:generic-relaxation-and-conditioning}
    The genericity conditions above can be relaxed to apply only to part of the dataset.
    This is relevant in settings where some datapoints have been rounded or are correlated with each other.
    Namely, fix arbitrary $x_{m+1},\dots,x_n\in [-L,L]$.
    Then if $m\geq CL^4$ and $x_1,\dots,x_m\in [-L,L]$ are generic conditionally on $(x_{m+1},\dots,x_n)$, all the conclusions of Theorem~\ref{thm:main} still apply with essentially the same proof.
\end{remark}

In the setting of Remark~\ref{rem:generic-relaxation-and-conditioning}, we show in the following Theorem~\ref{thm:absolute-continuity-false} that Theorem~\ref{thm:main}\ref{it:abs-cont} is sharp in some sense. Namely with $m\leq O(k^2)$ generic datapoints, the density of $\wh\pi$ on $\Pi_k$ may fail to be locally bounded. 
The proof is given in Appendix~\ref{app:positive-prob-k-atoms}, where we also prove Theorem~\ref{thm:landscape-consequences}.
It is an interesting open problem whether Theorem~\ref{thm:absolute-continuity-false} truly requires $m<n$.

\begin{theorem}
\label{thm:absolute-continuity-false}
    Suppose $m\leq (k-1)(2k-1)$ and $n-m\geq (2k+2)^2$.
    Let $x_1,\dots,x_n$ be IID uniform on $[0,O(k\sqrt{\log k})]$.
    Then with positive probability, the conditional law of $\wh\pi$ given $x_{m+1},\dots,x_n$ and the event $\wh\pi\in \Pi_k$ does not have locally bounded density on $\Pi_k$.
\end{theorem}

\subsection{Efficient Computability of $\wh\pi$}

A related and somewhat subtle question is whether $\wh\pi$ is efficiently computable. 
Because $\ell$ is concave, one can find in $poly(1/\eps)$ time a probability distribution $\pi_{\eps}$ such that $|\ell_X(\pi_{\eps})-\ell_X(\wh\pi)|\leq \eps$ (as recently observed in \cite{fan2023gradient}).
Uniqueness of $\wh\pi$ implies the convergence $\pi_{\eps}\to\wh\pi$ in the space of probability measures on $\bbR$ as $\eps\to 0$.
However $\ell_X$ may be ill-conditioned, so it is not clear that one obtains a quantitative rate of convergence of $\pi_{\eps}$ to $\wh\pi$.
Worse, this proof does not give any \emph{stopping rule} at which one can \emph{guarantee} that $\pi_{\eps}$ is within $\eps$ distance from $\wh\pi$.
From the point of view of computational complexity theory, this means one does not yet have an actual algorithm to compute $\wh\pi$.

We will aim to compute $\wh\pi$ in the strong \emph{parameter distance} in which all centers and mixing weights are well-approximated.
This includes the following two questions:

\begin{enumerate}
    \item \emph{Can $\wh\pi$ be certifiably approximated in Wasserstein distance to error $\eps$ in $\eps^{-O(1)}$ time?}
    \item \emph{Is the support size $|\supp(\wh\pi)|$ a computable function of $X=(x_1,\dots,x_n)$?}
\end{enumerate}

We have illustrated that the first question is non-obvious, but the second merits its own discussion.
Note that from a Wasserstein approximation to $\wh\pi$, one can \emph{never} directly conclude the value of $|\supp(\wh\pi)|$: any atom may split onto several smaller atoms with negligible effect in Wasserstein distance. 
Although support size is a very brittle property, Proposition~\ref{prop:PW20} shows a strong upper bound on $|\supp(\wh\pi)|$, so it is natural to hope that exact computation is feasible.
This task requires understanding the distributional behavior of $\wh\pi$ at microscopic scales, a challenging task which has seen no previous work.
Intuitively, it is natural to expect those $\wh\pi$ exhibiting potential instability to correspond to a low-dimensional submanifold in parameter space, which should not occur for generic $X$. Formalizing this intuition is challenging due to the multi-scale nature of the parameter space $\cup_{k\geq 1}\Pi_k$ for $\wh\pi$.

To formulate the computational problem in a satisfactory way, we use the notion of a Shub--Smale approximate solution to a system of equations, originally due to \cite{shub1993complexity}.
In short, $\pi$ is a Shub--Smale approximate NPMLE if it is within $\eps$ parameter distance of $\wh\pi$, and Newton's method exhibits quadratic convergence to $\wh\pi$ starting from $\pi$.
To be precise, starting from any $\pi\in\Pi_k$, one may naturally define a Newton--Raphson iteration to attempt to compute $\wh\pi\in\Pi_k$ (viewing $\Pi_k$ as an open subset of $\mathbb R^{2k-1}$).
Recalling Proposition~\ref{prop:stationarity-conditions}, the point is to view $\wh\pi$ as the zero of the function $\gamma_k:(p_1,\dots,p_k,y_1,\dots,y_k)\in \Pi_k\to\mathbb R^{2k-1}$ defined by
\[
\gamma_k(\pi)
=
\big(D_{\pi}(y_1)-1,\dots,
D_{\pi}(y_k)-1,
D_{\pi}'(y_1),
\dots
D_{\pi}'(y_k)\big).
\]
($\gamma_k$ is essentially the gradient of $\ell_X$ in the space $\Pi_k$.)
The iteration begins with $\pi^{(0)}=\pi$ and (with $J$ the Jacobian) then recursively sets:
\begin{equation}
\label{eq:newton-method}
\pi^{(t+1)}
=
\pi^{(t)}
-
[J\gamma_k(\pi^{(t)})]^{-1} \gamma_k(\pi^{(t)}).
\end{equation}
Since $\Pi_k$ is a proper subset of $\bbR^{2k-1}$, it is possible that even if $\pi^{(t)}\in\Pi_k$ and $\wh\pi\in\Pi_k$, the next iterate $\pi^{(t+1)}\notin\Pi_k$ may satisfy $p_i<0$ or $y_i>y_{i+1}$ for some $i$.
If this occurs, we stop the iteration and declare failure.

\begin{definition}
\label{def:shub-smale-approx-npmle}
    We say $\pi\in \Pi_k$ is a Shub--Smale $\eps$-approximate NPMLE if:
    \begin{enumerate}
        \item $\wh\pi\in\Pi_k$ and $d(\pi^{(0)},\wh\pi)_{\Pi_k}\leq \eps$.
        \item The iteration \eqref{eq:newton-method} satisfies $\pi^{(t)}\in\Pi_k$ for all $t$ (and never declares failure).
        \item We have $d(\pi^{(t)},\wh\pi)_{\Pi_k}\leq 2^{1-2^t} d(\pi^{(0)},\wh\pi)_{\Pi_k}$ for all $t$.
    \end{enumerate} 
    Here we use the parameter distance $d(\cdot,\cdot)_{\Pi_k}$ defined in \eqref{eq:Pik-distance}, which is simply Euclidean distance on the parametrization $\pi=(p_1,\dots,p_k,y_1,\dots,y_k)$.
\end{definition}

Note that the value of $\eps$ in Definition~\ref{def:shub-smale-approx-npmle} is not particularly important: once the condition holds for some $\eps<1/2$, applying Newton's method converges quadratically.
Our computational results use the genericity of Theorem~\ref{thm:main} to find a Shub--Smale approximate NPMLE almost surely.
We note that the issue of an unknown number of parameters $k$ is not present in \cite{shub1993complexity}. However this arises naturally in our setting because $|\supp(\wh\pi)|$ must be known exactly for a Newton--Raphson iteration to make sense.
Indeed, it underlines the point that exact computation of $\wh\pi$ inherently \textbf{requires} exact computation of the support size.

\subsection{Na\"ive Brute-Force Approximation of $\wh\pi$}
\label{subsec:naive}

To further motivate and illustrate our algorithmic results, we discuss two na\"ive algorithms to approximate $\wh\pi$ in Wasserstein distance. The first algorithm is essentially given by the next proposition.
Here and throughout the rest of the paper, for $L,\eps>0$ we let $Z_{\eps}\subseteq\bbR$ be any set satisfying:
\begin{equation}
\label{eq:Z-eps-def}
\begin{aligned}
    \{-L,L\}\subseteq Z_{\eps}
    &\subseteq 
    [-L,L],
    \\
    |Z_{\eps}|
    &\leq 
    3L/\eps,
    \\
    \min\limits_{y\in [-L,L]} d(y,Z_{\eps})
    &\leq 
    \eps
    .
\end{aligned}
\end{equation}
The sets $Z_{\eps}$ will serve as supports for approximations $\wh\pi_{\eps}$ to $\wh\pi$.\footnote{Throughout the paper, one can just take $Z_{\eps}=\eps\bbZ\cap [-L,L]$ for any $L,\eps^{-1}\in\bbN$. However in practical algorithms, the atoms of an approximating $\wh\pi_{\eps}$ are often adjusted during the optimization. Augmenting $Z_{\eps}$ with these additional points allows us to cover such situations.}

\begin{proposition}
\label{prop:W1-approx-inefficient}
    For each $\eps>0$, let 
    \begin{equation}
    \label{eq:hat-pi-eps-def}
        \wh\pi_{\eps} \in \argmax_{\supp(\pi)\subseteq Z_{\eps}} \ell_X(\pi)
    \end{equation}
    maximize the log-likelihood for $X$ among all distributions supported in $Z_{\eps}$. 
    Then
    \[
    \lim_{\eps\to 0}\wh\pi_{\eps} = \wh\pi,
    \]
    where convergence is in the Wasserstein $\bbW_1$ sense.
\end{proposition}

\begin{proof}
    As we show in Lemma~\ref{lem:lipschitz-bounds}, the map $\pi\mapsto \ell_X(\pi)$ from $\bbW_1([-L,L])\to\bbR$ has Lipschitz constant at most $C_L\leq O(e^{4L^2})$.
    It follows easily that 
    \begin{equation}
    \label{eq:ell-X-eps-approx}
    \ell_X(\wh\pi_{\eps})\leq \ell_X(\wh\pi)+C_L \eps.
    \end{equation}
    Since the sets $Z_\eps$ are uniformly bounded, there must exist subsequential limits of $\wh\pi_\eps$ as $\eps\to 0$. Since the functional
    $\pi \mapsto \ell_X(\pi)$ is weakly continuous, all subsequential limits globally minimize $\ell_X$. Uniqueness of
    $\wh\pi$ (see Proposition~\ref{prop:stationarity-conditions}) completes the proof.
\end{proof}

It also follows from \eqref{eq:ell-X-eps-approx} that $\ell_X(\pi)$ is $C_{L}\eps^{-1}$-Lipschitz when $\pi$ is metrized by the Euclidean norm on $\bbR^{|Z_{\eps}|}$ (via its probability mass function).
Since $\pi\mapsto \ell_X(\pi)$ is concave for each $X$, given a choice of $Z_{\eps}$, computation
of $\wh\pi_{\eps}$ to $\bbW_1$ to accuracy $\delta$ requires $C_L \eps^{-O(1)}\log(1/\delta)$ gradient evaluations of $\ell_X$ using a cutting plane method (see e.g. \cite[Theorem 2.4]{bubeck2015convex}).

Unfortunately the algorithm ``compute $\wh\pi_{\eps}$ for small $\eps>0$'' gives \emph{no quantitative guarantees} for the approximation of $\wh\pi$ itself. That is, the above argument cannot certify upper bounds $\bbW_1(\wh\pi_{\eps},\wh\pi)\leq \eta$ for any nontrivial $\eta>0$.
Indeed, $\wh\pi_{\eps}$ above is essentially the maximum of the concave function $\ell_X:\bbW_1([-L,L])\to\bbR$ on the $2L/\eps$ dimensional subspace of $\pi$ supported in $Z_{\eps}$. 
Although $\wh\pi$ is very close to this subspace, a \emph{strong} concavity estimate would be needed to upper bound $\bbW_1(\wh\pi_{\eps},\wh\pi)$ from this. 
For example the concave function 
\begin{equation}
\label{eq:example}
f(x_1,\dots,x_d)=-\lt(x_1-\frac{x_2}{10^4}\rt)^2
-
\frac{1}{d}\sum_{i\in [d]\backslash \{2\}}x_i^2
\end{equation}
is maximized at the origin, but its minimizer on the nearby hyperplane $\{x_1=0.01\}$ has $x_2=100$.

A second try is to explicitly search the entirety of $\bbW_1([-L,L])$.
Namely if $\cN_{\eps}\subseteq \bbW_1([-L,L])$ is an $C_L^{-1}\eps$-net, then the choice
\[
\wh\pi_{\eps}
=
\argmin_{\pi\in \cN_{\eps}}
\ell_X(\pi)
\]
also satisfies \eqref{eq:ell-X-eps-approx}.
Moreover \textbf{if} for some $\eta>0$ it happens to be the case that 
\begin{equation}
\label{eq:brute-force-2-sufficient}
\min_{\substack{\pi\in\cN_{\eps}\\ \bbW_1(\pi,\wh\pi_{\eps})\geq \eta-\eps}}
\ell_X(\pi)
\geq \ell_X(\wh\pi_{\eps})+3\eps,
\end{equation}
\textbf{then} this would immediately certify the bound $\bbW_1(\wh\pi_{\eps},\wh\pi)\leq \eta$.
Therefore a a natural approach to certifiable approximation guarantees would be to verify the estimate in \eqref{eq:brute-force-2-sufficient}. 
However a bit of thought reveals this approach is also impractical.
Firstly, it is again non-quantitative: uniqueness of $\wh\pi$ implies \eqref{eq:brute-force-2-sufficient} holds eventually (i.e. for $0<\eps\leq \eps_0(\eta)$ sufficiently small depending on $\eta>0$), but this argument does not predict any quantitative dependence between $\eps$ and $\eta$.
Additionally, $|\cN_{\eps}|$ grows exponentially\footnote{
For example let $S_1,\dots,S_K\subseteq Z_{\eps}=[-L,L]\cap\eps\bbZ$ be IID uniformly random subsets of size $|Z_{\eps}|/2$, and let $\pi_i$ be the uniform distribution on $S_i$ for each $i$. For an absolute constant $c>0$ and $K\leq c\exp(c/\eps)$, we will have $|S_i\cap S_j|\leq |Z_{\eps}|/3$ and hence $\bbW_1(\pi_i,\pi_j)\geq \eps/10$ with high probability, simultaneously for all $1\leq i<j\leq K$.
} 
in $1/\eps$, making this computationally inefficient even if $\eta$ and $\eps$ turn out to be polynomially related.

\subsection{Results on Computability and Generic Behavior}

Having discussed several pitfalls, we now present our main computability results, which show that $\wh\pi$ can be efficiently approximated and the support size $k$ can be exactly computed.
We equip $Z_{\eps}$ (recall \eqref{eq:Z-eps-def}) with the adjacent-neighbors graph structure, making it isomorphic to a path.
Below we write $O_X(\cdot)$ to indicate an implicit constant factor which is random and depends on $X$, but not on e.g. $\eps$.

\begin{theorem}
\label{thm:certify-intro}
    Assume $n,L$ satisfy $n\geq CL^4$ for an absolute constant $C$. Let $X=(x_1,\dots,x_n)\in [-L,L]^n$ be generic.
    Let $\wh\pi_{\eps}=\argmax\limits_{\supp(\pi)\subseteq Z_{\eps}} \ell_X(\pi)$.
    Then almost surely, for $\eps$ small enough depending on $X$:
    \begin{enumerate}[label=(\alph*)]
        \item 
        All connected components of $\supp(\wh\pi_{\eps})\subseteq Z_{\eps}$ (w.r.t. the graph structure) have size $1$ or $2$.
        \item 
        \label{it:number-of-atoms}
        $|\supp(\wh\pi)|$ is equal to the number of connected components of $\supp(\wh\pi_{\eps})$.
        \item 
        \label{it:W1-bound-main}
        $\bbW_1(\wh\pi,\wh\pi_{\eps})\leq O_X(\eps^{1/3})$,
    \end{enumerate}
    Further, the statement in \ref{it:number-of-atoms} and the upper bound in
    \ref{it:W1-bound-main} are efficiently \textbf{certifiable} given $\wh\pi_{\eps}$.
    Finally, the rounding of $\wh\pi_{\eps}$ which merges adjacent atoms $p_j \delta_{y_j}+p_{j+1}\delta_{y_{j+1}}$ in $Z_{\eps}$ to $(p_j+p_{j+1})\delta_{\frac{p_j y_j+p_{j+1}y_{j+1}}{p_j+p_{j+1}}}$ satisfies $d_{\Pi_k}(\wh\pi,\wh\pi_{\eps})\leq O_X(\eps^{1/3})$ and is certifiably a Shub--Smale approximate NPMLE.
\end{theorem}

In proving Theorem~\ref{thm:certify-intro}, we consider a slight modification $\wt\pi_{\eps}$ of $\wh\pi_{\eps}$ in which adjacent atoms are merged together. 
This in fact gives small error in the stronger parameter distance $d_{\Pi_k}$ (defined in \eqref{eq:Pik-distance}).
Since the exact maximization defining $\wh\pi_{\eps}$ may not be computationally feasible, we provide another variant of Theorem~\ref{thm:certify-intro} that does not require exact computation of $\wh\pi_{\eps}$ (proved in the Appendix).
Instead only an approximation is required, obtained using the Frank--Wolfe algorithm and a careful rounding scheme; we denote by $(\breve\pi_{\eps},\mathring\pi_{\eps})$ the corresponding analogs of $(\wh\pi_{\eps},\wt\pi_{\eps})$.

\begin{theorem}
\label{thm:certify-intro-approx}
    Under the conditions of Theorem~\ref{thm:certify-intro}, there exists for each $\eps>0$ a deterministic $O(Ln\eps^{-11})$ time algorithm which computes $\mathring\pi_{\eps}\in\cP(Z_{\eps})$ such that almost surely, for $\eps$ small enough depending on $X$:
    \begin{enumerate}[label=(\alph*)]
        \item 
        \label{it:number-of-atoms-approx}
        $|\supp(\wh\pi)|=|\supp(\mathring\pi_{\eps})|$.
        \item 
        \label{it:W1-bound-main-approx}
        $d_{\Pi_k}(\wh\pi,\mathring\pi_{\eps})\leq O_X(\eps^{1/4})$ and $\bbW_1(\wh\pi,\mathring\pi_{\eps})\leq O_X(\eps^{1/4})$.
    \end{enumerate}
    Further, the statement in \ref{it:number-of-atoms-approx} and the upper bound in \ref{it:W1-bound-main-approx} are efficiently certifiable given $\mathring\pi_{\eps}$, which is a Shub--Smale approximate NPMLE.
\end{theorem}

In other words, Theorem~\ref{thm:certify-intro-approx} provides a pair of efficient algorithms. The first computes a $\bbW_1$-approximation of $\wh \pi$, while the second attempts to certify the $\bbW_1$ bound and support size equality. The second either returns a checkable proof or fails. What we show is that by rerunning these algorithms with smaller and smaller $\eps$, eventually the second algorithm will succeed; furthermore, the bound will decay as $\eps^{1/4}$. 
Once the algorithm succeeds for some $\eps=\eps_0$, there is no longer a need to continue rerunning the same algorithm to decrease $\eps$.
Instead, one can simply run Newton's method within $\Pi_k$.
Hence for $\delta\ll\eps_0$, the computational complexity will be $O(Ln\eps_0^{-11})+C_{\mathsf{NR}}\log\log(\eps_0/\delta))$, where $C_{\mathsf{NR}}\leq O(nk^2)$ is the complexity of a Newton--Raphson iteration. (Computing the gradient of $\gamma_k$ uses $nk^2$ time, inverting an $O(k)\times O(k)$ matrix takes $O(k^3)\leq O(nk^2)$ time, and other steps are faster.)

Let us explain briefly why Theorem~\ref{thm:main}\ref{it:no-coincidence-1} and \ref{it:no-coincidence-2} are useful towards
Theorem~\ref{thm:certify-intro}. The chief worry in Theorem~\ref{thm:certify-intro} is that
although the NPMLE objective \eqref{eq:NPMLE-def} is a concave maximization problem, it is infinite-dimensional and may be quite poorly conditioned.  
Because the log-likelihood can be shown to be relatively smooth, \textbf{if} it is not flat near $\wh\pi$, \textbf{then} one will be able to efficiently certify $\wh\pi\approx\wt\pi$ based on local information at $\wt\pi$ (namely approximate-stationarity and Hessian non-singularity).
In Section~\ref{sec:cert}, we employ Theorem~\ref{thm:main}\ref{it:no-coincidence-1} and \ref{it:no-coincidence-2} to show the necessary conditions hold once
$\wt\pi$ is a sufficiently accurate approximation for $\wh\pi$.

While Newton--Raphson iteration is appealing due to its quadratic local convergence rate, other approaches also suffice for asymptotic convergence from an approximate solution.
In particular Theorem~\ref{thm:landscape-consequences}\ref{it:EM-algorithm} shows that the EM algorithm converges linearly from suitable approximate solutions; this can be similarly made certifiable from a sufficiently good approximate solution $\mathring \pi_{\eps}$.

\subsection{Results for the Static Support NPMLE}
\label{subsec:static-support}

To illustrate the flexibility of our methods, we also consider the \emph{static support} NPMLE.
Given a fixed finite set $S\subseteq \bbR$ (independent of $X$), we define the static support NPMLE $\wh\pi_S$ as in \eqref{eq:NPMLE-def}, but restricted to $\pi$ supported within $S$.
Similarly let $\Pi_k(S)$ be the $(k-1)$-dimensional set of $k$-atomic distributions on $S$.
The following analog of Proposition~\ref{prop:stationarity-conditions} is immediate from \eqref{eq:ell-deriv-formula}.

\begin{proposition}
\label{prop:stationarity-conditions-warmup} 
Fix a finite $S\subseteq \bbR$ and $X=(x_1,\dots,x_n)$. Then any minimizer $\wh \pi_S$ of \eqref{eq:NPMLE-def} among $\pi$ supported on $S$ satisfies for all $y\in\supp(\wh\pi_S)$:
    \begin{equation}
    \label{eq:stationarity-warmup}
        D_{\wh\pi_S,X}(y)=1.
    \end{equation}
    Moreover $D_{\wh\pi_S,X}(y)\leq 1$ for all $y\in S$, so that
    \begin{equation}
    \label{eq:S-hat-pi-global-max}
        \supp(\wh\pi_S)
        \subseteq
        \argmax_S(D_{\wh\pi_S,X})
    \end{equation}
\end{proposition}

Uniqueness does not appear to follow from classical results; the proof in \cite[Theorem 5.1]{lindsay1983geometryB} requires $\supp(\wh\pi)$ to avoid the boundary of $S$, which is of course impossible when $S$ is discrete.
Thus in principle, $\wh\pi_S$ refers to any maximizer of $\ell_X$.
However for bounded data, uniqueness follows by the technique of \cite{polyanskiy2020self}.

\begin{proposition}
\label{prop:uniqueness-static-support}
Let $L\leq cn^{1/2}$ for a small absolute constant $c$.
Fix a $L$-bounded $X=(x_1,\dots,x_n)$, and finite $S\subseteq \bbR$ with $S\cap [-3L,3L]\neq\emptyset$.
Then $\wh\pi_S$ is unique and $|\supp(\wh\pi_S)|\leq O(L^2)$.
\end{proposition}

\begin{proof}
    Let $\wh\pi_S$ be any minimizer of \eqref{eq:NPMLE-def} among probability distributions supported in $S$. It is clear that $\wh\pi_S$ is supported in $[-10L,10L]$. 
    It follows from \cite[Proof of Theorem 3]{polyanskiy2020self} that $D_{\wh\pi_S,X}$ has $O(L^2)$ critical points.
    Since $D_{\wh\pi_S,X}$ takes the value $1$ at each point in $S$, Rolle's theorem implies that $|\supp(\wh\pi_S)|\leq O(L^2)<n/2$ (for $c$ small).
    Since this support bound holds for all minimizers $\wh\pi_S$, \cite[Lemma 6.1]{lindsay1983geometryB} implies $\wh\pi_S$ is unique. 
\end{proof}

Our main results on generic behavior and efficient computability both have analogs for finite $S$; the proofs are similar to those of the results presented so far.

\begin{theorem}
\label{thm:main-warmup}
    Fix $n\geq k^2$ and a finite set $S\subseteq \bbR$. 
    Let $X$ be generic and condition on the event $|\supp(\wh\pi_S)|=k$ (assuming $k$ is such that $\bbP[|\supp(\wh\pi_S)|=k]>0$).
    Then with $\{y_1,\dots,y_k\}=\supp(\wh\pi)$:
    \begin{enumerate}[label=(\Roman*)]
    \item 
    \label{it:abs-cont-warmup}
    The conditional law of $\wh\pi_S$ on $\Pi_k(S)$ is absolutely continuous. 
    \item
    \label{it:no-coincidence-warmup}
    The restriction of $D_{\wh\pi,X}$ to $S$ almost surely has exactly $k$ global maxima $y_1,\dots,y_k$.
    \end{enumerate}
\end{theorem}

\begin{theorem}
\label{thm:static-landscape-consequences}
For any $1\leq k\leq n$, if $|\supp(\wh\pi_S)|=k$ and Theorem~\ref{thm:main-warmup}\ref{it:no-coincidence-warmup} holds:
\begin{enumerate}[label=(\Alph*)]
    \item
    \label{it:static-local-strong-convexity}
    There is $c>0$ and an open neighborhood of $\wh\pi_S$ in $\Pi_k(S)$ on which $\ell_X$ is locally $c$-strongly concave.
    \item
    \label{it:static-local-smoothness}
    There is an open neighborhood of $X$ in $\mathbb R^n$, such that $\wh\pi_S(\hat X)\in\Pi_k(S)$ for all $\hat X$ in this neighborhood.
    Moreover, $\wh\pi_S$ is a smooth function on this neighborhood.
\end{enumerate}
\end{theorem}

\begin{theorem}
\label{thm:certify-finite-S}
    Assume $n,L$ satisfy $n\geq CL^4$ for an absolute constant $C$, and let $S\subseteq \mathbb R$ be deterministic and finite with $S\cap [-3L,3L]\neq \emptyset$.
    Let $X=(x_1,\dots,x_n)\in [-L,L]^n$ be generic.
    Then there exists a deterministic $C(L)\eps^{-O(1)}$ time algorithm which computes $\breve\pi_{S,\eps}$ with the following properties.
    Almost surely, for sufficiently small $\eps$, $\breve\pi_{S,\eps}$ is a certifiable $C(n,k,S,X)\eps$-approximation to $\wh\pi_S$ in both $\bbW_1$ and $d_{\Pi_k(S)}$ distance.
    Furthermore, $\supp(\breve\pi_{S,\eps})=\supp(\wh\pi_S)$ almost surely holds certifiably for small enough $\eps$.
    
    Finally, almost surely for small enough $\eps$, $\breve\pi_{S,\eps}$ is a Shub--Smale approximate static support NPMLE: Newton's method within $\Pi_k(S)$ started from $\breve\pi_{S,\eps}$ certifiably converges to $\wh\pi_S$ at the rate in Definition~\ref{def:shub-smale-approx-npmle}.
\end{theorem}

\subsection{Other Related Work}

Gaussian mixture models have been studied since the pioneering work of Pearson \cite{pearson1894contributions}, which proposed that the ratio of forehead width to body length of crabs might follow such a distribution.
Much work has focused on statistical recovery of such mixtures. In the $1$-dimensional Gaussian location model
we consider, optimal convergence rates for recovering the mixing distribution were obtained in the case of $k$ components by
\cite{wu2020optimal}, via an extension of the method of moments. 
See \cite{doss2020optimal} for extensions to higher dimensions.
The theoretical computer science literature has also studied Gaussian mixture models since \cite{dasgupta1999learning}. 
In the special case $k=2$, \cite{hardt2015tight} gave sharp bounds for parameter recovery, and showed an exponential-in-$k$ sample complexity lower bound.
This line of work led to accurate polynomial time estimators for the 
underlying parameters of the Gaussian mixture even in high-dimensions, under minimal assumptions that ensure statistical identifiability.
These algorithms succeed even if the components may have different covariances \cite{kalai2010efficiently,moitra2010settling,hardt2015tight} and more recently if a small fraction of the data is adversarially corrupted \cite{kane2021robust,liu2021settling,bakshi2022robustly,liu2022learning} thanks to the sum-of-squares framework.

The NPMLE was introduced for general mixture models in \cite{kiefer1956consistency}, where its consistency was shown in quite general settings including the one we study.
\cite{genovese2000rates,ghosal2001entropies,zhang2009generalized,jiang2009general} upper-bounded its rate of
convergence for density-estimation, for IID data generated from a mixture of unit variance
Gaussians. See also \cite{dicker2016high,saha2020nonparametric} for higher-dimensional extensions.
We emphasize that we always metrize convergence based on $\wh\pi$ itself, rather than the convolution $\wh\pi*\cN(0,1)$ which is done in some of these works. (This yields e.g.\ the \emph{smoothed} Wasserstein distance, which gives the same topology on probability measures but can be exponentially smaller.)
However we emphasize that by contrast to the algorithms mentioned above, the NPMLE's behavior can be fruitfully analyzed without assuming that the underlying data actually comes from a Gaussian mixture.

To estimate a $k$-component Gaussian mixture for small $k$, a standard approach is the expectation-maximization (EM) algorithm \cite{dempster1977maximum}.
However a key advantage of the NPMLE is that it solves a \textbf{concave} maximization problem.
To take full advantage of this, one may discretize space as in \cite{koenker2014convex,feng2018approximate,soloff2024multivariate}. Namely one fixes an $\eps$-net $Z_{\eps}$ and optimizes \eqref{eq:NPMLE-def} subject to the additional constraint $\supp(\wh\pi)\subseteq Z_{\eps}$, which is now a finite-dimensional concave problem. 
It is not hard to show (see Section~\ref{sec:cert}) that the resulting estimate $\wh\pi_{\eps}$ converges to the true NPMLE $\wh\pi$ as $\eps\to 0$.
Our interest will be in \emph{certifying} that a candidate $\wh\pi_{\eps}$, computed in essentially arbitrary manner, approximates $\wh\pi$ to some explicit accuracy and additionally satisfies $|\supp(\wh\pi_{\eps})|=|\supp(\wh\pi)|$.

The local convergence rate of the EM algorithm has long been of interest, and was studied in e.g. \cite{redner1984mixture,meng1994global,xu1996convergence,ma2000asymptotic,ma2005correct,ma2005asymptotic}.
More recently, \cite{balakrishnan2017statistical,zhao2020statistical} established high-probability linear local convergence rates for well-separated Gaussian mixtures in general dimension, via perturbative analysis around the population dynamics.
Theorem~\ref{thm:landscape-consequences}\ref{it:EM-algorithm} shows almost sure local linear convergence for generic datasets, without even requiring the existence of an underlying mixture model generating the data.
On the other hand, it is currently limited to dimension $1$ and does not give quantitative bounds on $\eta$.
Among the vast literature in this direction, we also mention a few recent works \cite{daskalakis2017ten,wu2022randomly,weinberger2022algorithm} showing rapid \emph{global} convergence of the EM algorithm for $2$-component Gaussian mixtures.

We also mention the recent work \cite{wei2022convergence} which considers rather general mixture models and uses harmonic analysis tools related to those we employ (see Section 7 therein). However their work focuses on asymptotic posterior contraction rates and uses these tools differently.
Also recently, \cite{mukherjee2023mean,fan2023gradient} investigated properties of and algorithms for the NPMLE in high-dimensional regression, which is different from the present setting.
Theorem 3.3 of the latter also provides a quantitative convergence rate for computing the NPMLE (in the same setting as the present work) via gradient flow, as measured by the log-likelihood objective \eqref{eq:NPMLE-def}.
As outlined in Subsection~\ref{subsec:naive}, this does not yield any convergence rate for the NPMLE itself, while our Theorems~\ref{thm:certify-intro}, \ref{thm:certify-intro-approx} give asymptotic convergence rates in Wasserstein distance.

\subsection{Notations and Spaces of Measures}
\label{subsec:notation}

Denote by $\cP(X)$ the space of probability measures on a space $X$, and by $\rho_{\nu}(x)$ the density function of an absolutely continuous distribution $\nu$ at $x$.
Let $d_{\cH}$ denote the Hausdorff distance between compact sets in $\bbR$:
\begin{equation}
\label{eq:hausdorff-dist-def}
    d_{\cH}(K_1,K_2)
    =
    \max\big(
    \max_{x_1\in K_1} d(x_1,K_2),
    \max_{x_2\in K_2} d(K_1,x_2)
    \big).
\end{equation}
For $\pi,\pi'\in \cP(\bbR)$, denote by $\bbW_1(\pi,\pi')$ the usual Wasserstein-$1$ distance
\[
    \bbW_1(\pi,\pi') 
    =
    \inf_{\Gamma \in \cC(\pi,\pi')}
    \bbE^{\Gamma}|y-y'|.
\]
    Here the infimum is over all couplings $(y,y') \sim \Gamma$ with marginals $y \sim \pi$ and $y' \sim \pi'$. 
Define $\Pi_k$ to be $2k-1$ dimensional space of $k$-atomic $\pi$. We use the parametrization 
\begin{equation}
\label{eq:Pik-def}
    \Pi_k
    \triangleq
    \lt\{
    (p_1,\dots,p_k,y_1,\dots,y_k)\in\bbR^{2k}~:~
    p_j\geq 0,~\sum_{j=1}^k p_j=1
    ,~y_1<y_2\dots<y_k
    \rt\},
\end{equation}
thus identifying $(p_1,\dots,p_k,y_1,\dots,y_k)$ with the probability distribution $\sum_{j=1}^k p_j\delta_{y_j}$.
As previously mentioned, we let $\Pi_{k,\eps}$ consist of all $\pi\in\Pi_k$ with $\min_i p_i\geq\eps$ and $\min_{i<j} |y_i-y_j|\geq \eps$.
Similarly $\Pi_k^L$ consists of all $\pi\in \Pi_k$ supported in $[-L,L]$, and $\Pi_{k,\eps}^L=\Pi_{k,\eps}\cap\Pi_k^L$.
Note that \eqref{eq:Pik-def} gives $\Pi_k$ the structure of a smooth $2k-1$ dimensional manifold.
We metrize $\Pi_k$ and its subsets via parameter distance:
\begin{equation}
    \label{eq:Pik-distance}
    d_{\Pi_k}(\pi,\pi')^2
    =
    \sum_{j=1}^k 
    \big(|p_j-p_j'|^2+|y_j-y_j'|^2\big)
\end{equation}
where $\pi,\pi'$ correspond respectively to $(p_1,\dots,p_k,y_1,\dots,y_k),(p_1',\dots,p_k',y_1',\dots,y_k')$.
Note that as a subset of $\bbR^{2k}$, (the closure of) $\Pi_k$ is a codimension $1$ convex polytope, hence has a natural $(2k-1)$-dimensional Lebesgue probability measure. 
The next proposition states that $\Pi_k^L$ parameter distance controls $\bbW_1$ distance.
There is no bound in the opposite direction because far away atoms with small probability are significant only for the former.

\begin{proposition}
\label{prop:W1-Pik-bound}
    For any $k$-atomic $\pi,\pi'\in \Pi_k^L$,
    \[
    \bbW_1(\pi,\pi')\leq (L^{3/2}+1) \, d_{\Pi_k}(\pi,\pi').
    \]
\end{proposition}

\subsection{Preliminary Smoothness Estimates}

For $j\geq 0$ let $H_j(t)=e^{t^2/2} \lt(\frac{\de}{\de t}\rt)^j e^{-t^2/2}$ be the $j$-th Hermite polynomial. The following proposition is immediate from the formula~\eqref{eq:D-formula}.

\begin{proposition}
    For $j\geq 0$, the $j$-th derivative of $D_{\pi,X}$ is given by
\begin{equation}
\label{eq:D-deriv-hermite}
    D_{\pi,X}^{(j)}(y)
    =
    \frac{1}{n}\sum_{i=1}^n
    \frac{H_j(x_i-y) e^{-|x_i-y|^2/2}}{P_{\pi}(x_i)\sqrt{2\pi}}\,
    .
\end{equation}
\end{proposition}

We next show basic estimates on these functions.

\begin{lemma}
\label{lem:lipschitz-bounds}
    There exist universal constants $C>0$ and $(C_1,C_2,\dots)$ such that the following estimates hold.
    Suppose $\bbW_1(\pi,\pi')\leq \delta$ where $\pi,\pi'$ are supported in $[-L,L]$. Then for $x,z\in [-L,L]$:
    \begin{enumerate}
        \item 
        \label{it:P-LB}
        $C^{-1}e^{-2L^2}\leq P_{\pi}(x)\leq 1$.
        \item 
        \label{it:P-lip}
        $|P_{\pi}(x)-P_{\pi'}(x)|\leq C\delta$.
        \item 
        \label{it:P-invlip}
        $|P_{\pi}(x)^{-1}-P_{\pi'}(x)^{-1}|\leq Ce^{4L^2}\delta$. 
        \item 
        \label{it:l-lip}
        $|\ell_X(\pi)-\ell_X(\pi')|\leq Ce^{2L^2}\delta$.
        \item
        \label{it:T-lip}
        For $j\geq 0$ the $j$-th derivative of $T$ (recall \eqref{eq:D-formula}) with respect to $y$ satisfies
        \begin{align}
        \label{eq:T-bound}
        |T^{(j)}_{\pi,x}(z)|
        &\leq 
        C_j L^j e^{2L^2},
        \\
        \label{eq:T-lip}
        |T^{(j)}_{\pi,x}(z)-T^{(j)}_{\pi',x}(z)| 
        &\leq 
        C_j L^j e^{4L^2}\delta.
        \end{align}
        \item 
        \label{it:D-lip}
        Similarly for $D$:
        \begin{align}
        \label{eq:D-bound}
        |D^{(j)}_{\pi,X}(z)|
        &\leq 
        C_jL^j e^{2L^2},
        \\
        \label{eq:D-lip}
        |D^{(j)}_{\pi,X}(z)-D^{(j)}_{\pi',X}(z)| 
        &\leq 
        C_jL^j e^{4L^2}\delta.
        \end{align}
    \end{enumerate}
\end{lemma}

\begin{lemma}
\label{lem:derivative-bounds}
    Fix $\pi=\sum_{i=1}^k p_j\delta_{y_i}$ and $\pi'=\sum_{i=1}^k q_i\delta_{y_i}$, both in $\Pi_k^L$. Consider
    \[
    \ell(t)\triangleq \ell_X(\pi_t)=\ell_X((1-t)\pi+t\pi')
    \]
    If $\max_i |p_i-q_i|\leq \tau$, then 
    \[
    \sup_{0\leq t\leq 1}
    \lt|
    \frac{\de}{\de t}
    \ell(t)
    \rt|
    \leq 
    O(e^{2jL^2} \tau^j),\quad 0\leq j\leq 3.
    \]
    Further if $\pi_t([-10,10])\geq 1/10$, the improved upper bound $O(e^{0.51jL^2} \tau^j)$ holds.
\end{lemma}

The next lemma will be used to ensure that $\wh\pi$, or any reasonable approximation thereof, is supported in $[-L,L]$.
Lemmas~\ref{lem:derivative-bounds} and \ref{lem:support-within-[-L,L]} are proved in Appendix~\ref{app:lemma-proofs}.

\begin{lemma}
\label{lem:support-within-[-L,L]}
    If $X=(x_1,\dots,x_n)\in [-L,L]^n$, then for any $\pi\in\cP(\bbR)$, the function $D_{\pi,X}(\cdot)$ is strictly increasing on $(-\infty,L]$ and strictly decreasing on $[L,\infty)$.
\end{lemma}

\begin{proof}
    Each term in \eqref{eq:D-formula} obeys these monotonicity conditions.
\end{proof}

\section{Proof of Theorem~\ref{thm:main}}
\label{sec:main-proof}

In this section we prove Theorem~\ref{thm:main} on generic behavior of $\wh\pi$ and $D_{\wh\pi,X}$. The different parts are established by variations of the same core argument. For concreteness, we center our discussion on absolute continuity of $\wh\pi$.
This is natural to expect from naive dimension-counting: if $\wh\pi$ is $k$-atomic, then it varies over a $2k-1$ dimensional parameter space $\Pi_k$ and must satisfy the $2k$ equations \eqref{eq:stationarity}.
However one of these equations is redundant since one always has $\bbE^{y\sim\pi} D_{\pi,X}(y)=1$.
Thus intuitively, all $2k-1$ dimensions of $\Pi_k$ should be needed to solve this many equations, suggesting the law of $\wh\pi$ is generic.
Similar remarks apply in the static support case; here the conditions on $D_{\wh\pi,X}'$ disappear, corresponding to the reduced dimension of $\Pi_k(S)$.

Proving absolute continuity of the law of $\wh\pi$ amounts to upper-bounding small-ball probabilities near arbitrary $\pi_0\in\Pi_{k,\eps}$. We first reduce this to bounding the probability that the stationarity conditions \eqref{eq:stationarity} for $\pi_0$ hold approximately. 
Below, $B_{\delta}(\pi_0)$ denotes a ball in $d_{\Pi_k}$-distance around $\pi_0$. 
Similarly, let $B_{\delta,S}(\pi_0)$ be the $\delta$-neighborhood of $\pi_0\in\Pi_k(S)$ in the space $\Pi_k(S)$. Note that $B_{\delta,S}(\pi_0)\subseteq B_{\delta}(\pi_0)\cap \Pi_k(S)$.

\begin{proposition}
\label{prop:almost-solution}
    Fix $\pi_0=\sum_{i=1}^k p_i \delta_{y_i}\in\Pi_k([-L,L])$. Then for $\wh\pi=\sum_{i=1}^k \hat p_i \delta_{\hat y_i}\in B_{\delta}(\pi_0)$ to
    hold (where $\delta$-ball is defined with respect to $d_{\Pi_k}$), $\pi_0$ must approximately solve the system \eqref{eq:stationarity} in the sense that
\begin{equation}
\label{eq:optimality-conditions-approx}
    \begin{aligned}
    \max_{1\leq j\leq k} |D_{\pi_0,X}(y_j)-1|
    &\leq 
    e^{O(L^2)}\delta,
    \\
    \max_{1\leq j\leq k} |D'_{\pi_0,X}(y_j)|
    &\leq
    e^{O(L^2)}\delta
    \end{aligned}
\end{equation}
    Moreover to have $D_{\wh\pi,X}(y)=1$ for some $y$ with $|y-y_*|\leq
    \delta$, or $D''_{\wh\pi,X}(\hat y_j)=0$ for some $1\leq j\leq k$, we must respectively have
\begin{align}
    \label{eq:extra-support-condition}
    \max\lt(
    |D_{\pi_0,X}(y_*)-1|,
    |D_{\pi_0,X}'(y_*)|
    \rt)
    &\leq 
    e^{O(L^2)}\delta,
    \\
    \label{eq:unstable-optimum-condition}
    |D''_{\pi_0,X}(y_j)|
    &\leq 
    e^{O(L^2)}\delta.
\end{align}
Similarly if $\wh\pi_S=\sum_{i=1}^k \hat p_i\delta_{\hat y_i}\in B_{\delta,S}(\pi'_S)$ for some $\pi'_S$ supported in $S\subseteq L$, then:
\begin{equation}
\label{eq:optimality-conditions-approx-S}
    \max_{1\leq j\leq k} |D_{\pi'_S,X}(y_j)-1|
    \leq 
    e^{O(L^2)}\delta.
\end{equation}
And to have $D_{\wh\pi_S,X}(y)=1$ for some $y$ with $|y-y_*|\leq \delta$, we must have 
\begin{align}
    \label{eq:extra-support-condition-S}
    |D_{\pi'_S,X}(y_*)-1|
    \leq 
    e^{O(L^2)}\delta.
\end{align}
\end{proposition}

\begin{proof}
    Using \eqref{eq:D-lip} and Proposition~\ref{prop:W1-Pik-bound}, we obtain
    \[
    \|D_{\wh\pi,X}-D_{\pi_0,X}\|_{C^2([-L,L])}
    \leq
     e^{O(L^2)} \bbW_1(\wh\pi,\pi_0)
    \leq 
     e^{O(L^2)}
    d_{\Pi_k}(\wh\pi,\pi_0))
    .
    \]
    We now apply Proposition~\ref{prop:stationarity-conditions}.
    For the first estimate, we write
    \begin{align*}
    |D_{\pi_0,X}(y_j)-1|
    &\leq 
    |D_{\pi_0,X}(y_j)-D_{\pi_0,X}(\hat y_j)|
    +
    |D_{\pi_0,X}(\hat y_j)-D_{\wh\pi,X}(\hat y_j)|
    +
    |D_{\wh\pi,X}(\hat y_j)-1|
    \\
    &\leq 
    e^{O(L^2)}|y_j-\hat y_j|
    +
    e^{O(L^2)}
    d_{\Pi_k}(\wh\pi,\pi_0)
    +
    0.
    \end{align*}
    Similar arguments imply the other claims.
\end{proof}

Proposition~\ref{prop:almost-solution} effectively linearizes the stationarity conditions for $\wh\pi$. For $\pi_0$ fixed, the function $D_{\pi_0,X}$ and its derivatives are simply IID sums over $x_i\in X$.
In particular their law becomes smoother as $n$ increases; we make this precise using classical estimates from harmonic analysis.

\subsection{Harmonic Analysis and Non-Degeneracy of Exponential Curves}

\begin{definition}
\label{def:curve-loop}
    We call a function $\gamma\in C^{\infty}([0,1];\bbR^d)$ a \emph{smooth curve}, and $\gamma\in C^{\infty}(\bbS^1;\bbR^d)$ a \emph{smooth loop}. We say $\gamma$ is \textbf{non-degenerate} if for each $x$, the vectors 
    \[
    \big(
    \gamma'(x),
    \gamma''(x),
    \dots,
    \gamma^{(d)}(x)
    \big)
    \]
    form a basis for $\bbR^d$.
    By compactness, this is equivalent to the matrix $M(x)$ with these vectors as columns having determinant bounded away from zero, uniformly over $x\in [0,1]$ or $x\in \bbS^1$.
    Let $\mu_{\gamma}$ denote the pushforward of the uniform measure on $[0,1]$ by $\gamma$.
    Given a continuous function $f$ defined on the range of $\gamma$, let $\mu_{\gamma,f}$ be the signed measure with Radon--Nikodym derivative $\de \mu_{\gamma,f}(x)/\de \mu_{\gamma}(x)=f(x)$.
\end{definition}

The next key estimate follows from \cite[Page 334]{stein1993harmonic}, see also \cite{marshall1988decay,brandolini2007average}.

\begin{proposition}
\label{prop:fourier}
    Let $\gamma$ be a non-degenerate smooth loop, and $f$ a $C^1$ function on its range. 
    Then the Fourier transform $\hat \mu$ of $\mu$ satisfies:
    \[
    |\hat \mu(\omega)|
    \leq C(\gamma) (1+\|\omega\|)^{-1/d} 
    \cdot  
    \|f\|_{C^1},\quad\forall \omega\in\bbR^d.
    \]
\end{proposition}

We next deduce that for non-degenerate $\gamma$, at most $d^2+1$ self-convolutions suffice for a bounded density.

\begin{corollary}
\label{cor:fourier}
    Let $\gamma$ be a non-degenerate smooth curve or loop, and $f$ a non-negative bounded function on its range.
    Then with $(\cdot)^{* j}$ denoting $j$-fold self-convolution, the probability measure $\mu^{* (d^2+1)}$ has bounded density on $\bbR^d$.
\end{corollary}

\begin{proof}
    First, if $\gamma$ is a smooth non-degenerate curve then it easily extends to a smooth non-degenerate loop $\wt\gamma$. In either case, since $\|f\|_{L^{\infty}}<\infty$ there is a constant (in particular $C^1$) function on $\gamma$ or $\wt\gamma$ which is point-wise larger than $f$.
    Therefore by domination, it suffices to consider the case that $\gamma$ is a loop.
    
    For $\gamma$ a smooth non-degenerate loop, Proposition~\ref{prop:fourier} implies $\mu^{* (d^2+1)}$ has integrable Fourier transform, hence bounded density by Fourier inversion, completing the proof.
\end{proof}

\begin{remark}
    When $\gamma$ is a loop and $f$ is $C^1$ in Corollary~\ref{cor:fourier}, $\mu^{*(d^2+1)}$ in fact has a uniformly continuous density (being the inverse Fourier transform of an integrable function).
    However we can only use Proposition~\ref{prop:almost-solution} to \emph{upper bound} the law
    of $\wh\pi$, so this additional information does not improve our final results. 
\end{remark}

\begin{remark}
\label{rem:fourier-analysis-tight}
    The appearance of $\Theta(d^2)$ convolutions in Corollary~\ref{cor:fourier} is easily seen to be sharp for any smooth $\gamma$, degenerate or not. 
    Indeed for small $\eps>0$, the curve $\gamma([0,\eps])$ is contained inside a rectangular box with $O(\eps^{d(d+1)/2})$ volume, spanned by vectors $O(\eps^j)\gamma^{(j)}(x)$ for $1\leq j\leq d$. 
    (Here implicit constant may depend on both $\gamma$ and $k$ but not $\eps$.)
    Thus $\mu^{*k}$ assigns measure $\Omega(\eps^k)$ to the $k$-dilate of this box, which still has $O(\eps^{d(d+1)/2})$ volume: each independent summand lands in the box with probability $\Omega(\eps)$.
    Thus for any smooth $\gamma$, we must have $k\geq d(d+1)/2$ for $\mu^{*k}$ to have a bounded density on $\bbR^d$.
\end{remark}

The following lemma will be used to verify non-degeneracy for the curves relevant to $\wh\pi$.

\begin{lemma}
\label{lem:nondegen-vandermonde}
    Let $P:\bbR\to (0,\infty)$ be a smooth, strictly positive function. Let $d_1,\dots,d_k\geq 0$ be non-negative integers and let $a_{i,j}$ be real constants for $0\leq i<d_j$ which are not all zero. Define the function
    \[
    F(x)=
    \frac{
    \sum_{j=1}^k \sum_{i=0}^{d_j-1}
    a_{i,j} x^i e^{x y_j}
    }
    {P(x)}
    .
    \]
    Then with $D=\sum_{j=1}^k d_j$, $F$ can have at most $D-1$ on $\mathbb R$, counting multiplicity. (Here $x$ is a root of multiplicity $r$ if $F(x)=F'(x)\dots=F^{(r-1)}(x)=0$.)
\end{lemma}

\begin{proof}
    We can set $P(x)=1$ since it does not affect the multiplicity of a root.
    It remains to show $F(x)=\sum_{j=1}^k \sum_{i=0}^{d_j-1}
    a_{i,j} x^i e^{x y_j}$ cannot have $D$ roots. Indeed, this function satisfies a linear ODE with degree $d$ characteristic polynomial $z\mapsto\prod_{j=1}^k (z-y_j)^{d_j}$.
    We consider the corresponding multi-set of differential operators $\phi_j(f)=f'-y_j f$.
    Further, if $F$ has a non-zero term, then a subset of at most $D-1$ of these operators can be applied to reach a function $f(x)=A e^{xy_i}$ for some $A\neq 0$ and $i$.
    Indeed, $\phi_j$ turns $xe^{xy_j}$ into $e^{xy_j}$, kills the $e^{xy_j}$ term, and scales each other term by a non-zero constant.
    Crucially, this resulting $f$ has no roots.
    However by Rolle's theorem, each of these differential operators reduces the total number of roots of a function by at most $1$.
    It follows that $F$ has fewer than $D$ roots.
\end{proof}

A compactness argument immediately gives the following.

\begin{corollary}
In the setting of Lemma~\ref{lem:nondegen-vandermonde}, if $\min_{j,j'}|y_j-y_{j'}|\geq \eps$, there exists $c(L,k,\eps)$ such that
    \begin{equation}
    \label{eq:F-vanish-D-approx}
        \max_{0\leq d\leq D-1} |F^{(d)}(x)|\geq c\cdot \max_{i,j}|a_{i,j}|,
        \quad
        \forall\, x\in [-L,L]
        .
    \end{equation}
\end{corollary}

\subsection{Distributional Regularity of NPMLE}

Proposition~\ref{prop:almost-solution} implies that for any $\pi_0\in \Pi_k$, having $\wh\pi\in B_{\delta}(\pi_0)$ requires the vector 
\begin{equation}
\label{eq:V-def}
    V_{\pi_0}(X)
    \triangleq
    \lt(
    D_{\pi_0,X}(y_1),\dots, D_{\pi_0,X}(y_k)
    ,~
    D'_{\pi_0,X}(y_1),\dots, D'_{\pi_0,X}(y_k)
    \rt)
\end{equation}
to be within distance $C(n,L)\delta$ of the half-ones vector $(1,\dots,1,0,\dots,0)$.
Note that the $p_j$-weighted average of the first $k$ coordinates of this vector always equals $0$ by \eqref{eq:D-sum-formula}, so the image of $\gamma_{\pi_0}$ lies in a $2k-1$ dimensional affine subspace $U_{\pi_0}\subseteq\bbR^{2k}$.
In the $S$-restricted case, we denote the relevant curve by $\gamma_{\pi_0}^{\circ}$:
\begin{equation}
\label{eq:V-def-warmup}
    V_{\pi_0}^{\circ}(X)
    \triangleq
    \lt(
    D_{\pi_0,X}(y_1),\dots, D_{\pi_0,X}(y_k)
    \rt)\in\bbR^k
\end{equation}
obeys the same constraint, hence lies in a $k-1$ dimensional affine subspace $U_{\pi_0}^{\circ}$.

\begin{lemma}
\label{lem:bounded-density}
    Suppose $x_1,\dots,x_n$ are IID from a density supported in $[-L,L]$ and bounded above pointwise by $L'$.
    For $n\geq k^2$ and $\pi_0\in\Pi_{k,\eps}$, the random vector $V_{\pi_0}^{\circ}(X)\in \bbR^{k}$ has compactly supported distribution on $U_{\pi_0}^{\circ}$ with density uniformly at most $C(n,\eps,k,L,L')$. 
    When $n\geq 4k^2$, the same holds for $V_{\pi_0}(X)\in \bbR^{2k}$ and $U_{\pi_0}$.
\end{lemma}

\begin{proof}
    We can write $V_{\pi_0}^{\circ}(X)$ as the convolution
    $
    V_{\pi_0}^{\circ}(X)
    =
    \frac{1}{n}
    \sum_{i=1}^n
    \gamma_{\pi_0}^{\circ}(x_i)
    $
    where we define
    \begin{equation}
    \label{eq:gamma-pi0-warmup}
    \gamma_{\pi_0}^{\circ}(x)
    \triangleq
    \lt(
    \frac{e^{-|x-y_1|^2/2}}{P_{\pi_0}(x)\sqrt{2\pi}},
    \dots,
    \frac{e^{-|x-y_k|^2/2}}{P_{\pi_0}(x)\sqrt{2\pi}}
    \rt).
    \end{equation}
    Further, $\gamma_{\pi_0}^{\circ}(x)$ is an ($x$-independent) invertible linear transformation of 
    \begin{equation}
    \label{eq:wt-gamma-pi0}
    \wt\gamma_{\pi_0}^{\circ}(x)
    \triangleq
    \lt(
    \frac{e^{xy_1}}{P_{\pi_0}(x)e^{x^2/2}},
    \dots,
    \frac{e^{xy_k}}{P_{\pi_0}(x)e^{x^2/2}}
    \rt)
    \end{equation}

    We will apply Corollary~\ref{cor:fourier} with $\gamma=\wt\gamma_{\pi_0}^{\circ}$. Note that the relevant dimension is $d=k-1$, so $n\geq k^2\geq d^2+1$ as required.
    It suffices to check that $\wt\gamma_{\pi_0}^{\circ}$ is non-degenerate within the subspace $U_{\pi_0}^{\circ}$. 
    We show the stronger statement that within $\bbR^{k}$,
    \[
    \inf_{x\in [-L,L]}
    \det\Big(
    \wt\gamma_{\pi_0}^{\circ}(x),
    (\wt\gamma_{\pi_0}^{\circ})'(x),
    \dots,
    (\wt\gamma_{\pi_0}^{\circ})^{(k-1)}(x)
    \Big)
    >
    0.
    \]
    This follows from Lemma~\ref{lem:nondegen-vandermonde} with $P(x)=P_{\pi_0}(x)e^{x^2/2}$ and $d_1=\dots=d_k=1$.
    Indeed, letting $\vec a=(a_{0,1},\dots,a_{0,k})\in\bbR^{k}$, and with $F$ as in Lemma~\ref{lem:nondegen-vandermonde}, we have 
    \[
    F^{(j)}(x)=\la (\gamma_{\pi_0}^{\circ})^{(j)}(x),\va\ra,\quad \forall~0\leq j\leq k-1.
    \]
    This completes the proof.
    The case of $V_{\pi_0}$ is similar with
    \begin{align*}
    \gamma_{\pi_0}(x)
    &\triangleq
    \lt(
    \frac{e^{-|x-y_1|^2/2}}{P_{\pi_0}(x)\sqrt{2\pi}},
    \dots,
    \frac{e^{-|x-y_k|^2/2}}{P_{\pi_0}(x)\sqrt{2\pi}},
    \frac{(x-y_1)e^{-|x-y_1|^2/2}}{P_{\pi_0}(x)\sqrt{2\pi}},
    \dots,
    \frac{(x-y_k)e^{-|x-y_k|^2/2}}{P_{\pi_0}(x)\sqrt{2\pi}}
    \rt),
    \\
    \wt\gamma_{\pi_0}(x)
    &\triangleq
    \lt(
    \frac{e^{xy_1}}{P_{\pi_0}(x)e^{x^2/2}},
    \dots,
    \frac{e^{xy_k}}{P_{\pi_0}(x)e^{x^2/2}},
    \frac{xe^{xy_1}}{P_{\pi_0}(x)e^{x^2/2}},
    \dots,
    \frac{xe^{x_k}}{P_{\pi_0}(x)e^{x^2/2}}
    \rt)
    \end{align*}
    and $d_1=\dots=d_k=2$ (taking $\vec a=(a_{0,1},\dots,a_{0,k},a_{1,1},\dots,a_{1,k})\in\bbR^{2k}$).
\end{proof}

\begin{proof}[Proof of Theorem~\ref{thm:main}, Part~\ref{it:abs-cont}]
    
    Lemma~\ref{lem:bounded-density} states that uniformly in $\pi_0\in \Pi_{k,\eps}$ for any $\eps$ fixed, $V_{\pi_0}(X)$ has bounded density on the $2k-1$ dimensional subspace $U_{\pi_0}$. 
    Proposition~\ref{prop:almost-solution} 
    thus implies that for any $\delta>0$,
    \[
    \bbP[\wh\pi\in B_{\delta}(\pi_0)]\leq C(n,\eps,k,L,L')\delta^{2k-1}.
    \]
    This implies $\wh\pi$ has locally bounded density on each $\Pi_k$.
    Since $\Pi_k=\cup_{\eps>0}\Pi_{k,\eps}$ and using similar countable exhaustion for a general absolutely continuous data distribution, we also deduce absolute continuity in the general case.
\end{proof}

\begin{proof}[Proof of Theorem~\ref{thm:main-warmup}, Part~\ref{it:abs-cont-warmup}]
    As in the previous proof, by countable additivity it suffices to consider data distributions with compactly supported bounded density.
    Lemma~\ref{lem:bounded-density} states that $V_{\pi_0}^{\circ}(X)$ has bounded density on the $k-1$ dimensional subspace $U_{\pi_0}^{\circ}$. 
    Proposition~\ref{prop:almost-solution} thus implies that for $\pi_0\in \Pi_{k,\eps}(S)$ and any $\delta>0$,
    \[
    \bbP[\wh\pi_S\in B_{\delta,S}(\pi_0)]\leq C(n,L,\eps,k)\delta^{k-1}.
    \]
    Finally this implies absolute continuity of $\wh\pi_S$ within $\Pi_k(S)$, since any $\pi_0\in
    \Pi_k$ is in $\Pi_{k,\eps}$ for some $\eps>0$.
    This completes the proof.
\end{proof}

\subsection{Generic Behavior of $D_{\wh\pi,X}$}

Distributional regularity of $\wh\pi$ does not have direct consequences for $D_{\wh\pi,X}$ since the latter depends on both $\wh\pi$ and $X$. 
We prove genericity of $D_{\wh\pi,X}$ using the same approach as above, thus establishing the remaining statements of Theorems~\ref{thm:main} and \ref{thm:main-warmup}. The idea is that any non-generic behavior can be encoded into an extra constraint dimension for the curve $\gamma$ or $\gamma^{\circ}$.
Hence the chance for such behavior to hold approximately becomes $O(\delta^{2k})$ or $O(\delta^k)$ instead of $O(\delta^{2k-1})$ or $O(\delta^{k-1})$ respectively. 
These probabilities are now smaller than the inverse $\delta$-radius covering numbers for $\Pi_k$ and $\Pi_k(S)$.
Hence summing over such covers shows a probability upper bound $O(\delta)$; taking $\delta$ small, we deduce that such non-generic behavior occurs with probability $0$.

\begin{proof}[Proof of Theorem~\ref{thm:main} part \ref{it:no-coincidence-1} and \ref{it:no-coincidence-2}]
    We first show part \ref{it:no-coincidence-1}, namely that $D_{\wh\pi,X}$ has exactly $k$ global maxima almost surely, namely $\supp(\wh\pi)$.
    We fix an arbitrary $L>0$ (for use in e.g. Proposition~\ref{prop:almost-solution}) and show that the probability for $L$-bounded generic data to violate the claims is $0$, which suffices by countable additivity.
    
    We \textbf{claim} that for $\wh\pi\in \Pi_{k,\eps}$, and any $y_*\in [-L,L]$ with $\min_{1\leq j\leq k}|y_*-y_j|\geq 2\eps$, there is $C(n,L,\eps,k)$ such that for $\delta\leq \delta_*(n,L,\eps,k)$ small enough, 
    the probability that both \eqref{eq:optimality-conditions-approx} and \eqref{eq:extra-support-condition} hold is at most $C\delta^{2k+1}$. 
    Explicitly for any fixed $\pi_0$ and $y_*$:
    \begin{equation}
    \label{eq:claim-main-proof}
    \begin{aligned}
    &\bbP[\wh\pi \in B_{\delta}(\pi_0), \exists y\in [y_*-\delta,y_*+\delta], D_{\wh \pi}(y)=1] 
    \\
    &\le 
    \bbP[\wh\pi \in B_{\delta}(\pi_0)
    \text{ and }
     |D_{\pi_0}(y^*) -1|
    \le C'\delta 
    \text{ and }
    |D_{\pi_0}'(y^*)|\leq C'\delta 
    ] 
    \le 
    C\delta^{2k+1}.
    \end{aligned}
    \end{equation}
    (Here $C'\leq e^{O(L^2)}$.)
    Note that $\Pi_k$ has dimension $2k-1$ and it suffices to take $3L/\delta$ choices of $y_*$ to cover all possible extraneous minimizers. There are hence $(C/\delta)^{2k}$ elements of a $\delta$-net over pairs $(\pi_0,y_*)$. So using the claim and a union bound, the probability for $\wh\pi$ to be $(\eps,k)$-non-degenerate and for $D_{\wh\pi,X}$ to have an extraneous minimizer within distance $3\eps$ of $\supp(\wh\pi)$ is $O(\delta)$ for small enough $\delta$, hence zero. This suffices since $\eps$ was also arbitrary.

    To show the claim \eqref{eq:claim-main-proof} above, we proceed as in Lemma~\ref{lem:bounded-density} using a slightly augmented version of $V_{\pi_0}$ given by
    \begin{equation}
    \label{eq:V-augment-1}
    \big(V_{\pi_0}(X),D_{\pi_0,X}(y_*),D_{\pi_0,X}'(y_*)\big)\in\bbR^{2k+2}.
    \end{equation}
    The claim follows by using Lemma~\ref{lem:nondegen-vandermonde}, with $d_1=\dots=d_{k+1}=2$ and $y_*=y_{k+1}$, as in the proof of Lemma~\ref{lem:bounded-density}. Here we use the assumption $n>(2k+2)^2$. This finishes the proof of the first assertion on global maxima.

    Part~\ref{it:no-coincidence-2} is handled similarly. We now need to bound the probability that \eqref{eq:optimality-conditions-approx} and \eqref{eq:unstable-optimum-condition} both hold, and we will show it is at most $C(n,L,\eps,k)\delta^{2k}$. This suffices by a similar union bound since the relevant $\delta$-nets still have cardinality $(C/\delta)^{2k-1}$. 
    This time we consider for some fixed $1\leq j\leq k$:
    \begin{equation}
    \label{eq:V-augment-2}
    \big(V_{\pi_0}(X),D_{\pi_0,X}''(y_j)\big)\in\bbR^{2k+1}.
    \end{equation}
    Applying Lemma~\ref{lem:nondegen-vandermonde} with $d_j=3$ and $d_i=2$ for $i\neq j$ completes the proof in this case.
\end{proof}

\begin{proof}[Proof of Theorem~\ref{thm:main-warmup}\ref{it:no-coincidence-warmup}]
The proof is similar to the above.
Again we fix arbitrary $L$ and restrict attention to the event $S\cup X\subseteq [-L,L]$.
Consider $\eps>0$ such that $\wh\pi_S\in \in\Pi_{k,\eps}(S)$, and any $y_*\in S\backslash \supp(\wh\pi_S)$. Without loss of generality, we may choose $2\eps$ smaller than the minimum distance between distinct points of $S$, so $d(y_*,\supp(\wh\pi_S))\geq 2\eps$.
Then analogously to \eqref{eq:claim-main-proof-S}, we have
\begin{equation}
    \label{eq:claim-main-proof-S}
    \begin{aligned}
    &\bbP[\wh\pi_S \in B_{\delta,S}(\pi'_S), \exists y\in [y_*-\delta,y_*+\delta], D_{\wh \pi_S}(y)=1] 
    \\
    &\le 
    \bbP[\wh\pi \in B_{\delta,S}(\pi'_S)
    \text{ and }
     |D_{\pi'_S}(y^*)-1|
    \le C'\delta 
    ] 
    \le 
    C\delta^{k}.
    \end{aligned}
    \end{equation}
In this case, the claim follows by considering 
\[
    \big(V_{\pi'_S}^{\circ},D_{\pi'_S,X}(y_*)\big)\in\bbR^{k+1}
\]
and applying Lemma~\ref{lem:nondegen-vandermonde}, with $d_1=\dots=d_{k+1}=1$ and $y_*=y_{k+1}$.
This time $\Pi_k(S)$ has dimension $k-1$, and now $y_*$ ranges over a finite set, so there are $(C/\delta)^{k-1}$ elements of a $\delta$-net over pairs $(\pi'_S,y_*)$.
Hence the probability for $D_{\wh\pi_S}$ to have an extraneous minimizer $y_*\in S\backslash \supp(\wh\pi_S)$ is $O(\delta)$ for small enough $\delta$, hence zero.
This completes the proof.
\end{proof}

\section{Certification}
\label{sec:cert}

In this section we give algorithms to certifiably compute $\wh\pi$, for both $\eps$-approximation in $\bbW_1$ and exact support size. 
Theorem~\ref{thm:main} will be used to show that for generic data, they eventually succeed almost surely.

\subsection{Certification of Wasserstein Approximations}
\label{subsec:W1-cert}

In Proposition~\ref{prop:cert} we show how to certify a putative candidate $\bbW_1$-approximation to $\wh\pi$, denoted by $\wt \pi_\eps$. 
The conditions on $\wt \pi_\eps$ will be satisfied by (approximately) maximizing $\ell_X$ over $\cP(Z_\eps)$ and merging adjacent atoms.

\begin{proposition}
\label{prop:cert}
    Suppose there exists
    $\wt\pi_{\eps}=\sum_{j=1}^k p_j\delta_{y_j}
    \in\cP(Z_{\eps})$, which satisfies for some $c_1,c_2,\delta>0$ the following properties.
    \begin{enumerate}
        \item 
        \label{it:no-more-zeros}
        For $y$ such that $d(y,\supp(\wt\pi_{\eps}))\geq c_1$, we have 
        $D_{\wt\pi_{\eps},X}(y)\leq 1-c_2$.
        \item 
        \label{it:delta-LB}
        $\max\limits_{y\in\bbR} D_{\wt\pi_{\eps},X}(y)\leq 1+\delta$.
        \item As a function on $\cP(\supp(\wt\pi_{\eps}))$, for some $\lambda\in (0,1)$ the empirical loss $\ell_X$ has Hessian
        \begin{equation}
        \label{eq:loss-hessian-condition}
        \nabla^2 \ell_X(\wt\pi_{\eps})\preceq -\lambda I_k.
         \end{equation}
        (We treat \eqref{eq:loss-hessian-condition} as vacuously true in the case $k=1$.)
        \item 
        \label{it:pi-reasonable} Let $\eta = c_1 + L\delta/c_2$. Either $\eta \leq \frac{\lambda^3}{k^3 e^{14L^2}}$ or 
        $\wt\pi_{\eps}([-10,10])\geq 1/10$ and 
        \begin{equation}
        \label{eq:eta-bound}
        \eta\leq \frac{\lambda^3}{k^3 e^{5.1L^2}}.
        \end{equation}
    \end{enumerate}
    If the above properties hold, then we must have:
    \[
    \bbW_1(\wt\pi_{\eps},\wh\pi)\leq
    O\lt(L e^{L^2} \sqrt{\eta k/\lambda}\rt)
    .
    \]
\end{proposition}

The intuition is that conditions 1 and 2 imply $\wh \pi$ can be approximated by a
distribution (denoted $\wt \pi_*$ below) supported on $\supp(\wt \pi_\eps)$ and achieving a
similar value of likelihood, while
condition~\eqref{eq:loss-hessian-condition} implies any distribution on $\supp(\wt \pi_\eps)$
achieving high likelihood must be close to $\wt\pi_\eps$ itself. A detailed proof is given in Appendix~\ref{app:lemma-proofs}.

Several conditions in Proposition~\ref{prop:cert} will be verified using Theorem~\ref{thm:main}. However one condition follows easily from uniqueness of $\wh\pi$.

\begin{proposition}
\label{prop:lambda-positive} For any sample $X=(x_1,\dots,x_n)\in\bbR^n$ there exists $\lambda=\lambda_X>0$ such that 
    condition~\eqref{eq:loss-hessian-condition} holds at $\wh\pi$. (Recall we take \eqref{eq:loss-hessian-condition} to be vacuous when $k=1$).
\end{proposition}

\begin{proof}
    Since $\ell_X$ is strongly convex in the vector $P_{\xi}(X)=\big(P_{\pi}(x_1),\dots,P_{\pi}(x_n)\big)$, if the claim did not hold then the directional derivative of $V$ would be zero along some line segment. However $V$ depends linearly on the weights $\vec p$ so we conclude that $\wh\pi$ is not unique, contradicting Proposition~\ref{prop:stationarity-conditions}.
\end{proof}

Next we turn to the support size of $\wh\pi$. We note that one direction is implied by a
Wasserstein approximation, as verified in the next result.

\begin{proposition}
\label{prop:cert_lb}
Suppose $\nu = \sum_{j=1}^k p_j \delta_{y_j}$. For each $j$ let $d_j$ be the minimum distance from $y_j$ to a nearest
(other) atom and let $\Delta = \Delta(\nu) \triangleq \min_j (p_j d_j)>0$. Then we have
\[
{\Delta\over 3} \le \inf\{\bbW_1(\nu,\mu): \mu \mbox{~has $<k$ atoms}\}
        \le \Delta\,.
\]
In particular,
\[
\bbW_1(\nu,\mu) \leq \Delta(\nu)/3 \qquad \implies \qquad |\supp(\mu)| \ge |\supp(\nu)|\,.
\]
Additionally if $\bbW_1(\nu,\mu) \leq \Delta(\nu)/3$ and $|\supp(\mu)| = |\supp(\nu)|$, then $\mu=\sum_{j=1}^k q_j \delta_{z_j}$ where
\begin{equation}
\label{eq:param-dist-general}
\sqrt{\sum_{j=1}^k \big(|p_j-q_j|^2 + d(y_j,z_j)^2\big)}
\leq 
\frac{12 \, \bbW_1(\mu,\nu)}{\Delta}.
\end{equation}
\end{proposition}
\begin{proof} For the lower bound, note that the open balls $B(y_j,d_j/3)$ of radius
$d_j/3$ centered at $y_j$ are mutually disjoint, separated by pairwise distance $\min(d_j)/3$. Hence if $\mu$ has $<k$ atoms, at least one of
these balls must have $\mu$-measure zero, and hence any coupling of $\mu$ to $\nu$ has to
incur at least $\Delta/2$ cost. For the upper bound, simply remove the atom $j$ minimizing the
definition of $\Delta$ and move its mass to the nearest neighbor.

For the second assertion, it follows from the preceding paragraph that $\mu$ has exactly $1$ atom in each ball $B(y_j,d_j/3)$; without loss of generality call this atom $z_j$.
Using a similar transportation argument and then $\min(p,q)+|p-q|=\max(p,q)$, we find
\begin{align*}
\bbW_1(\mu,\nu)
&\geq 
\sum_j d(y_j,z_j)\cdot \min(p_j,q_j)
+
\frac{\sum_j d_j|p_j-q_j|}{6}
\geq 
\frac{\sum_j \big(d(y_j,z_j)p_j + d_j|p_j-q_j|\big)}{12}
\\
&\geq 
\frac{\Delta}{12}\sum_j \big(d(y_j,z_j) + |p_j-q_j|\big)
\geq 
\frac{\Delta }{12} \sqrt{\sum_j \big(d(y_j,z_j)^2 + |p_j-q_j|^2\big)}
.~~
\qedhere
\end{align*}
\end{proof}

The first part of the above lemma will be used to show that the Wasserstein bound from the previous Proposition eventually certifies that
\[
\bbW_1(\wh \pi_\eps,\wh \pi) < \Delta(\wh \pi_\eps)/2 \,.
\]
The second part will be used to obtain convergence in parameter distance $d_{\Pi_k}$.

\begin{proposition}
\label{prop:support-size-cert}
    There exists an absolute constant $C$ such that the following holds.
    Suppose that $\wt\pi$ is given and there exists $c>0$ such that:
    \begin{enumerate}
        \item 
        \label{it:W1-assume}
        $\bbW_1(\wh\pi,\wt\pi)\leq \alpha$.
        \item 
        \label{it:second-order-LB}
        For $C$ an absolute constant and all $y\in\supp(\wt\pi)$,
        \[
        D''_{\wt\pi,X}(y)
        \leq 
        -CL^2 e^{4L^2} (\alpha+cL)
        .
        \]
        \item 
        \label{it:zeroth-order-LB}
        Condition~\ref{it:no-more-zeros} of Proposition~\ref{prop:cert} holds with $(c_1,c_2)=(c,Ce^{4L^2}\alpha)$.
        I.e. for all $z$ with $d(z,\supp(\wt\pi))\geq c$, we have $D_{\wt\pi,X}(z)\leq 
        1-Ce^{4L^2}\alpha$.
    \end{enumerate}
    Then $|\supp(\wh\pi)|\leq|\supp(\wt\pi)|$
    .
\end{proposition}
\begin{proof}
We take $C$ (polynomially) large compared to the constants $C_j$ from Lemma~\ref{lem:lipschitz-bounds}.
By the bound~\eqref{eq:D-bound} applied with $j=3$ and first-order Taylor expansion for $D''$ we
conclude that $D''_{\wt\pi,X}(y)\le -CL^2 e^{4L^2}\alpha$ for
all $y$ in the $c$-neighborhood of any support element of $\wt\pi$. In turn, by \eqref{eq:D-lip} with $j=2$ we have that also $D''_{\wh\pi,X}(y)\le -{C\over2}L^2 e^{4L^2}\alpha$ in the same
neighborhood. Thus, in each such neighborhood $D_{\wh\pi,X}$ can have at most one local maximum,
and in turn at most one $y$ such that $D_{\wh\pi,X}(y)=1$. Outside of $c$-neighborhoods of each
support element of $\wh\pi_\eps$ there can be no atoms of $\wh\pi$ because by~\eqref{eq:D-lip}
with $j=0$ we should have that $D_{\wh\pi,X}(y)<1$ for all $d(y,\supp(\wt \pi))\ge c$.
\end{proof}

\subsection{Approximability of $\wh\pi$}

Here we synthesize the preceding results to obtain an approximation scheme for $\wh\pi$.
Again, we assume here access to an exact maximization oracle for the concave function $\ell_X$ over $\cP(Z_{\eps})$, deferring issues of approximate maximization to the Appendix.
We first define the constants that determine the dependence on $X$. Let 
\begin{equation}
\label{eq:AX}
A_X=-\max_{y\in\supp(\wh\pi)} D''_{\wh\pi,X}(y)
\end{equation}
and 
\begin{align*}
B_X
&=
1-\max_{y\in S_X}
D_{\wh\pi,X}(y),
\\
S_X
&\triangleq
\lt\{
y\in\bbR~:~d(y,\supp(\wh\pi))
\geq 
a_X/4
\rt\},
\\
a_X
&\triangleq
\frac{A_X}{2C_3 L^2 e^{4L^2}}
\,,
\end{align*}
where $C_j$'s are from \eqref{eq:D-lip}.
Both $A_X,B_X$ are almost surely strictly positive by Theorem~\ref{thm:main}, and we will implicitly assume so below.
To approximate $\wh\pi$, we begin as usual with
\[
\wh\pi_{\eps}=\argmax_{\supp(\pi)\subseteq Z_{\eps}} \ell_X(\pi)
\]
for $Z_{\eps}$ as in \eqref{eq:Z-eps-def}.
First, we show a version of Proposition~\ref{prop:stationarity-conditions} for $\wh\pi_{\eps}$.

\begin{lemma}
\label{lem:eps-delta-bound}
    Given $\wh\pi_{\eps}$ we have:
    \begin{enumerate}[label=(\alph*)]
        \item 
        \label{it:D-hatpieps-0}
        $D_{\wh\pi_{\eps},X}(y)=1$ for all $y\in \supp(\wh\pi_{\eps})$.
        \item 
        \label{it:D-hatpieps-positive}
        $D_{\wh\pi_{\eps},X}(y)\leq 1$ for all $y\in Z_{\eps}$.
        \item 
        \label{it:D-hatpieps-delta}
        $\max\limits_{y\in\bbR} D_{\wh\pi_{\eps},X}(y)\leq 1+\delta$ for $\delta = C\eps^2 L^2
    e^{2L^2}$.
    \end{enumerate}
    Further, for any $y\in \supp(\wh\pi_\eps)$ there exists a local maximum $\hat y\in [-L,L]$ of
    $D_{\wh\pi_\eps,X}$ such that $|y-\hat y|<\eps$.
\end{lemma}

\begin{proof}     By \eqref{eq:ell-deriv-formula} and optimality of $\wh\pi_{\eps}$ we see that
    $D_{\wh\pi_{\eps},X}(y)$ is maximized at each $y\in \supp(\wh\pi_{\eps})$ (possibly other $y$ as well). 
    Since the derivative there vanishes when $\pi'=\pi$, we see that this maximum value must be
    $1$, yielding points~\ref{it:D-hatpieps-0} and \ref{it:D-hatpieps-positive}.  

    Point~\ref{it:D-hatpieps-delta} follows by smoothness estimates.
    Let $y\in [y_1,y_1+\eps]$
    with $y_1\in Z_{\eps}$. We recall the following simple analytical fact: if $g$ is a smooth
    function on $[0,\eps]$ and $g(0),g(\eps)\le 1$ then we must have 
        \[\sup_{y\in [0,\eps]}|g(y)| \le 1+{\eps^2\over 2} \sup_{y\in[0,\eps]}|g''(y)|\,,
        \]
    which can be shown by a Taylor expansion. Applying this fact to $D_{\wh\pi_{\eps}, X}$ we find
    that 
    \[
    D_{\wh\pi_{\eps},X}(y)
    \leq 
    1+{\eps^2\over 2} \sup_{z\in\bbR} |D''_{\wh\pi_{\eps},X}(z)|.
    \]
    Recalling \eqref{eq:D-bound} completes the proof of~\ref{it:D-hatpieps-delta} for $y\in [-L,L]$.
    Finally recall from \eqref{eq:Z-eps-def} that we assume $\{-L,L\}\subseteq Z_{\eps}$, and from Lemma~\ref{lem:support-within-[-L,L]} that $D_{\wh\pi_{\eps},X}(\cdot)$ is strictly decreasing on $[-L,\infty)$ and strictly increasing on $(-\infty,-L]$. Hence if $y>L$, then $D_{\wh\pi_{\eps},X}(y)> D_{\wh\pi_{\eps},X}(L)\leq 1+\delta$.
    Proceeding similarly for $y\leq -L$ completes the proof of~\ref{it:D-hatpieps-delta}.

    For the last claim, given $y\in \supp(\wh\pi_\eps)$ let $y_1<y<y_2$ be the closest points to $y$ on each side in $\supp(\wh\pi_\eps)$. First assuming $y\notin \{-L,L\}$ is not an extreme point of $Z_{\eps}$, by definition $[y_1,y_2]\subseteq [y-\eps,y+\eps]$.
    From~\ref{it:D-hatpieps-0} and \ref{it:D-hatpieps-positive}, we see that $D_{\wh \pi_\eps}(y)\geq D_{\wh \pi_\eps}(y_i)$ for $i\in \{1,2\}$.
    Hence $\max_{z\in [y_1,y_2]} D_{\wh \pi_\eps}(z)$ is attained within the interior of the interval, so such a $\hat y$ exists.
    Finally the boundary cases $y=\pm L$ are easily handled using Lemma~\ref{lem:support-within-[-L,L]} to rule out local maxima outside $[-L,L]$.
\end{proof}

We now list several properties of $\wh \pi_\eps$ that hold for sufficiently small $\eps$. 
They do not immediately suffice for \textit{certification} because one needs an
explicit bound on how small $\eps$ must be. 
Instead these properties will be used to ensure the certification criteria in Propositions~\ref{prop:cert},~\ref{prop:cert_lb}
and~\ref{prop:support-size-cert} are applicable to $\wh\pi_\eps$ or its modification $\wt\pi_\eps$, for small enough $\eps$.

\begin{proposition}
\label{prop:apriori-approx-new}
    We have 
    \begin{equation}\label{eq:hausdorff-convergence-n}
        \lim_{\eps \to 0} d_{\cH}(\supp(\wh\pi_{\eps}),\supp(\wh\pi)) = 0\,,
\end{equation}    
    where $d_{\cH}$ denotes Hausdorff distance (recall \eqref{eq:hausdorff-dist-def}). Consequently, for all $\eps <
    \eps_0(X)$, the following hold:
    \begin{align*}
    d(y,\supp(\wh\pi_{\eps}))\leq a_X/2
    \implies 
    D''_{\wh\pi_{\eps},X}(y)
    &\leq
    -A_X/2;
    \\
    d(y,\supp(\wh\pi_{\eps}))\geq a_X/2
    \implies 
    D_{\wh\pi_{\eps},X}(y)
    &\leq 1-\frac{B_X}{2}.
    \end{align*}
\end{proposition}
\begin{proof} First, note that 
    \[
    \lim_{\eps\to 0}
    \max_{y\in\supp(\wh\pi)} d(y,\supp(\wh\pi_\eps))
    =0,
    \]
    since otherwise there would have existed an atom of $\wh\pi$ that is bounded away from any
    atom of $\wh\pi_\eps$ contradicting $\bbW_1(\wh\pi,\wh\pi_\eps)\to0$. 
    Second, we also have
    \[
    \lim_{\eps\to 0}
    \max_{y\in\supp(\wh\pi_{\eps})} d(y,\supp(\wh\pi))
    =0.
    \]
    If we take any convergent subsequence $\supp (\wh\pi_\eps) \ni y_\eps \to y_* $
    then we have $1=D_{\wh\pi_\eps,X}(y_\eps) \to D_{\wh\pi,X}(y_*)$. But by Theorem~\ref{thm:main},
    $D_{\wh\pi,X}(y_*)=1$ (a global maximum) implies $y_* \in \supp(\wh \pi)$, thus showing the
    claim. 
    Together these convergence statements yield \eqref{eq:hausdorff-convergence-n}.
    
    Proposition~\ref{prop:W1-approx-inefficient} ensures that uniformly in the choice of $Z_{\eps}$, we have $\lim_{\eps\to 0}\bbW_1(\wh\pi,\wh\pi_{\eps})=0$.
    By Lemma~\ref{lem:derivative-bounds}, this implies convergence of $D_{\wh\pi_{\eps},X}$ to $D_{\wh\pi,X}$ in the space $C^2([-L,L])$. 
    The other two statements of the
    proposition now follow from the convergence of supports.
\end{proof}

\begin{proposition}
\label{prop:hat-pi-eps-support}
    For $\eps<\eps_0(X)$ sufficiently small depending on $X$, any pair of points $\supp(\wh\pi_{\eps})$
    within distance $a_X/2$ are adjacent elements in $Z_{\eps}$.
\end{proposition}

\begin{proof}
    Take a point $y \in \supp(\wh\pi_\eps)$. By Proposition~\ref{prop:apriori-approx-new} (which uses the assumption $\eps<\eps_0(X)$),
    $D''_{\wh\pi_{\eps},X}\le -A_X/2$ on the interval $[y-a_X/2,y+a_X/2]$. Thus, the equation
    $D_{\wh\pi_{\eps},X}(z)=1$ can have at most two solutions, one of which is $y$. 
    Further the concavity of $D_{\wh\pi_{\eps},X}$ on $[y-a_X/2,y+a_X/2]$ implies it is unimodal.
    Thus if $y'\in [y-a_X/2,y+a_X/2]$ also satisfies $D_{\wh\pi_{\eps},X}(y')=1$, then $D_{\wh\pi_{\eps},X}(y'')=1$ for all $y''\in (y,y')$. 
    Recalling Lemma~\ref{lem:eps-delta-bound} part~\ref{it:D-hatpieps-positive}, we see that no such $y''$ can lie in $Z_{\eps}$.
    We conclude that $y$ and $y'$ must be consecutive within $Z_{\eps}$.
\end{proof}

In light of Proposition~\ref{prop:hat-pi-eps-support}, for $\eps$ small enough atoms in
$\wh\pi_{\eps}$ are all separated by $a_X/2$ except those occuring in pairs of adjacent elements
of $Z_\eps$. We form $\wt\pi_{\eps}$ by taking the weighted average of each such pair, i.e. replacing $p_i \delta_{x_i}$ and $p_{i+1}\delta_{x_{i+1}}$ by $\wt p_i \delta_{\wt x_i}$ for
\[
\wt p_i=p_i+p_{i+1},
\quad\quad
\wt x_i=\frac{p_i x_i + p_{i+1}x_{i+1}}{x_i+x_{i+1}}.
\]

We now show crucial properties of $\wt \pi_\eps$ that will be used in certification. Again, the
results of the following proposition by itself are insufficient since $\eps_0(X)$ is
not determined explicitly. 

\begin{proposition}\label{prop:aan2}
    For all $\eps<\eps_0(X)$ we have that $|\supp(\wt \pi_\eps)| = |\supp(\wh \pi)| =: k$.
    Furthermore, representing $\wt \pi_\eps = \sum_{j=1}^k \tilde p_j \delta_{\tilde y_j}$ and
    $\wh \pi = \sum_{j=1}^k p_j \delta_{y_j}$ (with both atoms sorted):
    $$ \tilde p_j \to p_j, \qquad \tilde y_j \to y_j $$
    as $\eps \to 0$. Thus, in particular, we have 
    \begin{equation}\label{eq:delta_pitilde}
            \lim_{\eps\to 0} \Delta(\wt \pi_\eps) = \Delta(\wh\pi) > 0\,,
    \end{equation}      
\end{proposition}
\begin{proof}
    Given $d_{\cH}(\supp(\wt \pi_\eps),\supp(\wh \pi_\eps)) \le \eps$ and~\eqref{eq:hausdorff-convergence-n}, $\lim\limits_{\eps\to 0}d_{\cH}(\supp(\wt
    \pi_\eps),\supp(\wh \pi)) = 0$. 
    Since the atoms of both measures are separated by at
    least $a_X/4$ we also conclude the two sets have the same cardinality (for small $\eps$), and that 
    $\tilde y_j \to y_j$. 
\end{proof}

Next, we show that $\wt \pi_\eps$ admits
explicit $O_X(\eps^{1/3})$ certificates for Wasserstein convergence and a
certified lower bound on $|\supp(\wh \pi)|$ by virtue of Propositions~\ref{prop:cert}
and~\ref{prop:cert_lb}. (The former $\bbW_1$ convergence also holds for $\wh \pi_\eps$.)

\begin{proposition}
\label{prop:cert-main-new}
    For $\eps\leq \eps_0(X)$, the conditions of Proposition~\ref{prop:cert} apply to $\wt\pi_{\eps}$ with:
    \begin{equation}
    \label{eq:params-for-cert-main}
    \begin{aligned}
        \delta&=O(\eps^2 L^2 e^{4L^2}),
        \\
        c_1&=\sqrt[3]{10 L\delta/A_X},
        \\
        c_2&=\sqrt[3]{L^2\delta^2 A_X},
        \\
        \eta&=O(\sqrt[3]{L\delta/A_X}),
        \\
        \lambda&=\Omega_X(1)
    \end{aligned}
    \end{equation}
    In particular, the estimator $\wt\pi_{\eps}$ obeys certifiable bounds of the form
    $\bbW_1(\wh\pi,\wt\pi_{\eps})\leq O_X(\eps^{1/3})$.
    Furthermore, we also have a bound $|\supp(\wh\pi)|\geq |\supp(\wt\pi_{\eps})|$ certified by
    Proposition~\ref{prop:cert_lb}.
\end{proposition}
\begin{proof}
    The bound on $\delta$ for $\wh\pi_{\eps}$ is in Lemma~\ref{lem:eps-delta-bound}. 
    To extend it to $\wt\pi_{\eps}$, it suffices to show 
    \[
    \sup_{y\in [-L,L]} |D_{\wt\pi_{\eps},X}(y)-D_{\wh\pi_{\eps},X}(y)|
    \stackrel{?}{\leq}
    O(\eps^2 L^2 e^{4L^2}).
    \]
    By definition of $D$, for this it suffices to show that
    \[
    \sup_{y\in [-L,L]} |P_{\wt\pi_{\eps}}(y)-P_{\wh\pi_{\eps}}(y)|
    \stackrel{?}{\leq}
    O(\eps^2),
    \]
    which reduces to proving that 
    \[
    \lt|
    \lt(\frac{p_i}{p_i+p_{i+1}}\rt) e^{-|x_i-y|^2/2}
    +
    \lt(\frac{p_{i+1}}{p_i+p_{i+1}}\rt)
    e^{-|x_{i+1}-y|^2/2}
    -
    e^{-|\wt x_{i}-y|^2/2}
    \rt|
    \stackrel{?}{\leq}
    O(\eps^2).
    \]
    This holds by the following fact (applied with $f(x)=e^{-|x|^2/2}$): If $|f''(x)|\le C$
    for all $x\in[x_0,x_1]$ then for any $\lambda\in[0,1]$ we have 
        $$ |f(\lambda x_1 + (1-\lambda)x_0) - \lambda f(x_1) - (1-\lambda) f(x_0)| \le
    C(x_1-x_0)^2\,.$$
    (This is shown by noticing that this estimate is true at $\lambda=0$ and then expanding the
    right-hand side in $\lambda$ to second order.)
    
    Next, let us show that $D_{\wt \pi_\eps}$ satisfies the $(c_1,c_2)$ assumption in
    Proposition~\ref{prop:cert}. We already know this holds for $c_1 = a_X/2$ and $c_2 =
    B_X/2 - O(\eps^2)$ by Proposition~\ref{prop:apriori-approx-new}, but we need to improve this estimate to the case of vanishing $c_1,c_2$ to get a
    vanishing estimate on $\bbW_1$. 

    To that end, we claim that whenever $2\eps \leq c_1\leq \frac{a_X}{2}$, we may take 
    \begin{equation}
    \label{eq:c1-c2-claim}
    c_2=\frac{A_X (c_1-2\eps)^2}{4}-\delta
    \end{equation} 
    in Proposition~\ref{prop:cert} (where $\delta$ is still as in \eqref{eq:params-for-cert-main} above). 
    Indeed, if 
    \[
    d(y,\supp(\wt\pi_{\eps}))\geq c_1,
    \]
    then $d(y,\wt y)\geq c_1-\eps$ for some $\wt y\in \supp(\wh\pi_{\eps})$.
    Note that $\wt y$ is within distance $\eps$ of a local maximum $\wh y$ of $D_{\wh\pi_{\eps},X}$ by Lemma~\ref{lem:eps-delta-bound}.
    This implies by Proposition~\ref{prop:apriori-approx-new} that, as claimed,
    \[
    D_{\wh\pi_{\eps},X}(y)
    \leq 
    D_{\wh\pi_{\eps},X}(\wh y)
    -
    \frac{A_X(c_1-2\eps)^2 }{4}
    \leq 
    1-\lt(\frac{A_X(c_1-2\eps)^2}{4}-\delta\rt)
    =1-c_2.
    \]
    In particular for $\eps$ small enough, we can set 
    \[
    c_1=\sqrt[3]{10 L\delta/A_X},\quad
    c_2=\sqrt[3]{L^2\delta^2 A_X}
    \implies \eta=O(\sqrt[3]{L\delta/A_X}).
    \]
    Further, Proposition~\ref{prop:lambda-positive} ensures that \eqref{eq:eta-bound} holds for
    $\eps$ small enough, since the value $\lambda_{\eps}$ for $\wh\pi_{\eps}$ converges to that of
    $\wh\pi$. 
    (This follows from the fact that $\wh\pi_{\eps}\to\wh\pi$ in $\bbW_1$ and in Hausdorff distance of supports, by a similar computation as in Lemma~\ref{lem:derivative-bounds}.)
    We conclude that Proposition~\ref{prop:cert} suffices to certify bounds of the form $\bbW_1(\wh\pi,\wt\pi_{\eps})\leq O_X(\sqrt{\eta})=O_X(\eps^{1/3})$.

    Finally, from Proposition~\ref{prop:aan2} we know that conditions of
    Proposition~\ref{prop:cert_lb} are eventually certifiable and hence we get a
    certifiable lower bound $|\supp(\wh\pi)|\geq |\supp(\wt\pi_{\eps})|$ on the support size of
    the NPMLE.
\end{proof}

\begin{proof}[Proof of Theorem~\ref{thm:certify-intro}, except for the Shub--Smale property]
    Proposition~\ref{prop:cert-main-new} shows $\wt \pi_\eps$ has a certified
    $\bbW_1 \leq  O_X(\eps^{1/3})$ upper bound and that $|\supp(\wh \pi)|\ge |\supp(\wt \pi_\eps)|$.

    The upper bound on the support will be certified via Proposition~\ref{prop:support-size-cert}, where we may certify the bound $\alpha=O_X(\eps^{1/3})$ using Proposition~\ref{prop:cert-main-new}.
    Lemma~\ref{lem:derivative-bounds} then yields a (certified) estimate
    \[
    \|D_{\wh\pi_{\eps},X}-D_{\wh\pi,X}\|_{C^2([-L,L])}
    \leq 
    e^{O(L^2)}\alpha
    \leq
    O_X(\eps^{1/3}).
    \]
    To satisfy Condition~\ref{it:zeroth-order-LB} of Proposition~\ref{prop:support-size-cert}, it suffices by \eqref{eq:c1-c2-claim} to take $c=\Theta_X(\eps^{1/6})$. 
    Then Condition~\ref{it:second-order-LB} of Proposition~\ref{prop:support-size-cert} holds for $\eps$ small enough by Proposition~\ref{prop:aan2} (regardless of the $X$-dependent constant factors).
    Recalling how $\wt\pi_{\eps}$ was constructed shows $|\supp(\wh \pi)|= |\supp(\wt \pi_\eps)|$.
    Then \eqref{eq:param-dist-general} gives the same bound for parameter distance of the rounding $\wt\pi_{\eps}$.
    We refer to Appendix~\ref{app:shub-smale} for the proof that roundings are Shub--Smale approximate NPMLEs.
\end{proof}

In the preceding proof, one can actually take $c=\Theta_X(1)$ to be an $\eps$-independent constant.
However due to the unspecified dependence on $X$, this does not yield an actual algorithm.
By making $c$ decay with $\eps$, we ensure the conditions hold once $\eps<\eps_0(X)$ is small enough.

\section{Unbounded Support Size of Higher Dimensional NPMLE}
\label{sec:higher-dim}

In higher-dimensions $d\geq 2$, the NPMLE $\wh\pi$ for a spherical Gaussian mixture need not be unique as was observed in \cite[Lemma 2]{soloff2024multivariate} by taking $X=(x_1,x_2,x_3)$ to be the vertices of an equilateral triangle.
We show below that $|\supp(\wh\pi)|$ may be unbounded for any $d\geq 2$ even when the points $x_i\in X$ are uniformly bounded (and $\wh\pi$ is chosen to minimize the support size if non-unique). In other words, Proposition~\ref{prop:PW20} does not generalize beyond dimension $1$. 
Both of these points indicate that different ideas are required to obtain a fine-grained understanding of the NPMLE in higher dimensions.

It will be helpful to extend \eqref{eq:NPMLE-def} beyond discrete datasets: for any $\mu\in\cP(\bbR^d)$, let the associated NPMLE $\wh\pi(\mu)$ be any maximizer of
\begin{equation}
\label{eq:NPMLE-general-def}
    \ell_{\mu}(\pi)
    \triangleq
    \int \log P_{\pi}(x)~\de \mu(x)
\end{equation}
for $P_{\pi}=\pi*\cN(0,I_d)$. Moreover, let $\mu^{\sph}_{d,r}$ denote the uniform distribution on a centered sphere $\bbS_{d,r}$ of radius $r$ inside $\bbR^d$.

\begin{lemma}
\label{lem:sphere}
    For any dimension $d\geq 2$ and large enough $R\geq R_0(d)$, let $\mu=\mu^{\sph}_{d,R}$. Then $\wh\pi=\mu^{\sph}_{d,r}$ is unique, and $r$ satisfies $0< R-r < o_R(1)$.
\end{lemma}

Lemma~\ref{lem:sphere} is proved in Appendix~\ref{app:lemma-proofs}.
It immediately implies that Proposition~\ref{prop:PW20} does not extend to multiple dimensions.

\begin{proposition}
    For any $d\geq 2$, for large enough $R\geq R_0(d)$ there exists a sequence of data-sets $X_n=(x_1,\dots,x_n)\in (\bbR^d)^n$ satisfying $\|x_i\|\leq R$ for all $1\leq i\leq n$ such that the following holds. If $\wh\pi^{(n)}$ is an NPMLE of $X_n$ for each $n\geq 1$, then $\lim_{n\to\infty} |\supp(\wh\pi^{(n)})|=\infty$.
\end{proposition}

\begin{proof}
    Let $(X_n)_{n\geq 1}$ be any sequence of uniformly bounded datasets such that $\frac{1}{n}\sum_{i=1}^n \delta_{x_i}$ converges in distribution to $\mu^{\sph}_{d,R}$ where $R$ is as in Lemma~\ref{lem:sphere} with $r>0$. 
    It is well-known that any NPMLE $\wh\pi$ for $X_n$ supported inside the radius $R$ ball is also supported inside the radius $R$ ball.
    Note that the function $(X,\pi)\mapsto \ell_{X}(\pi)$ defined in \eqref{eq:NPMLE-def} is jointly continuous for $X,\pi$ supported inside the radius $R$ ball.
    Therefore any subsequential limit of NPMLEs $\wh\pi^{(n)}$ for $X_n$ must be an NPMLE for $\mu^{\sph}_{d,R}$. 
    By Lemma~\ref{lem:sphere} this means $\wh\pi^{(n)}\stackrel{d}{\to}\mu^{\sph}_{d,r}$.
    In particular their support sizes must grow to infinity, which concludes the proof.
\end{proof}

\subsubsection*{Acknowledgment}
YP was supported in part by NSF Grant CCF-2131115 and the MIT-IBM Watson AI Lab.

{\small
 \bibliographystyle{alphaabbr}
 \bibliography{bib}   

\newcommand{\etalchar}[1]{$^{#1}$}
\begin{thebibliography}{DWYZ23}

\bibitem[BDJ{\etalchar{+}}22]{bakshi2022robustly}
A.~Bakshi, I.~Diakonikolas, H.~Jia, D.~M. Kane, P.~K. Kothari, and S.~S.
  Vempala.
\newblock Robustly learning mixtures of $k$ arbitrary gaussians.
\newblock In {\em Proceedings of the 54th Annual ACM SIGACT Symposium on Theory
  of Computing}, pages 1234--1247, 2022.

\bibitem[BGG{\etalchar{+}}07]{brandolini2007average}
L.~Brandolini, G.~Gigante, A.~Greenleaf, A.~Iosevich, A.~Seeger, and
  G.~Travaglini.
\newblock {Average decay estimates for Fourier transforms of measures supported
  on curves}.
\newblock {\em J. Geom. Anal.}, 17(1):15--40, 2007.

\bibitem[B{\"o}h86]{bohning1986vertex}
D.~B{\"o}hning.
\newblock {A vertex-exchange-method in $D$-optimal design theory}.
\newblock {\em Metrika}, 33(1):337--347, 1986.

\bibitem[BSL92]{bohning1992computer}
D.~B{\"o}hning, P.~Schlattmann, and B.~Lindsay.
\newblock {Computer-assisted analysis of mixtures (C.A.MAN): statistical
  algorithms}.
\newblock {\em Biometrics}, pages 283--303, 1992.

\bibitem[Bub15]{bubeck2015convex}
S.~Bubeck.
\newblock Convex optimization: Algorithms and complexity.
\newblock {\em Foundations and Trends{\textregistered} in Machine Learning},
  8(3-4):231--357, 2015.

\bibitem[BWY17]{balakrishnan2017statistical}
S.~Balakrishnan, M.~J. Wainwright, and B.~Yu.
\newblock {Statistical guarantees for the EM algorithm: From population to
  sample-based analysis}.
\newblock {\em Ann. Stat.}, 45(1):77--120, 2017.

\bibitem[Das99]{dasgupta1999learning}
S.~Dasgupta.
\newblock Learning mixtures of gaussians.
\newblock In {\em 40th Annual Symposium on Foundations of Computer Science
  (Cat. No. 99CB37039)}, pages 634--644. IEEE, 1999.

\bibitem[Der86]{dersimonian1986algorithm}
R.~DerSimonian.
\newblock Maximum likelihood estimation of a mixing distribution.
\newblock {\em Applied Statistics}, pages 302--309, 1986.

\bibitem[DLR77]{dempster1977maximum}
A.~P. Dempster, N.~M. Laird, and D.~B. Rubin.
\newblock {Maximum likelihood from incomplete data via the EM algorithm}.
\newblock {\em Journal of the Royal Statistical Society: Series B
  (Methodological)}, 39(1):1--22, 1977.

\bibitem[DTZ17]{daskalakis2017ten}
C.~Daskalakis, C.~Tzamos, and M.~Zampetakis.
\newblock {Ten steps of EM suffice for mixtures of two Gaussians}.
\newblock In {\em Conference on Learning Theory}, pages 704--710. PMLR, 2017.

\bibitem[DWYZ23]{doss2020optimal}
N.~Doss, Y.~Wu, P.~Yang, and H.~H. Zhou.
\newblock {Optimal estimation of high-dimensional Gaussian location mixtures}.
\newblock {\em Ann. Stat.}, 51(1):62--95, 2023.

\bibitem[DZ16]{dicker2016high}
L.~H. Dicker and S.~D. Zhao.
\newblock {High-dimensional classification via nonparametric empirical Bayes
  and maximum likelihood inference}.
\newblock {\em Biometrika}, 103(1):21--34, 2016.

\bibitem[FD18]{feng2018approximate}
L.~Feng and L.~H. Dicker.
\newblock Approximate nonparametric maximum likelihood for mixture models: A
  convex optimization approach to fitting arbitrary multivariate mixing
  distributions.
\newblock {\em Computational Statistics \& Data Analysis}, 122:80--91, 2018.

\bibitem[FGSW23]{fan2023gradient}
Z.~Fan, L.~Guan, Y.~Shen, and Y.~Wu.
\newblock {Gradient flows for empirical Bayes in high-dimensional linear
  models}.
\newblock {\em arXiv preprint arXiv:2312.12708}, 2023.

\bibitem[GVDV01]{ghosal2001entropies}
S.~Ghosal and A.~W. Van Der~Vaart.
\newblock {Entropies and rates of convergence for maximum likelihood and Bayes
  estimation for mixtures of normal densities}.
\newblock {\em Ann. Stat.}, 29(5):1233--1263, 2001.

\bibitem[GW00]{genovese2000rates}
C.~R. Genovese and L.~Wasserman.
\newblock {Rates of convergence for the Gaussian mixture sieve}.
\newblock {\em Ann. Stat.}, 28(4):1105--1127, 2000.

\bibitem[GW12]{groeneboom2012information}
P.~Groeneboom and J.~A. Wellner.
\newblock {\em Information bounds and nonparametric maximum likelihood
  estimation}, volume~19.
\newblock Birkh{\"a}user, 2012.

\bibitem[HP15]{hardt2015tight}
M.~Hardt and E.~Price.
\newblock {Tight bounds for learning a mixture of two Gaussians}.
\newblock In {\em Proceedings of the forty-seventh annual ACM symposium on
  Theory of computing}, pages 753--760, 2015.

\bibitem[Jag13]{jaggi2013revisiting}
M.~Jaggi.
\newblock {Revisiting Frank-Wolfe: Projection-free sparse convex optimization}.
\newblock In {\em International conference on machine learning}, pages
  427--435. PMLR, 2013.

\bibitem[JZ09]{jiang2009general}
W.~Jiang and C.-H. Zhang.
\newblock {General maximum likelihood empirical Bayes estimation of normal
  means}.
\newblock {\em Ann. Stat.}, 37(1):1647--1684, 2009.

\bibitem[JZB{\etalchar{+}}16]{jin2016local}
C.~Jin, Y.~Zhang, S.~Balakrishnan, M.~J. Wainwright, and M.~I. Jordan.
\newblock Local maxima in the likelihood of {G}aussian mixture models:
  Structural results and algorithmic consequences.
\newblock In {\em Advances in neural information processing systems}, pages
  4116--4124, 2016.

\bibitem[Kan21]{kane2021robust}
D.~M. Kane.
\newblock Robust learning of mixtures of gaussians.
\newblock In {\em Proceedings of the 2021 ACM-SIAM Symposium on Discrete
  Algorithms (SODA)}, pages 1246--1258. SIAM, 2021.

\bibitem[KCSA20]{kim2020fast}
Y.~Kim, P.~Carbonetto, M.~Stephens, and M.~Anitescu.
\newblock A fast algorithm for maximum likelihood estimation of mixture
  proportions using sequential quadratic programming.
\newblock {\em Journal of Computational and Graphical Statistics},
  29(2):261--273, 2020.

\bibitem[KG17]{koenker2017rebayes}
R.~Koenker and J.~Gu.
\newblock Rebayes: an r package for empirical bayes mixture methods.
\newblock {\em Journal of Statistical Software}, 82:1--26, 2017.

\bibitem[KG19]{koenker2019minimalist}
R.~Koenker and J.~Gu.
\newblock {Minimalist $g$-Modeling}.
\newblock {\em Statistical Science}, 34(2):209 -- 213, 2019.

\bibitem[KM14]{koenker2014convex}
R.~Koenker and I.~Mizera.
\newblock {Convex optimization, shape constraints, compound decisions, and
  empirical Bayes rules}.
\newblock {\em Journal of the American Statistical Association},
  109(506):674--685, 2014.

\bibitem[KMV10]{kalai2010efficiently}
A.~T. Kalai, A.~Moitra, and G.~Valiant.
\newblock {Efficiently learning mixtures of two Gaussians}.
\newblock In {\em Proceedings of the forty-second ACM symposium on Theory of
  computing}, pages 553--562, 2010.

\bibitem[KW56]{kiefer1956consistency}
J.~Kiefer and J.~Wolfowitz.
\newblock Consistency of the maximum likelihood estimator in the presence of
  infinitely many incidental parameters.
\newblock {\em Ann. Math. Stat.}, pages 887--906, 1956.

\bibitem[Lai78]{Laird1978}
N.~Laird.
\newblock Nonparametric maximum likelihood estimation of a mixing distribution.
\newblock {\em Journal of the American Statistical Association},
  73(364):805--811, 1978.

\bibitem[Lin83a]{lindsay1983geometryA}
B.~G. Lindsay.
\newblock {The Geometry of Mixture Likelihoods: A General Theory}.
\newblock {\em Ann. Stat.}, pages 86--94, 1983.

\bibitem[Lin83b]{lindsay1983geometryB}
B.~G. Lindsay.
\newblock {The Geometry of Mixture Likelihoods, Part II: The Exponential
  Family}.
\newblock {\em Ann. Stat.}, 11(3):783--792, 1983.

\bibitem[LK92]{lesperance1992algorithm}
M.~L. Lesperance and J.~D. Kalbfleisch.
\newblock {An algorithm for computing the nonparametric MLE of a mixing
  distribution}.
\newblock {\em Journal of the American Statistical Association},
  87(417):120--126, 1992.

\bibitem[LM22]{liu2022learning}
A.~Liu and A.~Moitra.
\newblock {Learning GMMs with nearly optimal robustness guarantees}.
\newblock In {\em Conference on Learning Theory}, pages 2815--2895. PMLR, 2022.

\bibitem[LM23]{liu2021settling}
A.~Liu and A.~Moitra.
\newblock Robustly learning general mixtures of gaussians.
\newblock {\em J. ACM}, 70(3):1--53, 2023.

\bibitem[LR93]{lindsay1993uniqueness}
B.~G. Lindsay and K.~Roeder.
\newblock Uniqueness of estimation and identifiability in mixture models.
\newblock {\em Canadian J. Stat.}, 21(2):139--147, 1993.

\bibitem[Mar88]{marshall1988decay}
B.~Marshall.
\newblock {Decay rates of Fourier transforms of curves}.
\newblock {\em Trans. Amer. Math. Soc.}, 310(1):115--126, 1988.

\bibitem[MF05]{ma2005correct}
J.~Ma and S.~Fu.
\newblock {On the correct convergence of the EM algorithm for Gaussian
  mixtures}.
\newblock {\em Pattern Recognition}, 38(12):2602--2611, 2005.

\bibitem[MR94]{meng1994global}
X.-L. Meng and D.~B. Rubin.
\newblock {On the global and componentwise rates of convergence of the EM
  algorithm}.
\newblock {\em Linear Algebra and its Applications}, 199:413--425, 1994.

\bibitem[MSS23]{mukherjee2023mean}
S.~Mukherjee, B.~Sen, and S.~Sen.
\newblock {A Mean Field Approach to Empirical Bayes Estimation in
  High-dimensional Linear Regression}.
\newblock {\em arXiv preprint arXiv:2309.16843}, 2023.

\bibitem[MV10]{moitra2010settling}
A.~Moitra and G.~Valiant.
\newblock {Settling the polynomial learnability of mixtures of Gaussians}.
\newblock In {\em 2010 IEEE 51st Annual Symposium on Foundations of Computer
  Science}, pages 93--102. IEEE, 2010.

\bibitem[MX05]{ma2005asymptotic}
J.~Ma and L.~Xu.
\newblock {Asymptotic convergence properties of the EM algorithm with respect
  to the overlap in the mixture}.
\newblock {\em Neurocomputing}, 68:105--129, 2005.

\bibitem[MXJ00]{ma2000asymptotic}
J.~Ma, L.~Xu, and M.~I. Jordan.
\newblock {Asymptotic convergence rate of the EM algorithm for Gaussian
  mixtures}.
\newblock {\em Neural Computation}, 12(12):2881--2907, 2000.

\bibitem[Pea94]{pearson1894contributions}
K.~Pearson.
\newblock Contributions to the mathematical theory of evolution.
\newblock {\em Philosophical Transactions of the Royal Society of London. A},
  185:71--110, 1894.

\bibitem[PW20]{polyanskiy2020self}
Y.~Polyanskiy and Y.~Wu.
\newblock {Self-Regularizing Property of Nonparametric Maximum Likelihood
  Estimator in Mixture Models}.
\newblock {\em arXiv preprint arXiv:2008.08244}, 2020.

\bibitem[Rob50]{robbins1950generalization}
H.~Robbins.
\newblock A generalization of the method of maximum likelihood: {E}stimating a
  mixing distribution ({A}bstract).
\newblock In {\em Annals of Mathematical Statistics}, volume~21, pages
  314--315, 1950.

\bibitem[RW84]{redner1984mixture}
R.~A. Redner and H.~F. Walker.
\newblock {Mixture densities, maximum likelihood and the EM algorithm}.
\newblock {\em SIAM Review}, 26(2):195--239, 1984.

\bibitem[SB13]{stoer2013introduction}
J.~Stoer and R.~Bulirsch.
\newblock {\em {Introduction to Numerical Analysis}}, volume~12.
\newblock Springer Science \& Business Media, 2013.

\bibitem[SF99]{EC2}
R.~P. Stanley and S.~Fomin.
\newblock {\em Enumerative Combinatorics}.
\newblock Cambridge Studies in Advanced Mathematics. Cambridge University
  Press, 1999.

\bibitem[SG20]{saha2020nonparametric}
S.~Saha and A.~Guntuboyina.
\newblock {On the nonparametric maximum likelihood estimator for Gaussian
  location mixture densities with application to Gaussian denoising}.
\newblock {\em Ann. Stat.}, 48(2):738--762, 2020.

\bibitem[SGS24]{soloff2024multivariate}
J.~A. Soloff, A.~Guntuboyina, and B.~Sen.
\newblock {Multivariate, heteroscedastic empirical Bayes via nonparametric
  maximum likelihood}.
\newblock {\em Journal of the Royal Statistical Society Series B: Statistical
  Methodology}, 2024.

\bibitem[SM93]{stein1993harmonic}
E.~M. Stein and T.~S. Murphy.
\newblock {\em Harmonic analysis: real-variable methods, orthogonality, and
  oscillatory integrals}, volume~3.
\newblock Princeton University Press, 1993.

\bibitem[SS93]{shub1993complexity}
M.~Shub and S.~Smale.
\newblock {Complexity of Bezout's theorem I: Geometric aspects}.
\newblock {\em J. Amer. Math. Soc.}, 6(2):459--501, 1993.

\bibitem[Tan17]{tan2017energy}
Y.~S. Tan.
\newblock {Energy optimization for distributions on the sphere and improvement
  to the Welch bounds}.
\newblock {\em Electron. Comm. Probab.}, 22:1 -- 12, 2017.

\bibitem[WB22]{weinberger2022algorithm}
N.~Weinberger and G.~Bresler.
\newblock {The EM algorithm is adaptively-optimal for unbalanced symmetric
  Gaussian mixtures}.
\newblock {\em Journal of Machine Learning Research}, 23(103):1--79, 2022.

\bibitem[WIM25]{wang2025nonparametric}
H.~Wang, S.~Ibrahim, and R.~Mazumder.
\newblock {Nonparametric Finite Mixture Models with Possible Shape Constraints:
  A Cubic Newton Approach}.
\newblock {\em SIAM Journal on Mathematics of Data Science}, 7(1):163--188,
  2025.

\bibitem[WN22]{wei2022convergence}
Y.~Wei and X.~Nguyen.
\newblock {Convergence of de Finetti’s mixing measure in latent structure
  models for observed exchangeable sequences}.
\newblock {\em Ann. Stat.}, 50(4):1859--1889, 2022.

\bibitem[WY20]{wu2020optimal}
Y.~Wu and P.~Yang.
\newblock Optimal estimation of gaussian mixtures via denoised method of
  moments.
\newblock {\em Ann. Stat.}, 48(4):1981--2007, 2020.

\bibitem[WZ22]{wu2022randomly}
Y.~Wu and H.~H. Zhou.
\newblock {Randomly initialized EM algorithm for two-component Gaussian mixture
  achieves near optimality in $O(\sqrt{n})$ iterations}.
\newblock {\em Mathematical Statistics and Learning}, 4(3):143--220, 2022.

\bibitem[XJ96]{xu1996convergence}
L.~Xu and M.~I. Jordan.
\newblock {On convergence properties of the EM algorithm for Gaussian
  mixtures}.
\newblock {\em Neural Computation}, 8(1):129--151, 1996.

\bibitem[Zha09]{zhang2009generalized}
C.-H. Zhang.
\newblock Generalized maximum likelihood estimation of normal mixture
  densities.
\newblock {\em Statistica Sinica}, pages 1297--1318, 2009.

\bibitem[ZLS20]{zhao2020statistical}
R.~Zhao, Y.~Li, and Y.~Sun.
\newblock {Statistical convergence of the EM algorithm on Gaussian mixture
  models}.
\newblock {\em Electron. J. Stat.}, 14:632--660, 2020.

\end{thebibliography}
}

\begin{appendix}
\section{Proofs of Lemmas}
\label{app:lemma-proofs}

Here we provide self-contained proofs of several lemmas from the main body.

\begin{proof}[Proof of Proposition~\ref{prop:W1-Pik-bound}]
    Recalling \eqref{eq:Pik-distance}, define a coupling $(y,y')$ between $\pi,\pi'$ as follows.
    For each $j$, with probability $\min(p_j,p_j')$ we have $(y,y')=(y_j,y_j')$. The remainder of the coupling is arbitrary. Note that by definition this remaining probability is 
    \[
    1-\sum_{j=1}^k \min(p_j,p_j') = {1\over 2} \sum_j |p_j - p_j'| \le
    \frac{1}{2}\sqrt{L\sum_j |p_j - p_j'|^2}.
    \]
    Since $\diam([-L,L])=2L$, we obtain
    \[
    \bbW_1(\pi,\pi')
    \leq
    L^{3/2}\, \sqrt{\sum_j |p_j - p_j'|^2} + \max_{1\leq j\leq k}|y_j-y_j'| 
    \leq 
    (L^{3/2}+1)\,d_{\Pi_k}(\pi,\pi').\qedhere
    \]
\end{proof}

\begin{proof}[Proof of Lemma~\ref{lem:lipschitz-bounds}]
    Point~\ref{it:P-LB} is trivial.
    Letting $(y,y')\sim \Gamma$ be an optimal coupling of $\pi,\pi'$,
    point~\ref{it:P-lip} follows easily from \eqref{eq:P-def}:
    \[
    |P_{\pi}(x)-P_{\pi'}(x)|
    \leq
    \bbE^{(y,y')\sim\Gamma}
    |e^{-|x-y|^2/2}-e^{-|x-y'|^2/2}|
    \leq
    C\bbE^{(y,y')\sim\Gamma}|y-y'|
    \leq C\delta.
    \]
    Point~\ref{it:P-invlip} follows easily by combining the previous two.
    For point \ref{it:l-lip}, note that since \ref{it:P-LB} holds for both $\pi$ and $\pi'$.
    Hence recalling the definition \eqref{eq:NPMLE-def} of $\ell_X$, the Lipschitz constant of the logarithm on $[C^{-1}e^{-2L^2},\infty)$ is $O(e^{2L^2})$ which gives the claim.
    The bound \eqref{eq:T-bound} follows from point~\ref{it:P-LB}.
    Using \eqref{eq:D-deriv-hermite} for each term $T^{(j)}$ we have
    \[
    |T^{(j)}_{\pi,x}(z)-T^{(j)}_{\pi',x}(z)|
    \leq
    \lt|
    \frac{H_j(x-z)e^{-|x-z|^2/2}}{P_{\pi}(x_i)\sqrt{2\pi}}
    -
    \frac{H_j(x-z)e^{-|x-z|^2/2}}{P_{\pi'}(x_i)\sqrt{2\pi}}
    \rt|
    \leq 
    C_j L^j 
    \lt|
    \frac{1}{P_{\pi}(x_i)}
    -
    \frac{1}{P_{\pi'}(x_i)}
    \rt|
    .
    \]
    This proves \ref{it:T-lip} in light of point~\ref{it:P-invlip}. 
    The analogous bounds for $D$ follow from \eqref{eq:D-formula}.
\end{proof}

\begin{proof}[Proof of Lemma~\ref{lem:derivative-bounds}]
    We repeatedly differentiate. Let $v_j=q_j-p_j$ and $p_{t,j}=(1-t)p_j+tq_j$.
    Using \eqref{eq:ell-deriv-formula} in
    the first line: 
    \[
    \ell'(t)
    =
    \sum_{j=1}^k
    v_j D_{\pi_t,X}(y_j)
    \stackrel{\eqref{eq:D-formula}}{=}
    \frac{1}{n}
    \sum_{m=1}^n
    \lt(
    \sum_{j=1}^k
    \frac{v_j e^{-|x_m-y_j|^2/2}}
    {\sum_{\ell} p_{t,\ell} e^{-|x_m-y_{\ell}|^2/2}}
    \rt).
    \]
    Clearly it suffices to consider the case $n=1$, with $x_1=x$, which removes the outer average over $m$. 
    Then we find:
    \begin{align*}
    \ell''(t)
    &=
    \sum_{j=1}^k
    v_j \frac{\de}{\de t}\lt(\frac{e^{-|x-y_j|^2/2}}{\sum_{\ell} p_{t,\ell} e^{-|x-y_{\ell}|^2/2}}\rt)
    =
    -\sum_{j,j'=1}^k
    v_j 
    v_{j'}
    \lt(\frac{\exp\big(-\frac{|x-y_j|^2+|x-y_{j'}|^2}{2}\big)}
    {\lt(\sum_{\ell} p_{t,\ell} e^{-|x-y_{\ell}|^2/2}\rt)^2}\rt)
    ;
    \\
    \ell'''(t)
    &=
    2\sum_{j,j',j''=1}^k
    v_j 
    v_{j'}
    v_{j''}
    \lt(\frac{\exp\big(-\frac{|x-y_j|^2+|x-y_{j'}|^2+|x-y_{j''}|^2}{2}\big)}
    {\lt(\sum_{\ell} p_{t,\ell} e^{-|x-y_{\ell}|^2/2}\rt)^3}\rt)
    .
    \end{align*}
    To lower bound the denominator, note that
    \[
    \sum_{\ell} p_{t,\ell} e^{-|x-y_{\ell}|^2/2}\geq 
    \min_{x,y\in[-L,L]}
    e^{-|x-y|^2/2}
    =
    e^{-2L^2}.
    \]
    This easily gives the main conclusion. 
    When $\pi([-10,10])\geq 0.1$, we have $\sum_{\ell} p_{t,\ell} e^{-|x-y_{\ell}|^2/2}\geq e^{-(L+10)^2/2}/10\geq \Omega(e^{-0.51L^2})$.
\end{proof}

\begin{proof}[Proof of Proposition~\ref{prop:cert}]
We first consider the case $\wt\pi_{\eps}([-10,10])\geq 1/10$. 
    Note that for $\pi_t=(1-t)\wt\pi_{\eps}+t\wh\pi$, by definition of $\wh\pi$:
    \begin{equation}
    \label{eq:directional-derivative-general}
    0\leq 
    \frac{\de}{\de t}\ell(\pi_t)|_{t=0}
    =
    \int 
     D_{\wt\pi_{\eps},X}(y)\big(\de\wh\pi(y)-\de \wt\pi_{\eps}(y)\big)
     \stackrel{\eqref{eq:D-sum-formula}}{=}
     \int 
     \big(D_{\wt\pi_{\eps},X}(y)-1\big)\de\wh\pi(y)
    .
    \end{equation}
    This implies the bulk of $\wh\pi$ lies near the support of $\wt\pi_{\eps}$: by Assumptions~\ref{it:no-more-zeros} and \ref{it:delta-LB},
    \[
    \bbP^{y\sim \wh\pi}
    [
    d(y,\supp(\wt\pi_{\eps}))\geq c_1
    ]
    \leq
    \frac{\delta}{\delta+c_2}\leq \delta/c_2.
    \]

    We exploit this as follows. 
    By the previous display, there exists $\wt\pi_*$ supported in $\supp(\wt\pi_{\eps})$ such that $\bbW_1(\wt\pi_*,\wh\pi)\leq c_1 + 2L\delta/c_2\leq 2\eta$. This implies
    \[
    |\ell_X(\wh\pi)-\ell_X(\wt\pi_*)|\leq O(e^{2L^2}\eta)
    \]
    by Lemma~\ref{lem:lipschitz-bounds}. We write $\wt\pi_*=\sum_y q_j\delta_{y_j}$. 
    
    Now that $\wt\pi_*$ and $\wt\pi_{\eps}$ have the same support, it remains to show their weights are similar. We do this using strong convexity via \eqref{eq:loss-hessian-condition}.
    Set
    $c=\sum_{j} |p_j-q_j|$, which we will show is small.
    We write 
    \[
    \ell(t)=\ell_X((1-t)\wt\pi_{\eps}+t\wt\pi_*)
    .
    \]
    Then we have $\ell'(0)=0$ since $\wt\pi_{\eps}$ is an exact optimum among distributions supported in $Z_{\eps}$, and $\ell''(0)\leq -c^2 \lambda/k$ by \eqref{eq:loss-hessian-condition}. 
    By Lemma~\ref{lem:derivative-bounds} with $j=3$, we find $\|\ell'''\|_{\infty}\leq O(c^3 e^{1.55L^2})$ and so $\ell''(t)\leq -c^2 \lambda/2k$ for all
    \[
    |t|\leq T\triangleq \Omega
    \lt(\frac{\lambda}{ck e^{1.55L^2}}\rt).
    \]
    Thus, since $\ell(t)$ is non-increasing, we have
    \[
    \ell_X(\wt\pi_*)-\ell_X(\wt\pi_{\eps})
    =
    \ell(1)-\ell(0)\leq \ell(1 \wedge T) - \ell(0) \leq - (1\wedge T)^2 c^2 \lambda/4k.
    \]
    Combining the above, we have
    \begin{align*}
    \ell_X(\wh\pi)
    -
    O\lt(\eta e^{2L^2}\rt)
    &\leq
    \ell_X(\wt\pi_*)
    \leq
    \ell_X(\wt\pi_{\eps})
    -
    (1\wedge T)^2c^2 \lambda/4k
    \\
    &\leq
    \ell_X(\wh\pi)
    -
    (1\wedge T)^2c^2 \lambda/4k
    .
    \end{align*}
    Therefore 
    \begin{equation}
    \label{eq:needs-k^3}
    (1\wedge T)c \leq 
    Ce^{L^2}\sqrt{\frac{\eta k}{\lambda}}
    .
    \end{equation}
    Due to the assumption~\eqref{eq:eta-bound}, the latter estimate cannot hold with $Tc$ on the left-hand side.
    Therefore we have an upper bound for $c$ which yields:
    \begin{align*}
    \bbW_1(\wh\pi,\wt\pi_{\eps})
    &\leq
    \bbW_1(\wh\pi,\pi_t)
    +
    \bbW_1(\pi_t,\wt\pi_{\eps})
    \leq
    Lc+O(\eta)
    \leq
     O\Big( Le^{L^2}\sqrt{\frac{\eta k}{\lambda}}+\eta\Big).
    \end{align*}
    The former term dominates the latter when \eqref{eq:eta-bound} holds, finishing the proof of the main bound. 

    Without assumption $\wt\pi_{\eps}([-10,10])\geq 1/10$, the proof is the same: $5.1$ comes from
    $5.1=2(1+1.55)$, and replacing $1.55$ by $6$ (based on Lemma~\ref{lem:derivative-bounds}) gives $14$. 
\end{proof}

We next turn to proving Lemma~\ref{lem:sphere}, which uses the following fact.
It can be shown by expansion into spherical harmonics, or see \cite[Theorem 4.4]{tan2017energy} for an elementary approach.

\begin{proposition}
\label{prop:welch}
    Let $F:\bbR\to\bbR$ be given by a globally absolutely convergent power series $F(t)=\sum_{k\geq 0} a_k t^k$ with strictly positive coefficients $a_k>0$. Then for any $\nu\in\cP(\bbS_{d,r})$,
    \[
    \iint 
    F(\la x,y\ra)
    ~\de \nu(x)\de\nu(y)
    \geq 
    \iint 
    F(\la x,y\ra)
    ~\de \mu^{\sph}_{d,r}(x)\de\mu^{\sph}_{d,r}(y)
    \]
    with equality if and only if $\nu=\mu^{\sph}_{d,r}$.
\end{proposition}

\begin{proof}[Proof of Lemma~\ref{lem:sphere}]
    First, let $\wh\pi_{\sym}$ be the spherical symmetrization of $\wh\pi$.
    It is easy to see by concavity of the logarithm that $\ell_{\mu}(\wh\pi)\leq \ell_{\mu}(\wh\pi_{\sym})$. Thus $\wh\pi_{\sym}$ is a mixture of $\mu^{\sph}_{d,r'}$ for $r'\in\bbR_{\geq 0}$. We will show separately that there is a unique optimal choice $r$ of $r'$, and that $\ell_{\mu}(\wh\pi)<\ell_{\mu}(\wh\pi_{\sym})$ unless $\wh\pi=\wh\pi_{\sym}$.

    \paragraph*{Uniqueness of Optimal Radius}
    First, note that for any $r'$, the density of $\mu^{\sph}_{d,r'}*\cN(0,I_d)$ is constant on the sphere of radius $R$. Hence maximizing \eqref{eq:NPMLE-general-def} is equivalent to maximizing this value in $r'$. Thus, let $f(r_1,r_2)$ be the density of $\mu^{\sph}_{d,r_1}*\cN(0,I_d)$ on the sphere of radius $r_2$. Letting $\mu^{\sph}_{d,r,\eps}$ be the uniform distribution on the dimension $d$ annulus $A_{d,r,\eps}$ of inner and outer radii $r$ and $r+\eps$, and using $\rho_{\nu}$ to denote the density of an absolutely continuous distribution:
    \begin{align*}
    f(r_1,r_2)
    &=
    \frac{1}{2\pi r_2}\lim_{\eps\to 0}
    \int_{A_{d,r_2,\eps}}~
    \de\mu^{\sph}_{d,r_1,\eps}*\cN(0,I_d)(x)
    =
    \lim_{\eps\to 0}
    \int
    \rho_{\mu^{\sph}_{d,r_2,\eps}}(x)
    ~
    \de\mu^{\sph}_{d,r_1,\eps}*\cN(0,I_d)(x) 
    \\
    &=
    \lim_{\eps\to 0}
    \int
    \rho_{\mu^{\sph}_{d,r_2,\eps}}(-x)
    ~
    \de\mu^{\sph}_{d,r_1,\eps}*\cN(0,I_d)(x) 
    =
    \lim_{\eps\to 0}
\rho_{\mu^{\sph}_{d,r_1,\eps}*\cN(0,I_d)*\mu^{\sph}_{d,r_2,\eps}}(0).
    \end{align*}
    In particular $f(r_1,r_2)=f(r_2,r_1)$, and so $r$ is the radius on which the density $\wh\rho=\rho_{\mu^{\sph}_{d,R}*\cN(0,I_d)}$ is maximized. By symmetry, it suffices to consider the restriction of $\wh\rho$ to $\{r' v~:~r'\geq 0\}$ for a unit vector $v\in\bbR^d$.
    The uniqueness of an optimal radius $R-o_R(1)< r<R$ is now geometrically rather clear and we outline a formal proof.
    
    First if $|R-r'|\geq R^{1/10}$, then $\wh\rho(r'v)\leq O(R^{-2d})$ which will turn out to be of lower order than the maximum value.
    Hence we restrict attention to $r'= R\pm O(R^{1/10})$; here the contribution to $\wh\rho$ from the part of $\mu^{\sph}_{d,R}$ supported at distance at least $R^{1/4}$ from $v$ is at most $O(R^{-2d})$ in $C^2$, thanks to the super-polynomial decay of the Gaussian density and its derivatives.

    For $x\in \supp(\mu^{\sph}_{d,R})$ with $\|x-v\|\leq R^{1/4}$, we replace $x$ by its projection $x'$ onto the tangent hyperplane to $\bbS_{d,R}$ at $Rv$. It is easy to see that $\|\bx-\bx'\|\leq O(R^{-3/2})$, so this change affects $\wh\rho(r' v)$ by at most $O_d\lt(e^{-\|x'-Rv\|^2/3} R^{-1}\rt)$ in $C^2$ per unit mass $\de\mu(x)$.
    Finally approximating $\wt\rho(r'v)$ by integrating over $x'$ instead of $x$, we get an approximation $\wt\rho(r' v)$ which is simply a Gaussian density centered at $R$ and rescaled by a factor of $\Theta_d(R^{-(d-1)})$ (for the fraction of $\bbS_{d,R}$ within an $O(1)$ distance of $Rv$).
    The error from changing $x\to x'$ and including $x$ at distance greater than $R^{1/4}$ is at most $O_d(R^{-d})$ in $C^2$ norm. Combining the above shows that a maximizing $r$ is unique and satisfies $|R-r|\leq o_R(1)$. Finally $r<R$ simply because any NPMLE must be supported on the strict interior of $\supp(\mu^{\sph}_{d,R})$.

    We note that at this point, $\mu^{\sph}_{d,r}$ has been shown to be \textbf{an} NPMLE for $\mu^{\sph}_{d,R}$. The latter part of the proof below shows it is the \textbf{only} NPMLE.

    \paragraph*{Spherical Symmetry of $\wh\pi$}

    Given the preceding discussion, we know that any NPMLE $\wh\pi$ is supported on $\bbS_{d,r}$ for some unique $r$. Recalling \eqref{eq:NPMLE-general-def}, note that without the logarithm, the quantity
    $\int P_{\pi}(x)~\de \mu(x)$ is constant over all such $\wh\pi$. Therefore by concavity of the logarithm, it suffices to prove that $\mu^{\sph}_{d,r}$ is the unique probability measure $\nu$ on $\bbS_{d,r}$ whose convolution with $\cN(0,I_d)$ produces a constant density when restricted to $\bbS_{d,R}$. We will do so by proving that it uniquely minimizes the $L^2$ energy of the density, given by
    \begin{equation}
    \label{eq:L2-energy-sphere}
    \int
    \rho_{\nu*\cN(0,I_d)}(w)^2
    ~\de \mu^{\sph}_{d,R}(w).
    \end{equation}

    To establish the latter fact, we will expand and rearrange the integral in order to apply Proposition~\ref{prop:welch}. Crucially, note that for $\|x\|=r$ and $\|w\|=R$, we have 
    $
    e^{-\|x-w\|^2/2}
    \propto
    e^{\la x,w\ra}$
    with constant of proportionality $e^{-(r^2+R^2)/2}$ depending only on $r,R$.
    Using this observation, we expand \eqref{eq:L2-energy-sphere} and interchange the order of integration:
    \begin{equation}
    \label{eq:change-order}
    \begin{aligned}
    \int
    \rho_{\nu*\cN(0,I_d)}(w)^2
    ~\de \mu^{\sph}_{d,R}(w)
    &\propto
    \iiint
    e^{\la x+y,w\ra}
    ~\de \mu^{\sph}_{d,R}(w)
    \de \nu(x)\de \nu(y)
    \\
    &=
    \iint 
    \lt(
    \int
    e^{\la x+y,w\ra}
    ~\de \mu^{\sph}_{d,R}(w)
    \rt)
    ~\de \nu(x)\de \nu(y)
    .
    \end{aligned}
    \end{equation}
    For the inner integral, let 
    \[
    C_{d,k}=R^{-k}\int w_1^k ~\de \mu^{\sph}_{d,R}(w)
    \]
    where $w_1$ is the first coordinate of $w=(w_1,\dots,w_d)\in \bbS_{d,R}$. Of course, $C_{d,2j+1}=0$ while $C_{d,2j}\in (0,1)$. Then for any $z\in\bbR^d$:
    \begin{align*}
    \int
    e^{\la z,w\ra}
    ~\de \mu^{\sph}_{d,R}(w)
    &=
    \sum_{k\geq 0}
    \int
    \frac{\la z,w\ra^k}{k!}
    ~\de \mu^{\sph}_{d,R}(w)
    =
    \sum_{k\geq 0}
    \frac{C_{d,k} R^k \|z\|^k}{k!}
    =
    \sum_{j\geq 0}
    \frac{C_{d,2j} R^{2j}\|z\|^{2j}}{(2j)!}.
    \end{align*}
    Recalling \eqref{eq:change-order}, we take $z=x+y$ and observe that $\|z\|^2 = 2(r^2+\la x,y\ra)$.
    Combining,  
    \[
    \int
    \rho_{\nu*\cN(0,I_d)}(w)^2
    ~\de \mu^{\sph}_{d,R}(w)
    =
    \iint 
    \Big(
    \sum_{j\geq 0}
    \frac{C_{d,2j} 2^j R^{2j}\big(r^2+\la x, y\ra\big)^j}{(2j)!}
    \Big)
    ~\de \nu(x)\de \nu(y).
    \]
    Recalling that $C_{d,2j}\in (0,1)$ for all $j$, we find that the right-hand side has all coefficients strictly positive as a power series in $\la x,y\ra$.
    Moreover it converges absolutely on all of $\bbR$ by inspection as a power series in $w=r^2+\la x,y\ra$, and global absolute convergence is invariant under an affine change of variable.
    Thus Proposition~\ref{prop:welch} applies and concludes the proof.
\end{proof}

\section{Approximate Stationarity Conditions}

Here we explain how to compute an approximate stationary point (in the relevant sense) for the concave function $\ell_X:\cP(Z_{\eps})\to\bbR$ in a provably efficient manner.
Note that Proposition~\ref{prop:PW20} implies $|\supp(\wh\pi)|\leq O(L^2)\ll |Z_{\eps}|\asymp L/\eps$, i.e. $\wh\pi$ is a sparse vector in $\cP(Z_{\eps})$. 
Indeed, for us the relevant notion of approximate-stationary point $\pi$ will be that
\[
\frac{\de}{\de s}\big(\ell_X((1-s)\pi+s\pi'\big)\Big|_{s=0}
\]
is small for all $\pi'\in \cP(Z_{\eps})$. 
This definition is sensitive to the support of $\pi$, and in particular it is easier to satisfy when small atoms of $\pi$ are rounded down to $0$.
This is precisely what we do below, based on the Frank--Wolfe conditional gradient method.
To start, we set $\pi_0=\delta_0$ and iteratively define:
\begin{equation}
\label{app-eq:frank-wolfe}
\begin{aligned}
    \pi^{(t+1)}
    &=
    \frac{t \pi^{(t)}+2\delta_{y_t}}{t+2},
    \\
    y_t
    &=
    \argmax_{y\in Z_{\eps}}D_{\pi^{(t)},X}(y).
\end{aligned}
\end{equation}
In using this manifestation of the Frank--Wolfe algorithm, we implicitly equip $\cP(Z_{\eps})$ with the Euclidean norm on its finite sequence of probability mass values (and also negated $\ell_X$ to make it convex). In particular the Wasserstein distance does not enter here.
We also point out that \cite[Theorem 2]{jaggi2013revisiting}, which we rely on below to understand \eqref{app-eq:frank-wolfe}, applies even if $y_t$ is only an approximate maximizer of $D_{\pi^{(t)},X}$. 
(This could easily be incorporated into Lemma~\ref{app-lem:Frank-Wolfe-main} to ensure that approximately maximal $y$ suffice in \eqref{app-eq:frank-wolfe}.)

Next we modify $\pi^{(t)}$ to $\breve\pi^{(t)}$; this will ensure $D_{\breve\pi^{(t)},X}(y)-1$ is close to $1$ for all $y\in\supp(\breve\pi^{(t)})$. 
Define for $\iota>0$ the subset $R_{\eps,\iota}=\{y\in Z_{\eps}~:~p_{\pi^{(t)}}(y)\leq \iota\}$. 
For $\iota\leq \frac{\eps}{4L}$, 
\begin{equation}
    \label{app-eq:sum-S-small}
    \sum_{y\in R_{\eps,\iota}} p_{\pi^{(t)}}(y)\leq 3\iota L/\eps<1/2.
\end{equation}
Thus we may define another probability measure
\begin{equation}
    \label{app-eq:rescaled-pi-t}
    \breve\pi^{(t)}(y)=
    \begin{cases}
    0,\quad y\in R_{\eps,\iota}
    \\
    \frac{\pi^{(t)}(y)}{1-\sum_{y\in R_{\eps,\iota}} p_{\pi^{(t)}}(y)},\quad y\notin R_{\eps,\iota}.
    \end{cases}
\end{equation}

\begin{lemma}
\label{app-lem:Frank-Wolfe-main}
    Fix any $X\in [-L,L]^n$. For $\eps$ small enough depending only on $L$ and for $t\geq \eps^{-8}$, the probability measure $\breve\pi^{(t)}$ satisfies:
    \begin{align}
    \label{app-eq:D-LB-approx}
    D_{\breve\pi^{(t)},X}(y)-1
    &\geq 
    -O\lt(
    \frac{e^{4L^2}}{\eps^{1/2}\sqrt[4]{t}}
    \rt),
    \quad\forall
    y\in\supp(\breve\pi^{(t)}),
    \\
    \label{app-eq:D-UB-approx}
    D_{\breve\pi^{(t)},X}(y)-1
    &\leq 
    O\lt(
    \frac{e^{4L^2}}{\eps^{1/2}\sqrt[4]{t}}
    \rt)
    ,
    \quad
    \forall y\in Z_{\eps}.
    \end{align}
\end{lemma}

\begin{proof}
    We apply \cite[Theorem 2]{jaggi2013revisiting} to the algorithm \eqref{app-eq:frank-wolfe}.
    The conclusion is a bound 
    \[
    g(\pi^{(t)})
    \leq 
    \frac{10\,
    \diam(\cP(Z_{\eps}))^2 
    \Lip(\nabla \ell_X)}{t+2}.
    \]
    where (using \eqref{eq:ell-deriv-formula}) we have by definition
    \begin{equation}
    \label{app-eq:g-bound}
    g(\pi^{(t)})
    =
    \max_{y\in Z_{\eps}}
    \big(D_{\pi^{(t)},X}(y)-1\big)\cdot (1-p_{\pi^{(t)}}(y)).
    \end{equation}
    It is easy to see that $\diam(\cP(Z_{\eps}))^2\leq 2$.
    Meanwhile, $\nabla \ell_X(\pi)=D_{\pi,X}(\cdot)$.
    Combining \eqref{eq:D-lip} and the fact that $\bbW_1(\pi,\pi')\leq 2L\|\pi-\pi'\|_{TV}$, it follows that the Lipschitz constant $\Lip(\nabla \ell_X)$ is at most $O(Le^{4L^2})$.
    Altogether we find that
    \begin{equation}
    \label{app-eq:D-1-bound}
    D_{\pi^{(t)},X}(y)-1
    \leq 
    \frac{O(L e^{4L^2})}{(t+2)(1-p_{\pi^{(t)}}(y))}.
    \end{equation}

    Additionally, \eqref{eq:D-sum-formula} implies that 
    \[
    D_{\pi^{(t)},X}(y)-1\leq \frac{1-p_{\pi^{(t)}}(y)}{p_{\pi^{(t)}}(y)}.
    \]
    Next we combine these two estimates. 
    It is easy to see that for small $\beta$, one has $\min(\frac{p-1}{p},\frac{\beta}{1-p})\leq O(\sqrt \beta)$ for all $p\in [0,1]$ by casework on the event $p\leq 1-\sqrt{\beta}$. 
    Hence for (say) $t\geq e^{6L^2}$ and all $y\in Z_{\eps}$,
    \begin{equation}
    \label{app-eq:D-UB}
    D_{\pi^{(t)},X}(y)-1
    \leq 
    \max\lt(
    \frac{O(L e^{4L^2})}{(t+2)(1-p_{\pi^{(t)}}(y))},
    \frac{1-p_{\pi^{(t)}}(y)}{p_{\pi^{(t)}}(y)}
    \rt)
    \leq 
    O\lt(
    \sqrt{\frac{L e^{4L^2}}{t+2}}
    \rt).
    \end{equation}

    Next we turn to $\breve\pi^{(t)}$, showing it approximately preserves the preceding upper bound for all $y\in Z_{\eps}$.
    Note that $\|\pi^{(t)}-\breve\pi^{(t)}\|_{TV}\leq 3\iota L/\eps$, hence $\bbW_1(\pi^{(t)},\breve\pi^{(t)})\leq 6\iota L^2/\eps$.
    Using again \eqref{eq:D-sum-formula} and the preceding upper bound, we find that for $y\notin R_{\eps,\iota}$:
    \begin{align*}
    D_{\pi^{(t)},X}(y)-1
    &\geq 
    -p_{\pi^{(t)}}(y)^{-1}\cdot O\lt(
    \sqrt{\frac{L e^{4L^2}}{t}}
    \rt)
    \\
    D_{\breve\pi^{(t)},X}(y)-1
    &\geq 
    \big(D_{\pi^{(t)},X}(y)-1\big)
    +
    \big(D_{\breve\pi^{(t)},X}(y)-D_{\pi^{(t)},X}(y)\big)
    \stackrel{\eqref{eq:D-lip}}{\geq}
    -O\lt(
    \frac{Le^{2L^2}}{\iota\sqrt{t}}
    +
    \frac{e^{4L^2} \iota L^2}{\eps}
    \rt).
    \end{align*}
    Taking $\iota= e^{-L^2}t^{-1/4}\eps^{1/2}$, we obtain 
    \[
    D_{\breve\pi^{(t)},X}(y)-1
    \geq 
    -O\lt(
    \frac{e^{4L^2}}{\eps^{1/2}\sqrt[4]{t}}
    \rt),
    \quad\forall
    y\in\supp(\breve\pi^{(t)}),
    \]
    under the condition $t\geq \eps^{-8}$ for $\eps$ sufficiently small (to ensure \eqref{app-eq:sum-S-small}).
    Finally bounding $|D_{\breve\pi^{(t)},X}(y)-D_{\pi^{(t)},X}(y)|$ in the same way, \eqref{app-eq:D-UB} becomes 
    \[
    D_{\breve\pi^{(t)},X}(y)-1
    \leq 
    O\lt(
    \frac{e^{4L^2}}{\eps^{1/2}\sqrt[4]{t}}
    \rt)
    ,
    \quad
    \forall y\in Z_{\eps}.
    \qedhere
    \]
\end{proof}

We can now show an approximate version of Lemma~\ref{lem:eps-delta-bound}.

\begin{lemma}
\label{app-lem:FW-eps-delta-bound}
    For $\eps$ small enough depending on $L$ there exists an algorithm with complexity $O(Ln\eps^{-11})$ which returns $\breve\pi_{\eps}\in\cP(Z_{\eps})$ obeying for $\delta = CL^2 e^{4L^2}\eps^2$:
    \begin{enumerate}[label=(\alph*)]
        \item 
        \label{app-it:D-hatpieps-0}
        $D_{\breve\pi_{\eps},X}(y)\geq 1-\delta$ for all $y\in \supp(\breve\pi_{\eps})$.
        \item 
        \label{app-it:D-hatpieps-delta}
        $\max\limits_{y\in \bbR} D_{\breve\pi_{\eps},X}(y)
        \leq 1+\delta$.
        \item 
        \label{app-it:D-FW-approximate-min}
        $\ell_X(\wh\pi)-\ell_X(\breve\pi_{\eps})\leq \delta$.
    \end{enumerate}
\end{lemma}

\begin{proof}
    Note that each Frank--Wolfe iteration \eqref{app-eq:frank-wolfe} requires $O(Ln/\eps)$ operations to find the maximal $y\in Z_{\eps}$. 
    Taking $\breve\pi_{\eps}=\breve\pi^{(t)}$ as defined in Lemma~\ref{app-lem:Frank-Wolfe-main}, the first two parts follow by setting $t=\eps^{-10}$, except that the upper bound must hold for all $y\in \bbR$ instead of just $y\in Z_{\eps}$. 
    The extension to $[-L,L]$ follows immediately via \eqref{eq:D-bound} (and after e.g. doubling $\delta$), and this suffices by Lemma~\ref{lem:support-within-[-L,L]}.
    The last assertion follows by \eqref{app-eq:g-bound} since $1-p_{\pi^{(t)}}(y)\leq 1$, since $g(\pi^{(t)})$ is a certified (dual) upper bound on the suboptimality of $\pi^{(t)}$.
\end{proof}

Below we let $\breve\pi_{\eps}$ be the approximation to $\wh\pi$ guaranteed by Lemma~\ref{app-lem:FW-eps-delta-bound}.
We first show that Proposition~\ref{prop:apriori-approx-new} extends to this setting, with $\breve\pi_{\eps}$ in place of $\wh\pi_{\eps}$ .
Recall the definitions of $A_X,B_X$ in and just below \eqref{eq:AX}.

\begin{proposition}
\label{app-prop:apriori-approx}
    For $\eps$ small enough, we have:
    \begin{align*}
    d(y,\supp(\breve\pi_{\eps}))\leq a_X/2
    \implies 
    D''_{\breve\pi_{\eps},X}(y)
    &\leq
    -A_X/2;
    \\
    d(y,\supp(\breve\pi_{\eps}))\geq a_X/2
    \implies 
    D_{\breve\pi_{\eps},X}(y)
    &\leq 1-\frac{B_X}{2}.
    \end{align*}
\end{proposition}

\begin{proof}
    Using Lemma~\ref{app-lem:FW-eps-delta-bound}\ref{app-it:D-FW-approximate-min}, the proof of
    Proposition~\ref{prop:W1-approx-inefficient} implies $\lim_{\eps\to 0}\bbW_1(\wh\pi,\breve\pi_{\eps})=0$ uniformly in the choice of $Z_{\eps}$.
    The rest is identical to Proposition~\ref{prop:apriori-approx-new}.
\end{proof}

\begin{proposition}
\label{app-prop:hat-pi-eps-support-simple}
    For $\eps$ sufficiently small depending on $X$, any pair of points in $\supp(\breve\pi_{\eps})$ within distance $a_X/5$ are within distance $O(\sqrt{\delta/A_X})$.
\end{proposition}

\begin{proof}
    By Proposition~\ref{app-prop:apriori-approx}, $D''_{\breve\pi_{\eps},X}$ is negative on the interval between any two such points (for small enough $\eps)$. 
    Lemma~\ref{app-lem:FW-eps-delta-bound} and simple calculus completes the proof.
\end{proof}

For $\eps$ small enough that $\sqrt{\delta/A_X}\ll a_X$, Proposition~\ref{app-prop:hat-pi-eps-support-simple} implies that the graph of atoms in $\breve\pi_{\eps}$ under the distance-at-most-$a_X/5$ graph is a union of cliques.
We form $\mathring\pi_{\eps}$ by taking the weighted average of each such clique, i.e. replacing $p_i \delta_{x_i},\dots, p_{j}\delta_{x_j}$ by $\wt p_i \delta_{\wt x_i}$ for
\begin{equation}
\label{eq:wt-pi-appendix}
\wt p_i=p_i+\dots+p_j,
\quad\quad
\wt x_i=\frac{p_i x_i + \dots +p_{j}x_j}{x_i+\dots +x_{j}}.
\end{equation}
It again follows from Proposition~\ref{app-prop:hat-pi-eps-support-simple} and the Wasserstein convergence $\breve\pi_{\eps}\to\wh\pi$ that $|\supp(\mathring\pi_{\eps})|=|\supp(\wh\pi)|$ for $\eps$ small enough, and that
\begin{equation}
\label{app-eq:hausdorff-convergence}
    \lim_{\eps\to 0}
    d_{\cH}
    \big(
    \supp(\mathring\pi_{\eps}),\supp(\wh\pi)
    \big)
    =0.
\end{equation}

\begin{proposition}
\label{app-prop:cert-main}
    For $\eps\leq \eps_0(X)$, the conditions of Proposition~\ref{prop:cert} apply to $\mathring\pi_{\eps}$ with:
    \begin{align*}
        \wt\delta
        &=
        O(L^2 e^{4L^2}\delta/A_X)
        =
        O(\eps^2 L^4 e^{8L^2}/A_X),
        \\
        c_1&=\sqrt[4]{CL\delta},
        \\
        c_2&=A_X \sqrt{\delta},
        \\
        \eta&=2c_1,
        \\
        \lambda_{\eps}\triangleq\lambda(\mathring\pi_{\eps})
        &=
        \lambda(\wh\pi)\pm o_{\eps}(1).
    \end{align*}
    In particular, the estimator $\mathring\pi_{\eps}$ obeys certifiable bounds of the form $\bbW_1(\wh\pi,\mathring\pi_{\eps})\leq O_X(\eps^{1/4})$ as well as $|\supp(\wh\pi)|\geq |\supp(\mathring\pi_{\eps})|$.
\end{proposition}

\begin{proof}
    The bound on $\delta$ for $\breve\pi_{\eps}$ is given in Lemma~\ref{app-lem:FW-eps-delta-bound}. 
    To extend it to $\mathring\pi_{\eps}$, the analogous argument in Proposition~\ref{prop:cert-main-new} reduces to showing that for $|x_i-x_{i+1}|\leq O(\sqrt{\delta/A_X})$,
    \[
    \lt|
    \lt(\frac{p_j}{p_i+\dots+p_{j}}\rt) e^{-|x_i-y|^2/2}
    +
    \dots
    +
    \lt(\frac{p_{j}}{p_i+\dots+p_{j}}\rt)
    e^{-|x_{j}-y|^2/2}
    -
    e^{-|\wt x_{i}-y|^2/2}
    \rt|
    \stackrel{?}{\leq}
    O(\delta/A_X).
    \]
    Taylor's theorem applied to $f(x)=e^{-|x|^2/2}$ easily again gives the bound.
    
    Continuing, we claim that whenever $2\eps \leq c_1\leq \frac{a_X}{2}$, we may take 
    \begin{equation}
    \label{app-eq:c1-c2-claim}
    c_2
    =
    A_X\cdot\frac{(c_1-\eps)^2-O(\sqrt{\delta})}{4}.
    \end{equation} 
    in Proposition~\ref{prop:cert}.
    Indeed, if 
    \[
    d(y,\supp(\mathring\pi_{\eps}))\geq c_1,
    \]
    then $d(y,\breve y)\geq c_1-\eps$ for some $\breve y\in \supp(\breve\pi_{\eps})$, so in particular $D_{\breve\pi_{\eps},X}(\breve y)\geq 1-\delta$ by Lemma~\ref{app-lem:FW-eps-delta-bound}.
    As an intermediate step, we upper bound $|D'_{\breve\pi_{\eps},X}(\breve y)|$ using Lemma~\ref{app-lem:FW-eps-delta-bound}\ref{app-it:D-hatpieps-delta} and bound \eqref{eq:D-bound} with $j=2$.
    In particular, optimizing over $y\in \bbR$ in the first line to obtain the inequality $(\dagger)$, we find
    \begin{align*}
    D_{\breve\pi_{\eps},X}(y)
    &\geq 
    D_{\breve\pi_{\eps},X}(\breve y)
    +
    D'_{\breve\pi_{\eps},X}(\breve y)(y-\breve y)
    -
    C_2 L^2 e^{4L^2}(y-\breve y)^2,\quad\forall y\in\bbR
    \\
    \implies
    \max\limits_{z\in\bbR}
    D_{\breve\pi_{\eps},X}(\breve y)
    +2\delta
    &\stackrel{Lem.~\ref{app-lem:FW-eps-delta-bound}}{\geq}
    D_{\breve\pi_{\eps},X}(z)
    \stackrel{(\dagger)}{\geq}
    D_{\breve\pi_{\eps},X}(\breve y)
    + \Omega\lt(\frac{|D'_{\breve\pi_{\eps},X}(\breve y)|^2}{L^2 e^{4L^2}}\rt)
    \\
    \implies
    |D'_{\breve\pi_{\eps},X}(\breve y)|
    &\leq 
    O(Le^{2L^2} \sqrt{\delta})
    .
    \end{align*}
    Using Proposition~\ref{app-prop:apriori-approx} we now obtain \eqref{eq:c1-c2-claim}:
    \begin{align*}
    D_{\mathring\pi_{\eps},X}(y)
    &\leq 
    D_{\breve\pi_{\eps},X}(\breve y)
    +
    O\big(Le^{L^2}\sqrt{\delta} |y-\breve y|\big)
    -
    \frac{A_X(c_1-\eps)^2 }{4}
    \\
    &\leq 
    1-\lt(\frac{A_X(c_1-\eps)^2}{4}-O\big(Le^{L^2}a_X\sqrt{\delta}\big)\rt)
    =
    1-c_2.
    \end{align*}
    In particular for $\eps$ small enough compared to $A_X$ and $C$ an absolute constant, we can set 
    \[
    c_1=\sqrt[4]{CL\delta},\quad
    c_2=A_X \sqrt{\delta}
    \implies \eta=c_1+\frac{L\wt\delta}{c_2}\leq 2\sqrt[4]{CL\delta}.
    \]
    Further, Proposition~\ref{prop:lambda-positive} ensures that \eqref{eq:eta-bound} holds for $\eps$ small enough, since the value $\lambda(\mathring\pi_{\eps})$ is easily seen to converge to that of $\wh\pi$.
    We conclude that Proposition~\ref{prop:cert} suffices to certify bounds of the form $\bbW_1(\wh\pi,\mathring\pi_{\eps})\leq O_X(\eps^{1/4})$.
    Finally, this implies $|\supp(\wh\pi)|\geq |\supp(\mathring\pi_{\eps})|$, again for $\eps$ small enough depending on $X$ (and in particular $\wh\pi$).
\end{proof}

\begin{proposition}
\label{app-prop:cert-main-support}
    For $\eps$ small enough, Proposition~\ref{prop:support-size-cert} certifies $|\supp(\wh\pi)|\leq |\supp(\mathring\pi_{\eps})|$.
\end{proposition}

\begin{proof}
    Follows exactly as in the Proof of Theorem~\ref{thm:certify-intro}.
\end{proof}

\begin{proof}[Proof of Theorem~\ref{thm:certify-intro-approx}, except for the Shub--Smale property]
   The result follows by Propositions~\ref{app-prop:cert-main}, \ref{app-prop:cert-main-support}, and \ref{prop:cert_lb}.
    See Appendix~\ref{app:shub-smale} for the Shub--Smale property.
\end{proof}

\subsection{Finite $S$ Case}

Here we prove Theorem~\ref{thm:certify-finite-S}, except for the Shub--Smale property which is addressed in Appendix~\ref{app:shub-smale}.
We use exactly the same version of Frank--Wolfe as above, but restricted to $S$ rather than $Z_{\eps}$.
Thus we take
\begin{equation}
\label{app-eq:S-frank-wolfe}
\begin{aligned}
    \pi_S^{(t+1)}
    &=
    \frac{t \pi_S^{(t)}+2\delta_{y_t}}{t+2},
    \\
    y_t
    &=
    \argmax_{y\in S}D_{\pi_S^{(t)},X}(y).
\end{aligned}
\end{equation}
 
Define for $\iota>0$ the subset $R_{S,\eps,\iota}=\{y\in S~:~p_{\pi^{(t)}}(y)\leq \iota\}$. 
For $\iota\leq \frac{\eps}{4L}$, 
\begin{equation}
    \label{app-eq:S-sum-R-small}
    \sum_{y\in R_{S,\eps,\iota}} p_{\pi_S^{(t)}}(y)\leq 3\iota L/\eps<1/2.
\end{equation}
Thus we may define another probability measure
\begin{equation}
    \label{app-eq:S-rescaled-pi-t}
    \breve\pi_S^{(t)}(y)=
    \begin{cases}
    0,\quad y\in R_{S,\eps,\iota}
    \\
    \frac{\pi_S^{(t)}(y)}{1-\sum_{y\in R_{S,\eps,\iota}} p_{\pi_S^{(t)}}(y)},\quad y\notin R_{S,\eps,\iota}.
    \end{cases}
\end{equation}

The proof of Lemma~\ref{app-lem:Frank-Wolfe-main} extends unchanged to give the following.
Fix any $X\in [-L,L]^n$ and $S\subseteq \mathbb R$ finite. For $\eps$ small enough depending only on $L$ and for $t\geq \eps^{-8}$, the probability measure $\breve\pi_{S,\eps}=\breve\pi^{(t)}$ satisfies:
\begin{align}
\label{app-eq:S-D-LB-approx}
D_{\breve\pi_{S,\eps},X}(y)-1
&\geq 
-O\lt(
\frac{e^{4L^2}}{\eps^{1/2}\sqrt[4]{t}}
\rt)
\geq 
-C(L) \eps^{3/2}
,
\quad\forall
y\in\supp(\breve\pi_S^{(t)}),
\\
\label{app-eq:S-D-UB-approx}
D_{\breve\pi_{S,\eps},X}(y)-1
&\leq 
O\lt(
\frac{e^{4L^2}}{\eps^{1/2}\sqrt[4]{t}}
\rt)
\leq C(L) \eps^{3/2}
,
\quad
\forall y\in S,
\\
\ell_X(\wh\pi_S)-\ell_X(\breve\pi_{S,\eps})&\leq C(L)\eps^2
.
\end{align}

\begin{proof}[Proof of Theorem~\ref{thm:certify-finite-S}, except for the Shub--Smale property]
    Similarly to the main case, let 
    \[
    B_{S,X}=1-\max_{y\in S\backslash \supp(\wh\pi_S)} D_{\wh\pi_S,X}(y)>0.
    \]
    Since $\wh\pi_S$ is unique, we have $\breve\pi_{S,\eps}\to \wh\pi_S$ in Wasserstein as $\eps\to 0$, which implies that for $\eps$ small enough we have $D_{\breve\pi_{S,\eps},X}(y)\leq 1-\frac{B_{S,X}}{2}$ for all $y\in S\backslash \supp(\wh\pi_S)$.

    The remaining proof is similar to before, and we just give an outline. 
    The initial steps follow the proof of Proposition~\ref{prop:cert}.
    First one can certify that $\wh\pi$ has at most $\eps$ mass outside $\supp(\breve\pi_{S,\eps})$, using the bound on $\|\ell''\|_{\infty}$ for $\ell(t)=\ell_X((1-t)\breve\pi_{S,\eps}+t\wh\pi_S)$.
    This is because
    \[
    \ell'(0)=\mathbb E^{y\sim\wh\pi_S}[D_{\breve\pi_{S,\eps},X}(y)]
    .
    \]
    Again \eqref{eq:loss-hessian-condition} holds for some positive $\lambda=\lambda(S,X)$ (with $S$ and $\wh\pi_{S}$ in place of $Z_{\eps}$ and $\tilde\pi_{\eps}$) by uniqueness of $\wh\pi_S$.
    Then \eqref{eq:loss-hessian-condition} will hold with constant $\lambda/2$ for $\breve\pi_{S,\eps}$ with small enough $\eps$.
    This allows to certify that all approximate local maxima of $\ell_X$ supported on $\supp(\breve\pi_{S,\eps})$ are within $O(\lambda)$ of $\breve\pi_{S,\eps}$ (using now the fact that $\ell$ has $3$ bounded derivatives).
    Since we could also certify above that $\wh\pi$ has at most $\eps$ mass outside $\supp(\breve\pi_{S,\eps})$, these certificates thus combine to certify that $\bbW_1(\wh\pi_S,\breve\pi_{S,\eps})\leq C(S,X)\eps$ almost surely for small enough $\eps$ (again similarly to Proposition~\ref{prop:cert}).

    Since for small $\eps$ we will have $D_{\breve\pi_{S,\eps},X}(y)\leq 1-\frac{B_{S,X}}{2}$ for all $y\in S\backslash \supp(\wh\pi_S)$, applying \eqref{eq:D-lip} to the Wasserstein bound certifies that 
    \[
    D_{\wh\pi_{S},X}(y)\leq 1-\frac{B_{S,X}}{4},\quad\forall y\in S\backslash \supp(\wh\pi_S).
    \] 
    In particular, this certifies that $\supp(\wh\pi_S)\subseteq \supp(\breve\pi_{S,\eps})$.
    The opposite inclusion is easily certified from $\bbW_1(\wh\pi_S,\breve\pi_{S,\eps})\leq C(S,X)\eps$ once $\eps$ is small compared to $p_* d_*$, where $p_*$ is the smallest atom size in $\wh\pi_S$ and $d_*$ is the minimum distance between points in $S$.
    Finally, a $\bbW_1$ bound immediately gives a $d_{\Pi_k(S)}$ bound as well.
\end{proof}

\section{Positive Probability to be $k$-Atomic}
\label{app:positive-prob-k-atoms}

We show that for any $n\geq k\geq 1$, there is a positive probability to have $\hat\pi\in\Pi_k$ with $L\leq O(k\sqrt{\log k})$, as mentioned in Remark~\ref{rem:generic-relaxation-and-conditioning}.
Thus shows that conditioning simultaneously on the events $\hat\pi\in\Pi_k$ and $\max_{i,j}(x_i-x_j)\leq cn^{1/4}$ is possible, for $n$ polynomially large in $k$.

We partition $[n]$ into $k$ parts $S_1,\dots,S_k$ of sizes $\lfloor n/k\rfloor$ or $\lceil n/k \rceil$.
We fix a large absolute constant $C$, and set $\tilde x_i=Ci\sqrt{\log(k+1)}$ for each $1\leq i\leq k$.
Let us suppose that $|x_a-\tilde x_i|\leq 0.1$ holds for all $a\in S_i$. This is a positive probability event, where roughly speaking the datapoints are clustered into groups of approximately equal size.
We will show that this implies $|\supp(\hat\pi)|=k$. Precisely, $\hat\pi$ has exactly $1$ atom close to each $\tilde x_i$, and no other atoms.

\begin{lemma}
\label{lem:k-pt-npmle-near-only-data}
For each $i\in [k]$, $\supp(\hat \pi) \subseteq \bigcup_i[\tilde x_i-0.2,\tilde x_i+0.2]$. 
\end{lemma}

\begin{proof}
Recall that the function $e^{-y^2/2}$ is convex outside of $[-1,1]$. 
First, suppose $\hat\pi$ has an atom $p\delta_y$ with $\min_i |y-\tilde x_i|\geq 2$.
Then we may replace $p\delta_y$ by $\frac{p\delta_{y+0.5}+p\delta_{y-0.5}}{2}$, and this will increase the value of $P_{\pi}(x_a)$ for each $a\in [n]$. 
This contradicts optimality of $\hat\pi$ in \eqref{eq:NPMLE-def}.

Next, suppose $\hat\pi$ has an atom $p\delta_y$ with $\min_i |y-\tilde x_i|\in [0.2,2]$.
Without loss of generality, suppose $y-\tilde x_i\in [0.2,2]$.
Then we can replace $p\delta_y$ in $\hat\pi$ by 
\[
p(1-c)\delta_{y+1}+pc\delta_{y-0.1}.
\]
For $c$ a small enough absolute constant, this will increase $P_{\pi}(x_a)$ for all $a\in S_i$.
Then choosing $C$ sufficiently large ensures $P_{\pi}(x_a)$ increases for all other $a\notin S_i$ as well. 
\end{proof}

\begin{lemma}
\label{lem:k-pt-npmle-near-all-data}
For each $i\in [k]$, $\hat\pi([\tilde x_i-0.2,\tilde x_i+0.2])\geq \frac{1}{20k}$. 
\end{lemma}

\begin{proof}
    We claim that if $\hat\pi$ violates this condition for some $i$, then $D_{\hat\pi,X}(\tilde x_i)>1$, which contradicts Proposition~\ref{prop:stationarity-conditions}.
    Indeed in this case, Lemma~\ref{lem:k-pt-npmle-near-only-data} implies the set 
    $\hat \pi([\tilde x_i-6\sqrt{\log(k+1)},\tilde x_i+ 6\sqrt{\log(k+1)}])\leq \frac{1}{20k}$.
    Then it easily follows that $P_{\hat\pi}(x_{a})\leq \frac{1}{10k}$ for each $a\in S_i$.
    We conclude that:
    \[
    D_{\hat\pi,X}(\tilde x_i)
    \geq 
    \frac{10k|S_i|}{n}
    \min_{x\in [\tilde x_i-0.1,\tilde x_i+0.1]}
    e^{-|x-\tilde x_i|^2/2}/\sqrt{2\pi}
    \geq 
    5e^{-0.005}/\sqrt{2\pi}
    >1.
    \]
    This is the desired contradiction.
\end{proof}

\begin{lemma}
\label{lem:k-pt-npmle-uniquely-near-data}
For each $i\in [k]$, $|\supp(\hat\pi)\cap [\tilde x_i-0.2,\tilde x_i+0.2]|=1$.
Thus $|\supp(\hat\pi)|=k$.
\end{lemma}

\begin{proof}
We use the previous lemmas to show $D_{\hat\pi,X}$ is concave on each interval $[\tilde x_i-0.2,\tilde x_i+0.2]$.
Fixing $i$, we decompose
\[
D_{\hat\pi,X}(y)n\sqrt{2\pi}
=
\sum_{j=1}^k
\sum_{a\in S_j} 
e^{-|x_a-y|^2/2}/P_{\hat\pi}(x_a)
\triangleq 
\sum_{j=1}^k
D_j(y)
.
\]
We first study $D_i(y)$.
Trivially, $P_{\hat\pi}(x_a)\sqrt{2\pi}\leq 1$ for all $a\in [n]$.
Furthermore, the function $f(y)=e^{-y^2/2}$ satisfies $f''(y)\leq 1/100$ on $|y|\leq 0.2$.
Therefore one easily finds
\[
D_i''(y)
\leq 
-\frac{|S_i|}{100}
\leq 
-\frac{n}{200k},
\quad
\forall y\in [\tilde x_i-0.2,\tilde x_i+0.2].
\]

Next we study $D_j''(y)$ for $j\neq i$.
Here, Proposition~\ref{lem:k-pt-npmle-near-all-data} implies that $P_{\hat\pi}(x_a)\geq \frac{1}{100k}$ for all $a\in [n]$.
Therefore for all $y\in [\tilde x_i-0.2,\tilde x_i+0.2]$:
\[
D_j''(y)
\geq 
-100k |S_j| e^{-C^2\log(k+1)/2}
\geq 
-\frac{S_j}{100(k+1)^{10}}
\geq 
-\frac{n}{10(k+1)^8}.
\]
Summing over $j\neq i$, it is clear that $D_i''(y)$ dominates, so $D_{\hat\pi}''(y)<0$ for all $y\in [\tilde x_i-0.2,\tilde x_i+0.2]$.
By Proposition~\ref{prop:stationarity-conditions}, $\hat\pi$ is supported within the set of global maxima of $D_{\hat\pi}$, which completes the proof.
\end{proof}

Combining the lemmas above immediately yields the claimed result.

\subsection{Example of Non-Generic Behavior}

Here we prove Theorem~\ref{thm:absolute-continuity-false}.
With $n=mk$, we restrict to the positive probability event that $X$ satisfies the conditions of Lemmas~\ref{lem:k-pt-npmle-near-only-data}, \ref{lem:k-pt-npmle-near-all-data}, \ref{lem:k-pt-npmle-uniquely-near-data}, i.e. $|x_a-\tilde x_i|\leq 0.1$ for all $a\in S_i$. 
For convenience, we say such $X$ is $k$-good; this condition implies $\wh\pi\in \Pi_k$. 
We claim that conditioned on $(x_{m+1},\dots,x_n)$, the conditional law of $\wh\pi$ does not admit a density on $\Pi_k$.
The proof is motivated by Remark~\ref{rem:fourier-analysis-tight}. The main remaining step is the smooth dependence of $\wh\pi$ on $(x_1,\dots,x_k)$.
To show it, we rely on the following standard version of the implicit function theorem.

\begin{lemma}
\label{lem:implicit-function-theorem}
    Let $U,V\subseteq\mathbb R^d$ be open sets and $F:U\times V\to\mathbb R^d$ a smooth function.
    Let $u_*\in U$ and $v_*\in V$ be such that $\partial_v F(u_*,v_*)\in\mathbb R^{d\times d}$ is invertible. 
    Then on a neighborhood $U_0\subseteq U$ of $u_*$, there exists a smooth function $v:U_0\to V$ such that 
    \[
    F(u,v(u))=F(u_*,v_*)
    \]
    holds for all $u\in U_0$.
\end{lemma}

The following lemma shows that Theorem~\ref{thm:main} parts \ref{it:no-coincidence-1} and \ref{it:no-coincidence-2} imply Theorem~\ref{thm:landscape-consequences}\ref{it:local-smoothness}.
Theorem~\ref{thm:landscape-consequences}\ref{it:local-strong-convexity} also follows from the proof, as the Hessian of $\ell_X$ is the Jacobian of $\gamma_k$ within, which is shown to be strictly negative definite.

\begin{lemma}
\label{lem:smooth-dependence-general}
    Suppose $\wh\pi$ is the NPMLE for $X$ and satisfies the conclusions of Theorem~\ref{thm:main} parts \ref{it:no-coincidence-1} and \ref{it:no-coincidence-2}.
    Then for $\tilde X$ in a sufficiently small neighborhood of $X$, the function $\tilde X\mapsto \wh\pi(\tilde X))$ has image in $\Pi_k$, and is smooth.
    Here we consider $\Pi_k$ to be an open subset of a $2k-1$ dimensional real vector space.
\end{lemma}

\begin{proof}
    It is clear that $\wh\pi(\tilde X)\in\Pi_k$ for $\tilde X$ close to $X$.
    We will apply Lemma~\ref{lem:implicit-function-theorem} with $d=2k$.
    We consider the Jacobian of the map
    \[
    \pi=\sum_{i=1}^k p_i\delta_{y_i}
    \mapsto 
    \gamma_k(\pi)
    =
    \Big(
    D_{\pi}(y_1)-1,
    \dots,
    D_{\pi}(y_k)-1,
    D_{\pi}'(y_1),
    \dots,
    D_{\pi}'(y_k)
    \Big)
    \in\mathbb R^{2k}.
    \]
    Here we temporarily relax the constraint $\sum_{i=1}^k p_i=1$ so the Jacobian is a $2k\times 2k$ matrix $J_{2k}$.
    We show below that this Jacobian is invertible.
    Then Lemma~\ref{lem:implicit-function-theorem} applied to $\gamma_k(\pi,\tilde X)$ implies existence of a smooth function $\pi(\tilde X)$ such that $\gamma_k(\pi(\tilde X),\tilde X)=\vec 0$ in a neighborhood of $(\wh\pi,X)$.
    Further, any such solution satisfies $\sum_{i=1}^k p_i D_{\pi}(y_i)=\sum_{i=1}^k p_i=1$, so $\pi(\tilde X)$ is always a probability measure (despite the relaxation to $2k$ dimensions in applying the inverse function theorem).

    We turn to invertibility of the Jacobian. The entries of $-J_{2k}$ are, for $1\leq i,j\leq k$:
    \begin{equation}
    \label{eq:2k-by-2k-jacobian}
    \begin{aligned}
    -\frac{\de D_{\pi}(y_j)}{\de p_i}
    &=
    \frac{1}{n}
    \sum_{\ell=1}^n
    \frac{\exp\lt(-\frac{(x_{\ell}-y_j)^2+(x_{\ell}-y_i)^2}{2}\rt)}{P_{\pi}(x_{\ell})^2};
    \\
    -\frac{\de D_{\pi}'(y_j)}{\de p_i}
    &=
    \frac{1}{n}
    \sum_{\ell=1}^n
    \frac{(x_{\ell}-y_j)\exp\lt(-\frac{(x_{\ell}-y_j)^2+(x_{\ell}-y_i)^2}{2}\rt)}{P_{\pi}(x_{\ell})^2};
    \\
    -\frac{\de D_{\pi}(y_j)}{\de y_i}
    &=
    \frac{1}{n}
    \sum_{\ell=1}^n
    \frac{(x_{\ell}-y_i)\exp\lt(-\frac{(x_{\ell}-y_j)^2+(x_{\ell}-y_i)^2}{2}\rt)}{P_{\pi}(x_{\ell})^2}
    \\
    &\quad -
    \underbrace{\frac{1}{n}
    \sum_{\ell=1}^n
    \frac{1_{i=j}(x_{\ell}-y_i)\exp\lt(-\frac{(x_{\ell}-y_i)^2}{2}\rt)}{P_{\pi}(x_{\ell})}
    }_{=D_{\pi}'(y_i)=0}
    ;
    \\
    -\frac{\de D_{\pi}'(y_j)}{\de y_i}
    &=
    \frac{1}{n}
    \sum_{\ell=1}^n
    \frac{(x_{\ell}-y_i)(x_{\ell}-y_j)\exp\lt(-\frac{(x_{\ell}-y_j)^2+(x_{\ell}-y_i)^2}{2}\rt)}{P_{\pi}(x_{\ell})^2}
    \\
    &\quad -
    ~
    \underbrace{\frac{1}{n}
    \sum_{\ell=1}^n
    \frac{1_{i=j}\big((x_{\ell}-y_i)^2-1\big)\exp\lt(-\frac{(x_{\ell}-y_i)^2}{2}\rt)}{P_{\pi}(x_{\ell})}
    }_{=D_{\pi}''(y_i)<0}
    .
    \end{aligned}
    \end{equation}
    Note the $1_{i=j}$ terms come because if $i=j$, we vary the input to $D_{\pi}$.

    We will show $-J_{2k}$ is strictly positive definite, hence invertible.
    We first show the $k\times k$ submatrix $-J_k$ formed by the $\frac{\de D_{\pi}(y_j)}{\de p_i}$ terms above is strictly positive definite.
    Indeed, we see from the above that 
    \[
    \frac{\de D_{\pi}(y_j)}{\de p_i}
    =
    A_i A_j
    \sum_{\ell=1}^n
    C_{\ell} e^{x_{\ell}(y_i+y_j)}
    \]
    for strictly positive constants $A_i$ and $C_{\ell}$.
    Thus $-J_k$ is positive semi-definite, represented as a positive combination of rank $1$ matrices $(\vec v^{(\ell)})^{\otimes 2}$ for $\vec v^{(\ell)}_i=e^{x_{\ell}y_i}$.
    We claim $-J_k$ is \emph{strictly} positive-definite.
    Indeed, Lemma~\ref{lem:nondegen-vandermonde} shows the vectors $\vec v^{(1)},\dots,\vec v^{(n)}$ are linearly independent, i.e. that if the real coefficients $B_1,\dots,B_k$ satisfy
    \[
    G(x)\triangleq 
    \sum_{j=1}^k
    B_j e^{xy_j}=0
    \quad 
    \forall x\in \{x_1,\dots,x_n\}
    \]
    then $B_1=\dots=B_k=0$.

    Next, using the simplifications indicated by identifying values of $D_{\pi}'$ and $D_{\pi}''$ above, we see that $-J_{2k}$ is positive semi-definite via another explicit representation as a sum of $n$ rank $1$ matrices, together with the diagonal terms from $D_{\pi}''$ appearing in the formula for $\frac{\de D_{\pi}'(y_j)}{\de y_i}$.
    Since all of these diagonal terms are non-zero, it easily follows that all of $-J_{2k}$ is strictly positive definite.
    Indeed, we can write $-J_{2k}=M+M'$ where $M'$ includes only the diagonal $D_{\pi}''$ terms.
    Then if $\la v,(M+M')v\ra=0$ for some $v\in \mathbb R^{2k}$, we see that $v$ cannot have any non-zero entry interacting with $M'$ since $M$ is already positive semi-definite.
    But then we are reduced to the strict positive definiteness of $-J_k$ shown above!
\end{proof}

\begin{lemma}
\label{lem:symmetric-dependence}
    Suppose $\wh\pi$ is the NPMLE for $X$ and satisfies $\min\limits_{y\in\supp(\wh\pi)}D_{\wh\pi,X}''(y)<0$.
    Suppose further that $x_1=\dots=x_m$.
    For $\eps>0$ and $Z=(z_1,\dots,z_m)\in [0,1]^m$, let $\wh\pi_Z=\wh\pi(x_1+\eps z_1,\dots,x_m+\eps z_m,x_{m+1},\dots,x_n)$. 
    For $\eps$ sufficiently small, $\wh\pi_Z$ lies in $\Pi_k$ and has a Taylor expansion of the form: 
    \begin{align*}
    \wh\pi_Z
    =
    \wh\pi
    +
    \eps f_1(P_1(Z))
    +
    \eps^2 f_2(P_1(Z),P_2(Z))
    +
    \dots
    +
    \eps^m f_k(P_1(Z),\dots,P_m(Z))
    +
    O(\eps^{m+1})
    \end{align*}
    where $P_j(Z)=\sum_{i=1}^m z_i^j$ is a power sum symmetric polynomial and each function $f_i$ is Lipschitz.
    Here we consider $\Pi_k$ to be an open subset of a $2k-1$ dimensional real vector space.
\end{lemma}

\begin{proof}
    Lemma~\ref{lem:smooth-dependence-general} implies that $\wh\pi_Z\in \Pi_k$ depends smoothly on $Z$.
    Its $k$-th Taylor coefficient is a degree $k$ polynomial in $Z$.
    Since $x_1=\dots=x_m$, it must be symmetric in $(z_1,\dots,z_m)$, hence is a polynomial in $(P_1(Z),\dots,P_k(Z))$ (see e.g. \cite[Chapter 7]{EC2}). 
    This easily gives the form of the Taylor expansion above.
\end{proof}

\begin{proof}[Proof of Theorem~\ref{thm:absolute-continuity-false}]
    We begin with $(X,\wh\pi)$ as in Lemma~\ref{lem:symmetric-dependence} such that $\wh\pi$ has $k$ atoms.
    This can be arranged by taking $x_1=\dots=x_m$ to be fixed and the remaining data generic.
    Then Section~\ref{app:positive-prob-k-atoms} shows $|\supp(\wh\pi)|=k$ with positive probability, while since $n-m>(2k+2)^2$, Theorem~\ref{thm:main} implies the remaining hypothesis holds almost surely.

    Next, take $\eps>0$ small and consider the set of achievable $\wh\pi_Z$ as above.
    Lemma~\ref{lem:symmetric-dependence} implies this set of $\wh\pi_Z$ is contained in the $O(\eps^{2k})$ neighborhood of the image of $[0,\eps]\times [0,\eps^2]\times \dots\times [0,\eps^{2k-1}]$ under an $O(1)$-Lipschitz function (i.e. with Lipschitz constant independent of $\eps$).
    This is because the Taylor expansion above is $O(\eps^j)$-Lipschitz in $p_j(Z)$, which is bounded by a constant that does not depend on $\eps$ (thus rescaling gives the claim).

    This set of possible $\wh\pi_Z$ can be covered by at most $O(\eps^{-k(2k-1)})$ balls of radius $\eps^{2k-1}$, hence has volume $O(\eps^{(2k-1)^2-k(2k-1)})=O(\eps^{(k-1)(2k-1)})$.
    On the other hand, the volume of the corresponding set of inputs $(x_1+\eps z_1,\dots,x_m+\eps z_m)$ is $\eps^m$. 
    Hence in order for $\wh\pi$ to have locally bounded density, we must have $m\leq (k-1)(2k-1)$.
    This completes the proof.
\end{proof}

\section{Asymptotic Convergence From Approximate Solutions}
\label{app:shub-smale}

Here we explain the rapid asymptotic convergence of Newton--Raphson and Expectation-Maximization starting from a good approximate solution.

\subsection{Approximate Shub--Smale Solutions}

Given our existing work in proving Theorem~\ref{thm:certify-intro}, the Shub--Smale approximate solution property follows from standard results in numerical analysis.
The notation in the following result is adapted to our setting.

\begin{proposition}[{\cite[Theorem 5.3.2]{stoer2013introduction}}]
\label{prop:newton-method-convergence-rate}
    Let $B_r(\pi)\subseteq\Pi_k$ be the $r$-neighborhood of $\pi\in\Pi_k\simeq \mathbb R^{2k-1}$. 
    Let $\bar B_r(\pi)$ be the closure of $B_r(\pi)$ in $\mathbb R^{2k-1}$ and assume $\bar B_r(\pi)\subseteq \Pi_k$ (i.e.\ the distance from $\pi$ to the boundary $\partial\Pi_k$ is larger than $r$).
    Let $F:\bar B_r(\pi)\to \mathbb R^{2k-1}$ be a smooth function.
    Suppose the positive constants $r,\alpha,\beta,C$ satisfy:
    \begin{align*}
        h&\triangleq \alpha\beta C<1;
        \\
        r&= \frac{\alpha}{1-h}.
    \end{align*}
    Suppose $\ell_X$ and its gradient and Hessian (in the space $\Pi_k$) satisfies:
    \begin{enumerate}
        \item $\nabla^2 \ell_X$ is $C$-Lipschitz on $B_r(\pi)$.
        \item $\nabla^2\ell_X\preceq -I_{2k-1}/\beta$ uniformly on $B_r(\pi)$.
        \item $\|[\nabla^2\ell_X(\pi)]^{-1} \nabla \ell_X(\pi)\|\leq\alpha$.
    \end{enumerate}
    Then Newton--Raphson iteration starting from $\pi_0=\pi$ converges to a limit $\pi_{\infty}\in\Pi_k$ such that for each $t\geq 0$:
    \begin{equation}
        \label{eq:Newton-quantitative-general}
        d_{\Pi_k}(\pi_t,\pi_{\infty})\leq \alpha \frac{h^{2^t -1}}{1-h^{2^t}}.
    \end{equation}
\end{proposition}

\begin{proof}[Proof of Shub--Smale property in Theorems~\ref{thm:certify-intro} and \ref{thm:certify-intro-approx}]
We show $\mathring\pi_{\eps}$ is a Shub--Smale approximate NPMLE for small enough $\eps$ (the proof for the adjacent-atom rounding of $\wh\pi_{\eps}$ is identical).
From the proof of the previous parts of Theorem~\ref{thm:certify-intro}, we know that $d_{\Pi_k}(\mathring\pi_{\eps},\wh\pi)\leq \eps^{1/4}$ for small enough $\eps$.
We have shown via Lemma~\ref{lem:smooth-dependence-general} that Theorem~\ref{thm:landscape-consequences}\ref{it:local-strong-convexity} holds, so $\ell_X$ is $c$-strongly concave in an open $\Pi_k$-neighborhood of $\wh\pi$, for some $c>0$.
Additionally, this Hessian is clearly $C$-Lipschitz for some $C=C(L)$ thanks to e.g. Lemma~\ref{lem:lipschitz-bounds}.
This means local $c/2$-strong concavity of $\ell_X$ is certifiable in a $\Pi_k$-neighborhood of $\mathring \pi_{\eps}$ which certifiably contains $\wh\pi$, for $\eps$ small enough.

We will apply Proposition~\ref{prop:newton-method-convergence-rate} to such $\pi=\mathring\pi_{\eps}$ and conclude the desired result.
Indeed we can take $C=C(L)$ as mentioned just above, and $\beta=2/c$.
Neither depends on $\eps$, and so we can take $\alpha\leq \eps^{1/5}$ for small enough $\eps$.
Then $h\leq 1/10<1$ and $r\leq 2\alpha\leq 2\eps^{1/5}$ is smaller than the $\Pi_k$-distance from $\mathring\pi_{\eps}$ to $\partial\Pi_k$, for small $\eps$.
Then \eqref{eq:Newton-quantitative-general} implies quadratic convergence to a limit $\pi_{\infty}\in\Pi_k$ with $\nabla\ell_X(\pi_{\infty})=\vec 0$ and $d_{\Pi_k}(\pi_0,\pi_{\infty})\leq 2\alpha$.
From the local strong concavity, it follows that $\pi_{\infty}=\wh\pi$.
The precise quantitative rate in Definition~\ref{def:shub-smale-approx-npmle} holds because $h\leq 1/10$.
This concludes the proof.
\end{proof}

\subsection{Static Support Case}

Here we explain Theorems~\ref{thm:static-landscape-consequences} and \ref{thm:certify-finite-S}, which are the static support analogs of Theorems~\ref{thm:landscape-consequences} and \ref{thm:certify-intro}, \ref{thm:certify-intro-approx}.

\begin{proof}[Proof of Theorem~\ref{thm:static-landscape-consequences}]
    Here the Hessian of $\ell_X$ is just the submatrix $J_k$ from the proof of Lemma~\ref{lem:smooth-dependence-general}, which is strictly negative definite for all $1\leq k\leq n$ (with no genericity conditions on $D$).
    This implies both results.
\end{proof}

\begin{proof}[Proof of Shub--Smale property in  Theorem~\ref{thm:certify-finite-S}]
    Given the local strong concavity around $\wh\pi_S$ from Theorem~\ref{thm:static-landscape-consequences}, the proof is identical to the corresponding parts of Theorems~\ref{thm:certify-intro} and \ref{thm:certify-intro-approx}.
\end{proof}

\subsection{Locally Linear Convergence of Expectation--Maximization}
\label{subsec:EM}

The Expectation--Maximization (EM) algorithm of \cite{dempster1977maximum} is another approach to optimize $\ell_X$.
Here one defines a sequence of iterates $\pi_0,\pi_1,\dots\in\Pi_k$ as follows.
Given an iterate $\pi_t=\sum_{i=1}^k p_i\delta_{y_i}$, the next iterate $\pi_{t+1}$ has parameters:
\begin{align}
\label{eq:EM-update-p}
\tilde p_i
&=
\frac{p_i}{n}
\sum_{j=1}^n
\frac{e^{-(x_j-y_i)^2/2}}{
\sum_{\ell=1}^k p_{\ell}e^{-(x_j-y_{\ell})^2/2}.
},
\\
\label{eq:EM-update-y}
\tilde y_i
&=
\frac{
\sum_{j=1}^n 
\frac{x_j e^{-(x_j-y_i)^2/2}}{\sum_{\ell=1}^k p_{\ell}e^{-(x_j-y_{\ell})^2/2}}
}
{\sum_{j=1}^n 
\frac{e^{-(x_j-y_i)^2/2}}{\sum_{\ell=1}^k p_{\ell}e^{-(x_j-y_{\ell})^2/2}}
}
=
y_i
+
\frac{
\sum_{j=1}^n 
\frac{(x_j-y_i) e^{-(x_j-y_i)^2/2}}{\sum_{\ell=1}^k p_{\ell}e^{-(x_j-y_{\ell})^2/2}}
}
{\sum_{j=1}^n 
\frac{e^{-(x_j-y_i)^2/2}}{\sum_{\ell=1}^k p_{\ell}e^{-(x_j-y_{\ell})^2/2}}
}
=
y_i
+
\frac{D'(y_i)}{D(y_i)}
\,.
\end{align}
This iteration improves the value of $\ell_X$ in each step; however it is not guaranteed to converge to a global optimum of $\ell_X$ even within the space $\Pi_k$.
On the other hand, one can verify using \eqref{eq:stationarity} that the NPMLE $\wh\pi$ is a fixed point of the iteration.

\begin{proof}[Proof of Theorem~\ref{thm:landscape-consequences}\ref{it:EM-algorithm}]
Let $F(\pi)=\tilde \pi$ correspond to the iteration from \eqref{eq:EM-update-p}, \eqref{eq:EM-update-y}.
Identifying $\Pi_k$ with an open subset of $\bbR^{2k}$ (i.e. not enforcing the constraint $\sum_{i=1}^k p_i=1$), we will show that $JF(\wh\pi)$ is strictly stable, with all eigenvalues in $(-1,1)$.
This immediately implies the same property on the $2k-1$ dimensional space $\Pi_k$ (since the output of \eqref{eq:EM-update-p} immediately satisfies $\sum_i p_i=1$).
It is well-known that this stability implies locally linear convergence rate for the EM algorithm; see e.g. \cite{dempster1977maximum,meng1994global}.

We first explicitly compute the Jacobian $JF(\wh\pi)$ at $\wh\pi$:
\begin{align*}
\frac{\partial \tilde p_i}{\partial p_r}
&=
1_{i=r}
-
\frac{p_i}{n}
\sum_{j=1}^n 
\frac{e^{-(x_j-y_i)^2/2 - (x_j-y_r)^2/2}}{\big(\sum_{\ell=1}^k p_{\ell} e^{-(x_j-y_{\ell})^2/2}\big)^2};
\\
\frac{\partial \tilde p_i}{\partial y_r}
&=
\frac{p_i}{n}
\underbrace{
\sum_{j=1}^n 
\frac{(x_j-y_r)e^{-(x_j-y_r)^2/2}}{\sum_{\ell=1}^k p_{\ell} e^{-(x_j-y_{\ell})^2/2}}
}_{=D'(y_r)=0}
-
\frac{p_i}{n}
\sum_{j=1}^n
\frac{p_r(x_j-y_r)e^{-(x_j-y_r)^2/2-(x_j-y_i)^2/2}}{\big(\sum_{\ell=1}^k p_{\ell} e^{-(x_j-y_{\ell})^2/2}\big)^2}
;
\\
\frac{\partial \tilde y_i}{\partial p_r}
&=
-\frac{1}{n}
\sum_{j=1}^n 
\frac{(x_j-y_i)e^{-(x_j-y_r)^2/2-(x_j-y_i)^2/2}}{\big(\sum_{\ell=1}^k p_{\ell} e^{-(x_j-y_{\ell})^2/2}\big)^2},
\\
\frac{\partial \tilde y_i}{\partial y_r}
&=
1_{i=r}\cdot (1+D''(y_i))
-
\frac{1}{n}
\sum_{j=1}^n 
\frac{p_r(x_j-y_r)(x_j-y_i)(e^{-(x_j-y_r)^2/2-(x_j-y_i)^2/2}}{\big(\sum_{\ell=1}^k p_{\ell} e^{-(x_j-y_{\ell})^2/2}\big)^2}.
\end{align*}
Here the factors of $1/n$ in the latter two lines come from $D(y_i)=1$ for all $i$.
We note that another term which equals $0$ for $\wh\pi$ was omitted from the formula for $\frac{\partial \tilde y_i}{\partial p_r}$, again because $D'(y_r)=0$ holds at $\wh\pi$.

The terms $1_{i=r}$ cancel $I_{2k}$.
Removing the first term in $\frac{\partial \tilde p_i}{\partial y_r}$ as indicated above, we may write:
\begin{align*}
&I_{2k}-JF+\diag(0,\dots,0, D''(y_1),\dots,D''(y_k))
\\
&=
\begin{bmatrix}
p_i/n~ & p_i p_r/n
\\
1/n~ & p_r/n
\end{bmatrix}
\odot
\sum_{j=1}^n
\Big(\sum_{\ell=1}^k p_{\ell} e^{-(x_j-y_{\ell})^2/2}\Big)^{-2}
\begin{bmatrix}
e^{-(x_j-y_i)^2/2}
\\
(x_j-y_i)e^{-(x_j-y_i)^2/2}
\end{bmatrix}^{\otimes 2}.
\end{align*}
Here the former matrix is $2k\times 2k$, with $(i,r)$ entry equal to $p_i/n$, $(i,k+r)$ entry equal to $\frac{p_ip_r}{n}$, etc.
The latter column vector has length $2k$ (and depends on $j$). 
The first $k$ entries are given by $\frac{p_i}{n} e^{-(x_j-y_i)^2/2}$ for $1\leq i\leq k$, and the next $k$ entries are given by $p_i (x_j-y_i)e^{-(x_j-y_i)^2/2}$ for $1\leq i\leq k$.
Note that the matrix 
$\begin{bmatrix}
p_i/n~ & p_i p_r/n
\\
1/n~ & p_r/n
\end{bmatrix}$ is similar to $\begin{bmatrix}
p_i/n & \frac{\sqrt{p_i p_r}}{n}
\\
\frac{\sqrt{p_i p_r}}{n} & p_r/n
\end{bmatrix}$
via conjugation by $\mathsf{M}=\diag(\sqrt{p_1},\dots,\sqrt{p_k},1,\dots,1)$; in particular their eigenvalues are equivalent. This conjugation extends to the preceding display, and shows
\[
I_{2k}-JF+\diag(0,\dots,0, D''(y_1),\dots,D''(y_k))\triangleq I_{2k}-JF+\mathsf{D}
\]
is similar to the positive semi-definite matrix $\mathsf{M}^{-1} (I_{2k}-JF+\mathsf{D})\mathsf{M}$ given explicitly by:
\begin{equation}
\label{eq:PSD-similar-matrix}
\sum_{j=1}^n
\Big(\sum_{\ell=1}^k p_{\ell} e^{-(x_j-y_{\ell})^2/2}\Big)^{-2}
\begin{bmatrix}
\sqrt{\frac{p_i}{n}} e^{-(x_j-y_i)^2/2}
\\
\sqrt{\frac{p_i}{n}} (x_j-y_i)e^{-(x_j-y_i)^2/2}
\end{bmatrix}^{\otimes 2}.
\end{equation}
We claim that $\mathsf{M}^{-1} (I_{2k}-JF)\mathsf{M}$ is strictly positive semi-definite.
Since $D''(y_i)<0$ for each $i$, it suffices to verify that the upper-left $k\times k$ block of \eqref{eq:PSD-similar-matrix} is strictly positive-definite.
This amounts to proving that the $n$ vectors $(e^{-(x_j-y_j)^2/2})_{i=1}^k$ (for $1\leq j\leq n$) are linearly independent.
As usual, this is equivalent to linear independence of the vectors $(e^{x_j y_i})_{i=1}^k$ which holds for any $k\leq n$ by Lemma~\ref{lem:nondegen-vandermonde}.

Next we need to show all eigenvalues of $JF$ are larger than $-1$, or equivalently that all eigenvalues of $\mathsf{M}^{-1}(I_{2k}-JF)\mathsf{M}$ are strictly smaller than $2$. 
For each $j$, we upper-bound the rank $1$ term in \eqref{eq:PSD-similar-matrix} by replacing $(p_i,p_r)$ entries of the form $A_{i,r,j}A_{r,i,j}$ by diagonal entries $A_{i,r,j}^2$ and $A_{r,i,j}^2$ in the $(p_i,p_i)$ and $(p_r,p_r)$ positions respectively, which always gives an upper bound in the positive semi-definite order.
In particular, the $(p_i,p_r)$ and $(p_r,p_i)$ diagonal entries 
\[
\frac{
\frac{\sqrt{p_i p_r}}{n}
\cdot e^{-(x_j-y_i)^2/2 - (x_j-y_r)^2/2}
}{\Big(\sum_{\ell=1}^k p_{\ell} e^{-(x_j-y_{\ell})^2/2}\Big)^{2}}
\]
can be upper-bounded via replacement by the respective $(p_i,p_i)$ and $(p_r,p_r)$ entries:
\begin{equation}
\label{eq:pi-pr-replacement-AMGM}
\frac{
\frac{p_r}{n}
\cdot e^{-(x_j-y_i)^2/2 - (x_j-y_r)^2/2}
}{\Big(\sum_{\ell=1}^k p_{\ell} e^{-(x_j-y_{\ell})^2/2}\Big)^{2}},\quad 
\frac{
\frac{p_i}{n}
\cdot e^{-(x_j-y_i)^2/2 - (x_j-y_r)^2/2}
}{\Big(\sum_{\ell=1}^k p_{\ell} e^{-(x_j-y_{\ell})^2/2}\Big)^{2}}.
\end{equation}
(Note that the $p_i$ factor appears in the $(p_r,p_r)$ position, and vice-versa.)
Similarly the $(p_i,y_r)$ and $(y_r,p_i)$ entries
\[
\frac{
\frac{\sqrt{p_i p_r}}{n}
\cdot (x_j-y_r)e^{-(x_j-y_i)^2/2 - (x_j-y_r)^2/2}
}{\Big(\sum_{\ell=1}^k p_{\ell} e^{-(x_j-y_{\ell})^2/2}\Big)^{2}}
\]
can be upper-bounded by the respective $(p_i,p_i)$ and $(y_r,y_r)$ entries
\[
\frac{
\frac{p_r}{n}
\cdot e^{-(x_j-y_i)^2/2 - (x_j-y_r)^2/2}
}{\Big(\sum_{\ell=1}^k p_{\ell} e^{-(x_j-y_{\ell})^2/2}\Big)^{2}},\quad 
\frac{
\frac{p_i}{n}
\cdot (x_j-y_r)^2e^{-(x_j-y_i)^2/2 - (x_j-y_r)^2/2}
}{\Big(\sum_{\ell=1}^k p_{\ell} e^{-(x_j-y_{\ell})^2/2}\Big)^{2}},
\]
and the $(y_i,y_r)$ and $(y_r,y_i)$ entries
\[
\frac{
\frac{\sqrt{p_i p_r}}{n}
\cdot (x_j-y_i)(x_j-y_r)e^{-(x_j-y_i)^2/2 - (x_j-y_r)^2/2}
}{\Big(\sum_{\ell=1}^k p_{\ell} e^{-(x_j-y_{\ell})^2/2}\Big)^{2}}
\]
can be upper-bounded by the respective $(y_i,y_i)$ and $(y_r,y_r)$ entries
\[
\frac{
\frac{p_r}{n}
\cdot e^{-(x_j-y_i)^2/2 - (x_j-y_r)^2/2}
}{\Big(\sum_{\ell=1}^k p_{\ell} e^{-(x_j-y_{\ell})^2/2}\Big)^{2}},\quad 
\frac{
\frac{p_i}{n}
\cdot (x_j-y_r)^2e^{-(x_j-y_i)^2/2 - (x_j-y_r)^2/2}
}{\Big(\sum_{\ell=1}^k p_{\ell} e^{-(x_j-y_{\ell})^2/2}\Big)^{2}}.
\]
After making all of these substitutions, the resulting diagonal matrix has $(p_i,p_i)$ and $(y_i,y_i)$ entries:
\[
\frac{
2e^{-(x_j-y_i)^2/2}
}{n\sum_{\ell=1}^k p_{\ell} e^{-(x_j-y_{\ell})^2/2}},
\quad 
\frac{
2(x_j-y_i)^2 e^{-(x_j-y_i)^2/2}
}{n\sum_{\ell=1}^k p_{\ell} e^{-(x_j-y_{\ell})^2/2}}
.
\]
Summing over $j$, the resulting matrix has $(p_i,p_i)$ entry $2D(x_j)=2$ and $(y_i,y_i)$ entry $2+2D''(y_i)$.
Since this matrix is an upper bound for $\mathsf{M}^{-1} (I_{2k}-JF+\mathsf{D})\mathsf{M}$ in the positive semidefinite order, we conclude that 
\begin{equation}
\label{eq:nonstrict-domination-EM-proof}
    \mathsf{M}^{-1} (I_{2k}-JF)\mathsf{M}
    \preceq
    \diag(2,\dots,2,2+D''(y_1),\dots,2+D''(y_k)).
\end{equation}
This shows the desired inequality, again except for strictness in the $(p_1,\dots,p_k)$ subspace. 
Note that for $v=(v_1,\dots,v_{2k})\neq \vec 0$ to satisfy
\[
\la \mathsf{M}^{-1} (I_{2k}-JF)\mathsf{M},v^{\otimes 2}\ra
=
\big\la \diag\big(2,\dots,2,2+D''(y_1),\dots,2+D''(y_k)\big),v^{\otimes 2}\big\ra,
\]
we must have $v_{k+1}=\dots=v_{2k}=0$, since $D''(y_i)<0$ for each $i$.
However then the contributions from $(p_i,y_r)$ entries to the left-hand expression are zero.
Hence for such $v$ the only contribution comes from \eqref{eq:pi-pr-replacement-AMGM} so we actually have
\[
\la \mathsf{M}^{-1} (I_{2k}-JF)\mathsf{M},v^{\otimes 2}\ra
\leq 
\|v\|^2.
\]
Combining, we see that \eqref{eq:nonstrict-domination-EM-proof} is strict, i.e. the difference is strictly positive definite. This completes the proof.
\end{proof}

\end{appendix}

\end{document}